\documentclass[11pt]{amsart}

\usepackage{amsmath,amssymb,amsxtra,epsf,amscd,graphics,color,array,ulem,etex, colortbl}
\usepackage{mathrsfs}
\usepackage{pstricks}
\usepackage[latin1]{inputenc}
\usepackage[T1]{fontenc}
\usepackage[frenchb, english]{babel}

\usepackage{amsfonts}

\usepackage{ytableau}

\usepackage{enumerate,fancybox,array,calc}
\usepackage{epsf}
\usepackage{epsfig}
\usepackage{verbatim,ifthen, fancyvrb,xkeyval,xcolor,epic,eepic,rotating}

\usepackage{hhline}
\usepackage{tikz}
\usetikzlibrary{patterns}

\newcommand{\g}{\mathfrak{g}}
\newcommand{\ad}{\operatorname{ad}}
\newcommand{\ind}{\operatorname{ind}}

\def\ep{\varepsilon}
\def\p{\mathfrak p}
\def\s{\mathfrak s}
\def\h{\mathfrak h}

\def\a{\mathfrak a}

\def\n{\mathfrak n}
\def\ch{\rm ch}
\def\q{\mathfrak q}

\newtheorem*{thm}{Theorem}

\newtheorem*{prop}{Proposition}
\newtheorem*{lm}{Lemma}
\newtheorem*{cor}{Corollary}

\newtheorem*{Rqs}{Remarks}
\newtheorem*{Rq}{Remark}

\newcommand{\isomto}{\overset{\sim}{\rightarrow}}

\title[Weierstrass sections]{Weierstrass sections for some truncated parabolic subalgebras.}

\author{Florence Fauquant-Millet}

\address{Univ Lyon,
UJM Saint-Etienne\\
CNRS UMR 5208\\
Institut Camille Jordan\\
F- 42023 Saint-Etienne\\
France}
\email{florence.millet@univ-st-etienne.fr}

\usepackage{hyperref}
\usepackage{amsfonts,amssymb,wasysym}

\begin{document}

\begin{abstract}

In this paper, using Bourbaki's convention, we consider a simple Lie algebra $\g\subset\g{\mathfrak l}_m$ of type B, C or D and a parabolic subalgebra $\p$ of $\g$ associated with a  Levi factor composed  essentially, on each side of the second diagonal, by successive blocks of size two, except possibly for the first and the last ones. 
Extending the notion of a Weierstrass section introduced by Popov to the coadjoint action of the truncated parabolic subalgebra associated with $\p$,
we construct explicitly Weierstrass sections, which give the polynomiality (when it was not yet known) for the algebra generated by semi-invariant polynomial  functions on the dual space $\p^*$ of $\p$ and which allow to linearize  semi-invariant generators.
Our Weierstrass sections  require the construction of an adapted pair, which is the analogue of a principal $\s\mathfrak l_2$-triple in the non reductive case.

\end{abstract}

\maketitle



\section{Introduction.}\label{Intro}

The base field $\Bbbk$ is algebraically closed of characteristic zero.

\subsection{}\label{Introbounds}

Let $\g$ be a simple Lie algebra over $\Bbbk$ and $\p$ be a standard parabolic subalgebra of $\g$, acting by coadjoint action on its dual space $\p^*$. Denote by $Sy(\p)$ the vector space generated by the semi-invariant polynomial functions on  $\p^*$. This is a subalgebra of the symmetric algebra $S(\p)$ of $\p$. Moreover there exists a canonically defined subalgebra $\p_{\Lambda}$ of $\p$, called the canonical truncation of $\p$ or the truncated parabolic subalgebra associated with $\p$, such that the algebra $Y(\p_{\Lambda})$ of  invariant polynomial functions on $\p_{\Lambda}^*$ coincides with the algebra $Sy(\p)$ (see \ref{CT} for more details). For some parabolic subalgebras $\p$ which we define below, we will study whether $Sy(\p)$ is isomorphic to a polynomial algebra over $\Bbbk$ and whether one can linearize generators of $Sy(\p)$.

\subsection{}\label{Cases}

Now we consider $\g$ simple of type ${\rm B}_n$, ${\rm C}_n$ or ${\rm D}_n$ and integers $\ell\in\mathbb N$ and $s\in\mathbb N^*$ with $s+2\ell\le n$. 

Using Bourbaki's labelling \cite{BOU} for a chosen set  $\pi=\{\alpha_1,\,\ldots,\,\alpha_n\}$ of simple roots of $\g$  with respect to some Cartan subalgebra $\h$, we focus on several standard parabolic subalgebras $\p$ of $\g$ associated with a particular subset $\pi'$ of  $\pi$, where roughly speaking every second root  in a chain of simple roots is deleted. 

 More precisely we consider the parabolic subalgebra $\p_{s,\,\ell}$ of $\g$ associated with the subset $\pi'\subset\pi$ such that $$\pi'=\pi\setminus\{\alpha_s,\,\alpha_{s+2},\,\ldots,\,\alpha_{s+2\ell}\}$$ with $s+2\ell\le n$.

When $\g$ is of type ${\rm D}_n$, we also study some parabolic subalgebras associated with a subset $\pi'\subset\pi$ which does not contain the last two roots $\alpha_{n-1}$ and $\alpha_n$ and also does not contain every second root  in a chain of simple roots.
Indeed we consider two other cases of parabolic subalgebras in $\g$ of type ${\rm D}_n$ which we define below.

The first case consists in deleting $\alpha_n$, $\alpha_{n-1}$ and then possibly every second simple root preceding $\alpha_{n-1}$ until $\alpha_{n-1-2\ell}$ with $0\le \ell\le(n-2)/2$. 
Thus  we denote by $\p_{\ell}$ the parabolic subalgebra  of $\g$ of type ${\rm D}_n$ associated with the subset $\pi'\subset\pi$ such that $$\pi'=\pi\setminus\{\alpha_{n-1-2k},\,\alpha_n\mid 0\le k\le \ell\}$$
 with $0\le \ell\le(n-2)/2$.
 
The second case consists in deleting $\alpha_n$, $\alpha_{n-1}$ and then every second simple root from some simple root $\alpha_s$ until $\alpha_{s+2\ell}$ with  $s+2\ell\le n-2$. 
Thus we denote by $\q_{s,\,\ell}$ the parabolic subalgebra of $\g$ of type ${\rm D}_n$ associated with 
$$\pi'=\pi\setminus\{\alpha_s,\,\alpha_{s+2},\,\ldots,\,\alpha_{s+2\ell},\,\alpha_{n-1},\,\alpha_n\}$$
with  $s+2\ell\le n-4$ or $s+2\ell=n-2$.
Note that, if $s+2\ell=n-3$, then $\q_{s,\,(n-3-s)/2}=\p_{(n-1-s)/2}$ or simplier $\q_{n-3-2\ell,\,\ell}=\p_{\ell+1}$.

Roughly speaking, identifying $\g$ with a Lie subalgebra of some $\g\mathfrak l_m$ and adopting the conventions in \cite[Chap VIII]{Bou1} the Levi factor of every parabolic subalgebra $\p$ as defined above is composed, on each side of the second diagonal, by $\ell$ successive blocks of size two, a first block and possibly a last block (in type ${\rm D}_n$, when $\alpha_n\in\pi'$ but $\alpha_{n-1}\not\in\pi'$, we may notice that we have a pair of blocks along the second diagonal, symmetric with respect to the first diagonal). In other words the Levi subalgebra of such $\p$ is of type ${\rm A}_{s-1}\times{\rm A}_1^{\ell}\times{\rm R}_{r}$, where
$$\begin{cases}
s=n-1-2\ell \;{\rm and}\; {\rm R}_{r}=\{0\}&{\rm for}\; \p_{\ell}\\
{\rm R}_r={\rm A}_{n-2-s-2\ell}&{\rm for}\; \q_{s,\,\ell} \\
{\rm R}_r={\rm B}_{n-s-2\ell}&{\rm for}\; \p_{s,\,\ell} \;{\rm and}\;\g\; {\rm of\; type}\;{\rm B}_n\\
{\rm R}_r={\rm D}_{n-s-2\ell}&{\rm for}\; \p_{s,\,\ell} \;{\rm and}\;\g \;{\rm of\; type}\; {\rm D}_n\\
\end{cases}$$
with the convention that ${\rm A}_0={\rm B}_0={\rm D}_0={\rm A}_1^{0}=\{0\}$, ${\rm B}_1={\rm D}_1={\rm A}_1$ and ${\rm D}_2={\rm A}_1\times{\rm A}_1$ (here, for any $k\in\mathbb N^*$, type $\{0\}\times {\rm A}_k$ or ${\rm A}_k\times\{0\}$, resp. $\{0\}\times{\rm B}_k$, resp. $\{0\}\times{\rm D}_k$, simply means type ${\rm A}_k$, resp. ${\rm B}_k$, resp. ${\rm D}_k$).

Note that the parabolic subalgebra  $\p_{s,\,0}$  is a maximal parabolic subalgebra and it has already been treated in \cite{F1},\, \cite{FL} and \cite{FL1}. Thus we will not consider this case.
This work is a continuation and a generalization of \cite{F1},\, \cite{FL} and \cite{FL1}.

\subsection{}

Let $X$ be a finite dimensional vector space on which a reductive Lie algebra $\a$ acts linearly. Denote by $S(X^*)$ the symmetric algebra of the dual space $X^*$ of $X$, which may be identified with the algebra of polynomial functions $\Bbbk[X]$ on $X$. Let $S(X^*)^{\a}$ denote the algebra of invariants in $S(X^*)$ under the action of $\a$ (induced by the action of $\a$ on $X$), which is also the algebra of invariant polynomial functions on $X$. By a Hilbert's theorem (see \cite[II, Thm. 3.5]{PoVin} for an exposition), the algebra of invariants $S(X^*)^{\a}$ is finitely generated and Popov considered in \cite[2.2.1]{Po}
the problem of linearizing  invariant generators in $S(X^*)^{\a}$  by introducing the so-called Weierstrass sections for the action of $\a$ on $X$.
Now assume that $\a$ is a finite dimensional Lie algebra, not necessarily reductive. We may extend Popov's notion for $X=\a^*$ the dual space of  $\a$, on which $\a$ acts by coadjoint action, and define a Weierstrass section for coadjoint action of $\a$ as an affine subspace $\mathscr S$ of $\a^*$ such that restriction of functions  to $\mathscr S$ induces an algebra isomorphism between
 the algebra of symmetric invariants $Y(\a)=S(\a)^{\a}$   and the algebra of polynomial functions $\Bbbk[\mathscr S]$ on $\mathscr S$. Then the existence of a Weierstrass section for coadjoint action of $\a$ implies the polynomiality of $Y(\a)$, and the restriction map gives a linearization of invariant generators of $Y(\a)$. More details on Weierstrass sections are given in \ref{WS}.
 
 In the semisimple case (that is, when $\a=\g$ a semisimple Lie algebra, see \ref{RC}), a Weierstrass section $\mathscr S$ was constructed  by Kostant  in \cite{K} using a principal $\s\mathfrak l_2$-triple. This particular Weierstrass section is called the Kostant slice, or Kostant section in \cite{Po}. The Kostant slice is also an affine slice in the sense that, if  $G$ is the adjoint group of $\g$, then $G.\mathscr S$  is dense in $\g^*$ and every coadjoint orbit in $\g^*$ meets $\mathscr S$ in at most one point, and transversally.  In \ref{AS} are more details on affine slices.

In this article, our aim is to
construct Weierstrass sections for coadjoint action of the canonical truncation $\p_{\Lambda}$ of the standard parabolic subalgebra $\p$ whenever $\p$ is either equal  to $\p_{s,\ell}$ or $\p_{\ell}$ or $\q_{s,\,\ell}$ defined in the previous subsection.

 \subsection{}\label{CUT} Similarly to the Kostant slice, a Weierstrass section for coadjoint action of $\p_{\Lambda}$ is also an affine slice to the coadjoint action of $\p_{\Lambda}$ by \cite{FJ4}. In particular if there exists a Weierstrass section $\mathscr S\subset\p_{\Lambda}^*$ for coadjoint action of $\p_{\Lambda}$, then every coadjoint orbit in $\p_{\Lambda}^*$ meets $\mathscr S$ in at most one point.
 
 \subsection{} Unlike the reductive case where  a principal $s\mathfrak l_2$-triple exists, a Weierstrass section in the non reductive case cannot be given by such a triple, since the latter does not exist.  To fill in this lack, the notion of an adapted pair  was
  introduced in~\cite{JL}.  Denote by $\h_{\Lambda}:=\h\cap\p_{\Lambda}$ the Cartan subalgebra of the truncated parabolic subalgebra $\p_{\Lambda}$ and by $ad$  the coadjoint action of $\p_{\Lambda}$ on $\p_{\Lambda}^*$. An adapted pair for $\p_{\Lambda}$  is a pair $(h,\,y)\in\h_{\Lambda}\times\p_{\Lambda}^*$  such that :
  \begin{enumerate}
 \item $ad\, h(y)=-y$  and
 \item $y$ is regular in $\p_{\Lambda}^*$ that is, there exists a subspace $V$ of $\p_{\Lambda}^*$ of minimal dimension (called the index of $\p_{\Lambda}$ and denoted by $\ind\p_{\Lambda}$) such that $ad\,\p_{\Lambda}(y)\oplus V=\p_{\Lambda}^*$.
 \end{enumerate}
More details on adapted pairs are given in \ref{ADP}. 
 Unfortunately adapted pairs do not always exist and are quite hard to construct. They may not exist even when Weierstrass sections for coadjoint action exist, as it was shown in \cite[Thm. 9.4]{J8} for the truncated Borel subalgebra $\mathfrak b_{\Lambda}$ in type ${\rm B}_{2n+1}$, ${\rm D}$, ${\rm E}$ and ${\rm G}_2$. However in \cite[11.4 Example 2]{J8}, although $Sy(\mathfrak b)=Y(\mathfrak b_{\Lambda})$ is always a polynomial algebra by \cite{J1}, it was also noticed that a Weierstrass section for coadjoint action of $\mathfrak b_{\Lambda}$ does not exist for $\g$ of type ${\rm C}_2$ since  the invariant generators cannot be linearized in this case. As  in  \cite{F1},\,\cite {FL} and \cite{FL1} we are able  in our present cases to construct Weierstrass sections thanks to adapted pairs.

\subsection{}
In  \cite{J5} Weierstrass sections  were constructed for coadjoint action of any truncated (bi)parabolic  subalgebra  in a simple Lie algebra of type A. Thus we do not consider this type.

\subsection{Main result}\label{Mainres}

Recall the notation of subsection \ref{Cases}.
In this paper we  prove that Weierstrass sections exist for the following cases :

\begin{enumerate}

\item
for coadjoint action of the canonical truncation of $\p_{s,\,\ell}$ when :

\begin{enumerate}

\item\label{cas1} $\g$ is of type ${\rm B}_n$ with $n\ge 2$, $s$ odd and $\ell\ge 1$.

\item\label{cas2} $\g$ is of type ${\rm D}_n$ with $n\ge 4$, $s$ odd and $\ell\ge 1$.

\item\label{cas3} $\g$ is of type ${\rm B}_n$ with $n\ge 4$, $s$ even and $\ell=1$.

\item\label{cas4} $\g$ is of type ${\rm D}_n$ with $n\ge 6$, $s$  even, $s\le n-4$ and $\ell=1$. 

\item\label{cas5} $\g$ is of type ${\rm C}_n$ with $n\ge 3$ and $\ell\ge 1$.

\end{enumerate}

\item for coadjoint action of the canonical truncation of $\p_{\ell}$ for $\g$ simple of type ${\rm D}_n$ when :
\begin{enumerate}
 \item\label{cas6} $n\ge 4$, and $n$ even. 
\item\label{cas7} $n\ge 5$, $n$ odd, and $\ell=0$.

\item\label{cas8} $n\ge 5$, $n$ odd, and $\ell=1$.

\end{enumerate}
\item\label{cas9} for coadjoint action of the canonical truncation of $\q_{s,\,\ell}$ 
when $\g$ is of type ${\rm D}_n$ with $n\ge 5$, $n$ odd and $s$ odd.

\end{enumerate}

\subsection{The proof}

The proof is in two steps and via a case by case consideration. Let $\p$ denote one of the above parabolic subalgebras and $\p_{\Lambda}$ its canonical truncation.

Step 1 consists of constructing explicitly an adapted pair for $\p_{\Lambda}$, thanks to Proposition \ref{propAP} which uses extensively the notion of Heisenberg sets, generalizing the sets of roots of generators in Heisenberg Lie algebras, see subsection \ref{HS}. 

Step 2 is to prove that this adapted pair gives the required Weierstrass section. For this purpose, two means are available. 
The simplest way is to check that the equality of a lower and an upper bounds for the formal character of $Sy(\p)$ (see Sect. \ref{bounds}) holds. This equality implies polynomiality of $Sy(\p)$ and then the existence of an adapted pair for $\p_{\Lambda}$ implies the existence of a Weierstrass section for coadjoint action of $\p_{\Lambda}$ (see also subsection \ref{WS}).
However in some of our cases the  lower and upper bounds mentioned above do not coincide and then the polynomiality of $Sy(\p)=Y(\p_{\Lambda})$ was not yet known. We then  check that the lower bound and a so-called improved upper bound introduced in \cite{J6bis} (see Sect. \ref{IUB}) coincide. The latter method concerns the cases \ref{cas3}, \ref{cas4}, \ref{cas7}, \ref{cas8}. The Weierstrass section we obtain in these cases assures then the polynomiality of $Sy(\p)$.

\section{Some definitions.}\label{Term}

In what follows, we specify the notions mentioned in Sect. \ref{Intro}.
Let $\a$ be an algebraic finite dimensional Lie algebra over $\Bbbk$, which acts on its symmetric algebra $S(\a)$ by the action
(denoted by $ad$) which extends by derivation the adjoint action of $\a$ on itself given by Lie bracket. We denote by $A$ the adjoint group of $\a$. 

\subsection{Algebra of symmetric invariants.}\label{SI} An invariant of $S(\a)$  (symmetric invariant of $\a$ for short) is  an element $s\in S(\a)$ such that, for all
$x\in\a$, $ad\,x(s)=0$. 

We denote by $Y(\a)=S(\a)^{\a}$ the set of symmetric invariants  of $\a$ : it is a subalgebra of $S(\a)$, called the algebra of symmetric invariants of $\a$.
We may notice that the algebra $Y(\a)$ also coincides with the centre of $S(\a)$ for its natural Poisson structure (and that is why it is sometimes also called the Poisson centre of $S(\a)$ or of $\a$ for short). Moreover $Y(\a)$ also coincides with the algebra $S(\a)^A$ of invariants of $S(\a)$ under the action of $A$ by automorphisms.

\subsection{Algebra of symmetric semi-invariants.}\label{SS}

An element $s\in S(\a)$ is called a (symmetric) semi-invariant of $\a$, if there exists $\lambda\in\a^*$ verifying that, for all $x\in\a$, $ad\,x(s)=\lambda(x) s$. We denote by $S(\a)_{\lambda}\subset S(\a)$ the space of such symmetric semi-invariants. The vector space generated by all symmetric semi-invariants of $\a$ will be denoted by $Sy(\a)$ : it is a subalgebra of $S(\a)$, called the algebra of symmetric semi-invariants of $\a$.  A linear form $\lambda\in\a^*$ such that $S(\a)_{\lambda}\neq\{0\}$  is said to be a weight of $Sy(\a)$. We denote by $\Lambda(\a)$ the set of weights of $Sy(\a)$. It is a semigroup. One has that $Sy(\a)=\bigoplus_{\lambda\in\Lambda(\a)}S(\a)_{\lambda}$.
Since $Y(\a)=S(\a)_0$, one always has that $Y(\a)\subset Sy(\a)$.

We will say that $\a$ has no proper semi-invariants when all the semi-invariants of $\a$ are invariant that is, when $Sy(\a)=Y(\a)$.

For example, when $\a=\g$ is a semisimple Lie algebra, then $\g$ has no proper semi-invariants.
Moreover if $\h$ is a Cartan subalgebra of $\g$, we will say that $s\in S(\g)$ is an $\h$-weight vector if there exists $\mu\in\h^*$ such that for all $x\in\h$, $ad\,x(s)=\mu(x) s$. If $\p$ is a (standard) parabolic subalgebra of $\g$, then the set of weights $\Lambda(\p)$ of the algebra of semi-invariants $Sy(\p)$ of $\p$ may be viewed as a subset of $\h^*$ (see Sect. \ref{not}). Hence the $\h$-weight vectors of $Sy(\p)$ are exactly the semi-invariants of $\p$.

A special case of a parabolic subalgebra is a Borel subalgebra $\mathfrak b=\n\oplus\h$ of $\g$ semi-simple, where $\n$ denotes the nilpotent radical  of $\mathfrak b$.
By \cite{J1} the algebra of symmetric semi-invariants $Sy(\mathfrak b)$, resp. the algebra of symmetric invariants $Y(\n)\subset Sy(\mathfrak b)$, is always a polynomial algebra, the former having  ${\rm rank}(\g)=\dim\h$ generators. Moreover both algebras have  the same set of weights.
(See \cite[Tables I and II]{J1} and \cite[Table]{FJ2} for an erratum, for an explicit description of weights and degrees of generators).

\subsection{Canonical truncation.}\label{CT}

Since $\a$ is algebraic,  there exists by~\cite{BGR} a canonically defined subalgebra of $\a$, called the canonical truncation of $\a$ and denoted by $\a_{\Lambda}$, such that $Y(\a_{\Lambda})=Sy(\a_{\Lambda})=Sy(\a)$. We also say that $\a_{\Lambda}$ is the truncated subalgebra of $\a$ : it is the largest subalgebra of $\a$ which vanishes on the weights of $Sy(\a)$. In particular, the canonical truncation of $\a$ has no proper semi-invariants.
By say \cite[29.4.3]{TY} a parabolic subalgebra $\p$ of a semisimple Lie algebra is algebraic, hence one has that $Sy(\p)=Y(\p_{\Lambda})=Sy(\p_{\Lambda})$ where $\p_{\Lambda}$ is the canonical truncation of $\p$. Moreover a result of Chevalley-Dixmier in \cite[Lem. 7]{D0}, also known as a theorem of Rosenlicht, implies that 
$$\ind\p_{\Lambda}={\rm degtr}_{\Bbbk}({\rm Fract}(Y(\p_{\Lambda}))).$$
In other words the index $\ind\p_{\Lambda}$ of $\p_{\Lambda}$ that is, the minimal codimension of a coadjoint orbit in $\p_{\Lambda}^*$, is also equal to the cardinality of a maximal set of algebraically independent elements in $Sy(\p)=Y(\p_{\Lambda})$. It is not known in general whether $Sy(\p)$ is or not finitely generated, but  the transcendence degree of its field of fractions was shown  to be finite with an explicit formula given in \eqref{index} of subsection \ref{ind}.
By~\cite[7.9]{J6} (see also~\cite[Chap. I, Sec. B, 8.2]{F}) the algebra of symmetric invariants $Y(\p)$ of a proper parabolic subalgebra $\p$ in a simple Lie algebra is always reduced to scalars, while by~\cite{D1} its algebra of symmetric semi-invariants $Sy(\p)$ is never. That is why we consider the algebra of symmetric semi-invariants $Sy(\p)=Y(\p_{\Lambda})$ of a parabolic subalgebra $\p$ rather than its algebra of symmetric invariants. Moreover the structure of $Sy(\p)$ may give informations about the field $C(\p)$ of invariant fractions  of  $S(\p)$. Specifically assume that $Sy(\p)=Y(\p_{\Lambda})$ is a polynomial algebra (freely generated by semi-invariants of $\p$). Then, since we have equality  ${\rm Fract}(Y(\p_{\Lambda}))=C(\p_{\Lambda})$, the latter is obviously a pure transcendental extension of $\Bbbk$. Moreover  by \cite[Thm. 66]{O} so is also the field $C(\p)$, answering positively to Dixmier's fourth problem \cite[Problem 4]{D}.

\subsection{Adapted pairs.}\label{ADP}

An adapted pair for $\a$ is a pair $(h,\,y)\in\a\times\a^*$ such that $ad\,h(y)=-y$, where $ad$ denotes here the coadjoint action of $\a$, $h$ is  a semisimple element of $\a$ and $y$ is a regular element in $\a^*$, that is, there exists a subspace $V$ of $\a^*$ of minimal dimension such that $ad\,\a(y)\oplus V=\a^*$ (the dimension of $V$ is called the index of $\a$, denoted by $\ind\a$).

 Call an element of $\a^*$ singular if it is not regular and denote by $\a^*_{sing}$ the set of singular elements in $\a^*$.
The set of regular elements in $\a^*$ is open dense in $\a^*$ and the codimension of $\a^*_{sing}$ is always bigger or equal to one. When equality holds the algebra $\a$ is said to be singular (nonsingular otherwise). The nonsingularity property is also called in~\cite[Def. 1.1]{P}  the ``codimension two property''.

If $(h,\,y)$ is an adapted pair for $\a$, then $y$ belongs to the zero set of the ideal of $S(\a)$ generated by the homogeneous elements of $Y(\a)$ with positive degree. 
 When $\a$ admits an adapted pair and has no proper semi-invariants, then it follows by \cite[1.7]{JS} that the algebra $\a$ is nonsingular.
In particular if $\a$ is a truncated parabolic subalgebra of a simple Lie algebra $\g$ and admits an adapted pair $(h,\,y)$ then by the above, $\a$ is nonsingular.
\subsection{Weierstrass sections.}\label{WS}

A Weierstrass section  for coadjoint action of $\a$ (see \cite{FJ4}) is an affine subspace $y+V$ of $\a^*$ (with $y\in\a^*$ and $V$ a vector subspace of $\a^*$) such that restriction of functions of $S(\a)=\Bbbk[\a^*]$ to $y+V$ induces an algebra isomorphism between $Y(\a)$ and the algebra of polynomial functions $\Bbbk[y+V]$ on $y+V$. Of course, since $\Bbbk[y+V]$ is isomorphic to $S(V^*)$, the existence of a Weierstrass section for coadjoint action of $\a$ implies that the algebra $Y(\a)$ is isomorphic to $S(V^*)$ and then that $Y(\a)$ is a polynomial algebra (on $\dim V$ generators). Moreover, under this isomorphism, a set of homogeneous algebraically independent generators of $Y(\a)$ is sent to  a basis of $V^*$, hence each element of this set is linearized. In \cite{J8} Weierstrass sections were called  algebraic slices.

 Assume that $\a$ has no proper semi-invariants, admits an adapted pair $(h,\,y)$, and that the algebra of symmetric  invariants $Y(\a)$ is polynomial. 
 Then by \cite[2.3]{JS}, for any $ad\, h$-stable complement $V$ of $ad\,\a(y)$ in $\a^*$,
the affine subspace $y+V$ is a Weierstrass section  for coadjoint action of $\a$.

Suppose now that $\a=\p_{\Lambda}$ is the canonical truncation of a proper parabolic subalgebra $\p$ in a simple Lie algebra. By \cite{FJ2} there exist a lower and an upper bounds for the formal character of $Sy(\p)=Y(\p_{\Lambda})$ (see also Sect. \ref{bounds}). Assume that these bounds coincide. This implies by \cite{FJ2} that $Y(\p_{\Lambda})$ is a polynomial algebra over $\Bbbk$. Assume further that we have constructed an adapted pair for $\p_{\Lambda}$. Thus by the above, this adapted pair provides a Weierstrass section for coadjoint action of $\p_{\Lambda}$. 
This method will be used in roughly half of the cases we will consider in this paper.

\subsection{Affine slice.}\label{AS}
An affine slice to the coadjoint action of $\a$  is an affine subspace $y+V$ of $\a^*$ such that
$A.(y+V)$ is dense in $\a^*$ and $y+V$ meets every coadjoint orbit in $A.(y+V)$ at exactly one point and transversally.
Assume that $\a$ has no proper semi-invariants. Then if there exists a Weierstrass section $y+V\subset\a^*$ for coadjoint action of $\a$, one has 
 by \cite[3.2]{FJ4} that $y+V$ is an affine slice to the coadjoint action of $\a$. The converse does not hold in general, but if $(y+V)_{sing}:=(y+V)\cap\a^*_{sing}$  is of codimension at least two in $y+V$ then it holds by \cite[3.3]{FJ4}.
One may also find in \cite{J8} more details on affine slices.

\subsection{The reductive case}\label{RC} Take  $\a=\g$ semisimple. Then  there exists a principal $\s\mathfrak l_2$-triple $(x,\,h,\,\,y)$ of $\g$ with $h\in\g$  a semisimple
element and $x$ and $y$  regular in $\g\simeq\g^*$, such that $[h,\,y]=-y$. Then the pair $(h,\,y)$ is an adapted pair for $\g$. Denote by $\g^x$ the centralizer of $x$ in $\g$.
Then by \cite{K}
$y+\g^x$ is a Weierstrass section and also an affine slice to the coadjoint action of $\g$. It is called the Kostant slice or Kostant section.

\subsection{Magic number and nonsingularity.}\label{MN}

The magic number  of $\a$ is $$c(\a)=\frac{1}{2}(\dim\a+\ind\a).$$ It is always an integer.
By~\cite[Prop. 3.1]{OV} one always has that $c(\a_{\Lambda})=c(\a)$, where $\a_{\Lambda}$ is the canonical truncation of $\a$. When $\a=\g$ is semisimple, one has that $c(\g)=\dim \mathfrak b$ where $\mathfrak b$ is a Borel subalgebra of $\g$.

Assume that $\a$ has no proper semi-invariants and is nonsingular (which is the case by \ref{ADP} when $\a$ admits an adapted pair for instance). Let $f_1,\,\ldots,\,f_l$ be $l=\ind\a$ homogeneous algebraically independent elements of $Y(\a)$.
Then by \cite[Thm. 1.2]{P} 
 \begin{equation}\label{eqdegree}\sum_{i=1}^l \deg(f_i)\ge c(\a)\tag{deg}.\end{equation}
   Moreover by~\cite[5.6]{JS} and \cite[Thm. 1.2]{P}, equality holds in \eqref{eqdegree} if and only if  
 $Y(\a)$ is generated by $f_1,\,,\ldots,\,f_l.$
 
In particular when $\a=\p_{\Lambda}$ is the canonical truncation of a parabolic subalgebra $\p$ then by the above  the existence of a Weierstrass section for coadjoint action of  $\p_{\Lambda}$, given by an adapted pair for $\p_{\Lambda}$, implies that equality  holds in \eqref{eqdegree}  for a set of $\ind\p_{\Lambda}$ homogeneous algebraically independent elements of $Y(\p_{\Lambda})$.

\section{Notation.}\label{not}

Let $\g$ be a semisimple Lie algebra over $\Bbbk$ and $\h$ be a fixed Cartan subalgebra of $\g$. 
Let $\Delta$ be the set of roots of $\g$ (or root system of $\g$) with respect  to $\h$ and $\pi$ a chosen set of simple roots.
Denote by $\Delta^\pm$ the subset of $\Delta$ formed by the positive, resp. negative, roots of $\Delta$, with respect to $\pi$.

With each root $\alpha\in\Delta$ is associated a root vector space $\g_{\alpha}$ and a nonzero root vector $x_{\alpha}\in\g_{\alpha}$. For all $A\subset\Delta$, set $\g_A=\bigoplus_{\alpha\in A}\g_{\alpha}$ and $-A=\{\gamma\in\Delta\mid -\gamma\in A\}$. We denote by $\alpha^\vee$ the coroot associated with the root $\alpha\in\Delta$.
Then $(\alpha^\vee)_{\alpha\in\pi}$ is a basis for the $\Bbbk$-vector space $\h$. We denote
 by $\n$, resp. $\n^-$, the subalgebra of $\g$ such that $\n=\g_{\Delta^+}$, resp. $\n^-=\g_{\Delta^-}$. We have the following triangular decomposition
$$\g=\n\oplus\h\oplus\n^-.$$

A standard parabolic subalgebra of $\g$ is given by the choice of a subset $\pi'$ of $\pi$. That is why we may denote it by $\p_{\pi'}$. 
Let $\Delta^\pm_{\pi'}$ denote the subset of $\Delta^\pm$ associated to $\pi'$, namely $\Delta^\pm_{\pi'}=\pm\mathbb N\pi'\cap\Delta^\pm$. Set $\n^\pm_{\pi'}=\g_{\Delta^\pm_{\pi'}}$. Then 
$\p_{\pi'}=\n\oplus\h\oplus\n^-_{\pi'}.$
Moreover
$\p^-_{\pi'}=\n^+_{\pi'}\oplus\h\oplus\n^-$
is the opposite algebra of $\p_{\pi'}$. 

Via the Killing form $K$ on $\g$, the dual space $\p^*_{\pi'}$ of $\p_{\pi'}$ is isomorphic to $\p^-_{\pi'}$ which is then endowed with the coadjoint action of $\p_{\pi'}$.

We denote by $(\,,\,)$ the non-degenerate symmetric bilinear form on $\h^*\times\h^*$, induced by the Killing form on $\h\times\h$, and denote by $\mathcal H:\h\longrightarrow\h^*$ the isomorphism induced by the latter. The form $(\,,\,)$ is invariant under the action of the Weyl group of $(\g,\,\h)$.  If $\g$ is simple of type ${\rm B}_n$, ${\rm C}_n$ or ${\rm D}_n$, resp. ${\rm A}_n$, we may also view the form $(\,,\,)$ as a scalar product on $\mathbb R^n$, resp. on $\mathbb R^{n+1}$.  For all $\gamma,\,\gamma'\in\h^*$, one has that $\gamma(\mathcal H^{-1}(\gamma'))=(\gamma,\,\gamma')$. We have that $\mathcal H(\alpha^\vee)=2\alpha/(\alpha,\,\alpha)$, for all $\alpha\in\Delta$ so that, for all $\alpha,\,\beta\in\Delta$, we have that $\beta(\alpha^\vee)=(2\alpha/(\alpha,\,\alpha),\,\beta)$.

We use Bourbaki's labelling for the roots, as in \cite[Planches I, resp. II, resp. III, resp. IV]{BOU} when $\g$ is simple of type ${\rm A}_n$, resp. ${\rm B}_n$, resp. ${\rm C}_n$, resp. ${\rm D}_n$.
We then set $\pi=\{\alpha_1,\,\ldots,\alpha_n\}$ and denote by $\varpi_i$, or sometimes $\varpi_{\alpha_i}$, $1\le i\le n$, the fundamental weight associated with $\alpha_i$.
Similarly, if $\pi'=\{\alpha_{i_1},\,\ldots,\,\alpha_{i_r}\}\subset\pi$ we denote by $\varpi'_{i_j}$, or sometimes $\varpi'_{\alpha_{i_j}}$, the fundamental weight associated with $\alpha_{i_j}$ with respect to $\pi'$.

We denote by $\varepsilon_i$, $1\le i\le n$, resp. $1\le i\le n+1$, the elements of an orthonormal basis of $\mathbb R^n$, resp. $\mathbb R^{n+1}$, with respect to the scalar product $(\,,\,)$ and according to which
the simple roots $\alpha_i$, $1\le i\le n$, are expanded as in \cite[Planches II,  III, IV, resp. I]{BOU} for type ${\rm B}_n$, ${\rm C}_n$,  ${\rm D}_n$, resp. ${\rm A}_n$.

Recall the definition of the canonical truncation given in \ref{CT} and denote by $\p_{\pi',\,\Lambda}$ the canonical truncation of $\p_{\pi'}$.
Then one has that
$$\p_{\pi',\,\Lambda}=\n\oplus\h_{\Lambda}\oplus\n^-_{\pi'}$$
where $\h_{\Lambda}\subset\h$ is the largest subalgebra of $\h$ which vanishes on $\Lambda(\p_{\pi'})$, the set of weights of $Sy(\p_{\pi'})$ which may be identified with a subset of $\h^*$. For an explicit description of $\h_{\Lambda}$, see \cite[5.2.2, 5.2.9 and 5.2.10]{FJ3} or \cite[2.2]{FL}.
Denote by $\p'_{\pi'}$ the derived subalgebra of $\p_{\pi'}$ and set $\h'=\h\cap\p'_{\pi'}$. Then $\h'$ is the vector space generated by the coroots $\alpha^\vee$ with $\alpha\in\pi'$ and $\h'\subset\h_{\Lambda}$. Let $w_0$ be the longest element of the Weyl group of $(\g,\,\h)$. If $w_0=-Id$ then $\h_{\Lambda}=\h'$. 
In particular if $\g$ is simple of type ${\rm B}_n$, ${\rm C}_n$, or ${\rm D}_{2m}$, then we have  that $\h_{\Lambda}=\h'$. Now assume that $\g$ is simple of type ${\rm D}_{n}$ with $n$ odd. Then  if both $\alpha_{n-1}$ and $\alpha_n$ do not belong to $\pi'$, we have that 
$$\h_{\Lambda}=\h'\oplus\Bbbk\mathcal H^{-1}(\varpi_n-\varpi_{n-1})=\h'\oplus\Bbbk\mathcal H^{-1}(\ep_n)=\h'\oplus\Bbbk(\alpha_n^\vee-\alpha_{n-1}^\vee),$$ otherwise $\h_{\Lambda}=\h'$.

For convenience we will replace $\p_{\pi'}$ by its opposite algebra $\p^-_{\pi'}$ (simply denoted by $\p$ from now on) and we will consider
the canonical truncation $\p_{\Lambda}=\p^-_{\pi',\,\Lambda}$ of $\p=\p^-_{\pi'}$.
We have that $$\p_{\Lambda}=\p^-_{\pi',\,\Lambda}=\n^-\oplus\h_{\Lambda}\oplus\n^+_{\pi'}$$ and its dual space $\p_{\Lambda}^*$ may be identified via the Killing form $K$ on $\g$ with $\p_{\pi',\,\Lambda}$ (since by \cite[5.2.2, 5.2.9]{FJ3} the restriction of $K$ to
$\h_{\Lambda}\times\h_{\Lambda}$ is non-degenerate).

We will denote by $\g'$ the  Levi subalgebra of $\p$ (and of $\p_{\Lambda}$), namely :
$$\g'=\n^+_{\pi'}\oplus\h'\oplus\n^-_{\pi'}.$$
Then $w_0'$ will denote the longest element of the Weyl group of $(\g',\,\h')$.

\section{Bounds for formal character.}\label{bounds}

Keep the notation of previous Section.
A $\h$-module $M$ is called a weight module if $M=\oplus_{\nu\in\h^*}M_{\nu}$, with finite dimensional weight subspaces $M_{\nu}:=\{m\in M\mid\forall h\in\h,\,h.m=\nu(h)m\}$. For a weight module $M$ one defines the formal character ${\ch}\,M$ of $M$ as follows :
$${\ch}\,M=\sum_{\nu\in\h^*}\dim M_{\nu}e^{\nu}$$ where $e^{\mu+\nu}=e^{\mu}e^{\nu}$ for all $\mu,\,\nu\in\h^*$.
Obviously the formal character is multiplicative on tensor products that is, if $M$ and $N$ are  weight modules, then 
$${\ch}\,(M\otimes N)={\ch}\,M{\ch\,}\,N.$$
Hence if $\mathscr A\subset S(\p)$ is a polynomial algebra with algebraically independent $\h$-weight generators $a_i$, $1\le i\le l$, each of them having a nonzero weight $\lambda_i\in\h^*$, then $${\ch}\,\mathscr A=\prod_{1\le i\le l}(1-e^{\lambda_i})^{-1}.$$
Moreover for weight modules $M$ and $N$,
we write ${\ch}\,M\le{\ch}\, N$ if $\dim M_{\nu}\le \dim N_{\nu}$ for all $\nu\in\h^*$. Hence if $M\subset N$, then ${\ch}\,M\le{\ch}\, N$ and if equality holds then $M=N$.
 
We will specify below (see subsection \ref{equalitybounds}) the lower and upper bounds for ${\ch}\,Y(\p_{\Lambda})$ mentioned in subsection \ref{WS}.
For this, we have to summarize results in \cite{FJ1},\,\cite{FJ2},\,\cite{FJ3} and \cite{J7}.

\subsection{}\label{ind}
Let $i$ and $j$ be involutions of $\pi$ defined as in \cite[5.1]{FJ3} or as in \cite[2.2]{FL}. More precisely $j=-w_0$ and $i(\alpha)=-w_0'(\alpha)$ for all $\alpha\in\pi'$. If now $\alpha\in\pi\setminus\pi'$, then $i(\alpha)=j(\alpha)$ if $j(\alpha)\not\in\pi'$, and otherwise $i(\alpha)=j(ij)^r(\alpha)$ where $r$ is the smallest integer such that $j(ij)^r(\alpha)\not\in\pi'$.
Let $E(\pi')$ be the set of $\langle ij\rangle$-orbits in $\pi$.
By \cite[2.5]{FJ3} and \cite[3.2]{FJ1}, we have that \begin{equation}\label{index}\ind\p_{\Lambda}={\rm degtr}_{\Bbbk}({\rm Fract}(Y(\p_{\Lambda})))=\lvert E(\pi')\rvert.\end{equation}

\subsection{}\label{comppoids}
Following \cite[5.2.1]{FJ3} one may set, for each $\Gamma\in E(\pi')$ :
\begin{equation}\label{poids}
\delta_{\Gamma}=-\sum_{\gamma\in\Gamma}\varpi_{\gamma}-\sum_{\gamma\in j(\Gamma)}\varpi_{\gamma}+\sum_{\gamma\in\Gamma\cap\pi'}\varpi'_{\gamma}+\sum_{\gamma\in i(\Gamma\cap\pi')}\varpi'_{\gamma}.\end{equation}

Note that, for all $\Gamma\in E(\pi')$, one has that $i(\Gamma\cap\pi')=j(\Gamma)\cap\pi'$ by \cite[3.2.2]{FJ2}.

\subsection{}\label{compepsilon}
 Let $\Gamma\in E(\pi')$.  One sets  $d_{\Gamma}=\sum_{\gamma\in\Gamma}\varpi_{\gamma}$ and $d'_{\Gamma}=\sum_{\gamma\in\Gamma\cap\pi'}\varpi'_{\gamma}$ and one denotes by $\mathcal B_{\pi}:=\Lambda(\n\oplus\h)\subset\h^*$, resp. $\mathcal B_{\pi'}:=\Lambda(\n^+_{\pi'}\oplus\h')\subset\h'^*$ the set of weights of the polynomial algebra of symmetric semi-invariants $Sy(\n\oplus\h)$, resp. $Sy(\n^+_{\pi'}\oplus\h')$ :  generators of the set $\mathcal B_{\pi}$ (and then also of $\mathcal B_{\pi'}$) are given in \cite[Table I and II]{J1} and in \cite[Table]{FJ2}.
 The set $\mathcal B_{\pi}$, resp. $\mathcal B_{\pi'}$, is equal to the set of weights of the polynomial algebra $Y(\n)$, resp. $Y(\n^+_{\pi'})$, see below subsection \ref{compdeg}.
 
 Then following \cite[3.2.7]{FJ2} one sets
\begin{equation}\label{epsilon}\ep_{\Gamma}=\begin{cases} 1/2 &{\rm if}\;\Gamma=j(\Gamma)\,{\rm and}\,d_{\Gamma}\in\mathcal B_{\pi}\,{\rm and}\; d'_{\Gamma}\in\mathcal B_{\pi'}\\
1&{\rm otherwise}.\\
\end{cases}\end{equation}

Below we give some details on the set $\mathcal B_{\pi}$, resp. $\mathcal B_{\pi'}$.
For a real number $x$, denote by $[x]$ the integer such that $x-1<[x]\le x$.

Assume  that $\g$ is simple of type ${\rm B}_n$, with $n\ge 2$. Recall that $j=Id_{\pi}$.
Let $\alpha\in\pi$.
If $\alpha=\alpha_{2k}$ with $1\le k\le [(n-1)/2]$, then $\varpi_{2k}\in\mathcal B_{\pi}$. Otherwise $2\varpi_{\alpha}\in\mathcal B_{\pi}$ but $\varpi_{\alpha}\not\in\mathcal B_{\pi}$. 

Now assume that $\g$ is simple of type ${\rm D}_n$, with $n\ge 4$. Then the same as above is true for the first $n-2$ simple roots. Moreover if $n$ is even then $j=Id_{\pi}$ and if $n$ is odd, then $j(\alpha_{n-1})=\alpha_n$ and $j$ is the identity if restricted to the $n-2$ first simple roots. In both cases, if  $\alpha\in\{\alpha_{n-1},\,\alpha_n\}$, then $\varpi_{\alpha}+\varpi_{j(\alpha)}\in\mathcal B_{\pi}$ but $\varpi_{\alpha}\not\in\mathcal B_{\pi}$. 

If $\g$ is simple of type ${\rm C}_n$, with $n\ge 2$, then, for all $1\le i\le n$, $2\varpi_i\in\mathcal B_{\pi}$  but $\varpi_i\not\in\mathcal B_{\pi}$.

Finally if $\alpha$ belongs to a connected component of $\pi'$ of type A, then $\varpi'_{\alpha}+\varpi'_{i(\alpha)}\in\mathcal B_{\pi'}$ but $\varpi'_{\alpha}\not\in\mathcal B_{\pi'}$. 

\subsection{}\label{equalitybounds}
Assume from now on that $\g$ is simple and that the parabolic subalgebra $\p$ is proper that is, $\pi'\subsetneq\pi$.
By \cite[Thm. 6.7]{J7} (see also \cite[7.1]{FJ2}) one has  that
\begin{equation}\label{ch}\prod_{\Gamma\in E(\pi')}(1-e^{\delta_{\Gamma}})^{-1}\le{\ch}\,(Y(\p_{\Lambda}))\le\prod_{\Gamma\in E(\pi')}(1-e^{\ep_{\Gamma}\delta_{\Gamma}})^{-1}.\end{equation}

Assume now that both bounds in \eqref{ch} coincide that is, that $\ep_{\Gamma}=1$ for all $\Gamma\in E(\pi')$.
 For example, it occurs when $\g$ is simple of type A or C.  Then  one deduces that $Sy(\p)=Y(\p_{\Lambda})$ is a polynomial algebra over $\Bbbk$
on $\lvert E(\pi')\rvert$ homogeneous and $\h$-weight algebraically independent generators. One generator corresponds to every $\Gamma\in E(\pi')$ and has  a weight $\delta_{\Gamma}$ given by \eqref{poids} above (recall that one has assumed that the parabolic subalgebra $\p$ contains the negative Borel subalgebra $\n^-\oplus\h$) and a degree  $\partial_{\Gamma}$ which may be easily computed by \cite[5.4.2]{FJ2}. To explain how one may compute this degree (see \eqref{degre} or \eqref{degrebis} below), we have to recall  results in subsection below. 

\subsection{}\label{compdeg}
By \cite{J1} $Y(\n^+_{\pi'})\subset Sy(\n^+_{\pi'}\oplus\h')$, resp. $Y(\n)\subset Sy(\n\oplus\h)$, is a polynomial algebra whose set of homogeneous and $\h'$-weight, resp. $\h$-weight, algebraically independent generators is formed by the elements $a_{\rho'_{\gamma}}$, resp. $a_{\rho_{\gamma}}$ :
their weight $\rho'_{\gamma}$, resp. $\rho_{\gamma}$, and their degree are given in \cite[Table I and II]{J1} and in \cite[Table]{FJ2} and we precise them below.
Recall the sets $\mathcal B_{\pi'}$, resp. $\mathcal B_{\pi}$, of subsection \ref{compepsilon} and that these sets are also the sets of weights of $Y(\n^+_{\pi'})$, resp. of $Y(\n)$. One has that, for all $\gamma\in\pi'$, resp. $\gamma\in\pi$,
$$\rho'_{\gamma}=\varpi'_{\gamma}\;\;\hbox{\rm if}\; \;\varpi'_{\gamma}\in\mathcal B_{\pi'},\;\;\hbox{\rm resp.}\;\;
\rho_{\gamma}=\varpi_{\gamma}\;\;\hbox{\rm if}\;\; \varpi_{\gamma}\in\mathcal B_{\pi}.$$
Otherwise $$\rho'_{\gamma}=\varpi'_{\gamma}+\varpi'_{i(\gamma)},\;\hbox{\rm resp.}\;\; \rho_{\gamma}=\varpi_{\gamma}+\varpi_{j(\gamma)}.$$

Assume that $\g$ is simple of type ${\rm B}_n$, resp. ${\rm D}_n$. For all $1\le u\le[(n-1)/2]$, resp. $1\le u\le[(n-2)/2]$, one has that $$\deg(a_{\rho_{\alpha_{2u}}})=\deg(a_{\varpi_{2u}})=u$$ and for all $1\le u\le [n/2]$, resp. $1\le u\le[(n-1)/2]$, $$\deg(a_{\rho_{\alpha_{2u-1}}})=\deg(a_{2\varpi_{2u-1}})=2u.$$ Moreover for $\g$ of type ${\rm B}_n$, $$\deg(a_{\rho_{\alpha_n}})=\deg(a_{2\varpi_n})=[(n+1)/2].$$
For $\g$ of type ${\rm D}_n$, then for $\alpha\in\{\alpha_{n-1},\,\alpha_n\}$, one has that
$$\deg(a_{\rho_{\alpha}})=\deg(a_{\varpi_{\alpha}+\varpi_{j(\alpha)}})=[n/2].$$

Finally assume that $\g$ is simple of type ${\rm A}_n$, resp. ${\rm C}_n$. Then for all $1\le u\le[(n+1)/2]$, resp. for all $1\le u\le n$, one has that $\deg(a_{\rho_{\alpha_u}})=\deg(a_{\varpi_u+\varpi_{n+1-u}})=u$, resp. $\deg(a_{\rho_{\alpha_u}})=\deg(a_{2\varpi_u})=u$.

\subsection{}\label{degreegenerator}
Assume now that, for all $\Gamma\in E(\pi')$, one has $\ep_{\Gamma}=1$.

Let $\Gamma\in E(\pi')$ be such that $\Gamma=j(\Gamma)$. The degree $\partial_{\Gamma}$ of the homogeneous generator of $Y(\p_{\Lambda})$ corresponding to $\Gamma$ verifies
\begin{multline}\label{degre}
\partial_{\Gamma}=\sum_{\gamma\in\Gamma\mid\rho_{\gamma}=\varpi_{\gamma}}2\deg(a_{\rho_{\gamma}})+\sum_{\gamma\in\Gamma\mid\rho_{\gamma}\neq\varpi_{\gamma}}\deg(a_{\rho_{\gamma}})+\\\sum_{\gamma\in\Gamma\cap\pi'\mid\rho'_{\gamma}=\varpi'_{\gamma}}2\deg(a_{\rho'_{\gamma}})+\sum_{\gamma\in\Gamma\cap\pi'\mid\rho'_{\gamma}\neq\varpi'_{\gamma}}\deg(a_{\rho'_{\gamma}}).\end{multline}

Let $\Gamma\in E(\pi')$ be such that $\Gamma=\{\alpha\}$ with $\alpha\in\pi\setminus\pi'$ and $i(\alpha)\neq\alpha$. Then necessarily one has that $\Gamma\neq j(\Gamma)$ (by \cite[5.2.6]{FJ3}) and there exist two homogeneous generators $s_{\Gamma}$ and $t_{\Gamma}$ of $Y(\p_{\Lambda})$ corresponding to $\Gamma$ (more precisely  one corresponds to $\Gamma$ and the other to $j(\Gamma)$) whose weight $\delta_{\Gamma}=\delta_{j(\Gamma)}$ is given by \eqref{poids} and whose degree $\partial_{\Gamma}$, resp. $\partial_{j(\Gamma)}$, is given by the  formula :

\begin{equation}\label{degrebis}
\partial_{\Gamma}=\deg(s_{\Gamma})=\deg(a_{\rho_{\alpha}})\;\;\hbox{\rm and}\;\;\partial_{j(\Gamma)}=\deg(t_{\Gamma})=\deg(a_{\rho_{\alpha}})+1.\end{equation}

The latter situation can occur when $\g$ is simple of type ${\rm D}_n$ with $n$ odd and when both $\alpha_{n-1}$ and $\alpha_n$ do not belong to $\pi'$ (see Sect. \ref{Cas9}).

\section{Improved upper bound.}\label{IUB}

Keep the notation and hypotheses of Section \ref{not} and  assume that $\p_{\Lambda}$ admits an adapted pair $(h,\,y)\in\h_{\Lambda}\times\p_{\Lambda}^*$.
Since $y$ is regular in $\p_{\Lambda}^*$ 
there exists an $ad\,h$-stable complement $V$ to $ad\,\p_{\Lambda}(y)$ in $\p_{\Lambda}^*$ of dimension $\ind\p_{\Lambda}$. Moreover by \cite[2.2.4]{FJ4}
we may assume that $V=\g_T$ with $T\subset\Delta^+\sqcup\Delta^-_{\pi'}$   that is, $ad\,\p_{\Lambda}(y)\oplus\g_T=\p_{\Lambda}^*$ with $\lvert T\rvert=\ind\p_{\Lambda}$.
Assume further that $y=\sum_{\gamma\in S}x_{\gamma}$ with $S\subset\Delta^+\sqcup\Delta^-_{\pi'}$ and that $S_{\mid\h_{\Lambda}}$ is a basis for $\h_{\Lambda}^*$.
Then for each  $\gamma\in T$, there exists a unique element $s(\gamma)\in\mathbb QS$ such that $\gamma+s(\gamma)$ vanishes on $\h_{\Lambda}$.
By \cite[Lem. 6.11]{J6bis}, one has that
\begin{equation}\label{Improv}
{\ch}\,(Y(\p_{\Lambda}))\le\prod_{\gamma\in T}(1-e^{-(\gamma+s(\gamma))})^{-1}.\end{equation}
The right hand side of the above inequality is called an improved upper bound for ${\ch}\,(Y(\p_{\Lambda}))$.

Assume now that 
\begin{equation}\label{eqIUB}\prod_{\Gamma\in E(\pi')}(1-e^{\delta_{\Gamma}})^{-1}=\prod_{\gamma\in T}(1-e^{-(\gamma+s(\gamma))})^{-1}.\end{equation} Then by \eqref{ch} of Sect. \ref{bounds} equality holds in \eqref{Improv} and  by \cite[Lem. 6.11]{J6bis} the restriction map gives an isomorphism $Y(\p_{\Lambda})\isomto\Bbbk[y+\g_T]$. Then $y+\g_T$ is a Weierstrass section  for coadjoint action of $\p_{\Lambda}$ as defined in \ref{WS}.

This implies that $Y(\p_{\Lambda})$ is a polynomial algebra over $\Bbbk$ on $\lvert E(\pi')\rvert=\lvert T\rvert$ algebraically independent homogeneous and $\h$-weight generators, each of them having $\delta_{\Gamma}$, for $\Gamma\in E(\pi')$, as a weight, given by \eqref{poids} of Sect. \ref{bounds} (this weight is also equal to $-(\gamma+s(\gamma))$, for some $\gamma\in T$). 
Moreover  the degree of each of these generators is equal to $1+\lvert s(\gamma)\rvert$, $\gamma\in T$, where $\lvert s(\gamma)\rvert=\sum_{\alpha\in S} m_{\alpha,\,\gamma}$ if $s(\gamma)=\sum_{\alpha\in S}m_{\alpha,\,\gamma}\alpha$ ($m_{\alpha,\,\gamma}\in\mathbb N$, actually). For all  $\gamma\in T$, the integer $\lvert s(\gamma)\rvert$  is also equal to the eigenvalue of $x_{\gamma}$ with respect to $ad\,h$. (For more details, see \cite[6.11]{J6bis}).

Conversely if $y+\g_T$ is a Weierstrass section for coadjoint action of $\p_{\Lambda}$, then equality holds in (\ref{Improv}) by \cite[Remark 6.11]{J6bis}.

\section{Construction of an adapted pair.}\label{CAP}

As we already said in the previous sections, our Weierstrass sections require the construction of an adapted pair.
This construction  uses the notions we already introduced in \cite{F1},\,\cite{FL} and \cite{FL1}. For convenience we recall some of them, notably the Heisenberg sets and the Kostant cascades.

\subsection{Heisenberg sets and Kostant cascades.}\label{HS}

A Heisenberg set  with centre $\gamma\in\Delta$ (\cite[Def. 7]{F1}) is a subset $\Gamma_{\gamma}$ of $\Delta$ such that $\gamma\in\Gamma_{\gamma}$ and for all $\alpha\in\Gamma_{\gamma}\setminus\{\gamma\}$, there exists a (unique) $\alpha'\in\Gamma_{\gamma}\setminus\{\gamma\}$ such that $\alpha+\alpha'=\gamma$. We may take care to not be confused by the above notation of a Heisenberg set and an element  $\Gamma\in E(\pi')$, resp. $\Gamma_u\in E(\pi')$, which denotes an $\langle ij\rangle$-orbit in $\pi$, the $\langle ij\rangle$-orbit of $\alpha_u\in\pi$.

A typical example of Heisenberg set is given by the Kostant cascade of $\g$ (see also \cite[Example 8]{F1}). More precisely assume that the semisimple Lie algebra $\g$ admits a set of roots $\Delta=\bigsqcup_{i\in I}\Delta_i$ with $I\subset \mathbb N^*$, each $\Delta_i$ being a maximal irreducible root system with highest root $\beta_i$. Then take $(\Delta_i)_{\beta_i}=\{\alpha\in\Delta_i\mid (\alpha,\,\beta_i)=0\}$. For every $i\in I$, set $(\Delta_i)_{\beta_i}=\bigsqcup_{j\in J}\Delta_{ij}$ with $J\subset\mathbb N^*$ and $\Delta_{ij}$ being a maximal irreducible root system  with highest root $\beta_{ij}$. Continuing we obtain a subset $\mathcal K(\g)\subset\mathbb N^*\cup\mathbb {N^*}^2\cup\ldots$ with ${\rm Card}\,\mathcal K(\g)\le{\rm rank}\,\g$, irreducible root systems $\Delta_K$, $K\in\mathcal K(\g)$ and a maximal set $\beta_{\pi}$ of strongly orthogonal positive roots $\beta_K$, $K\in\mathcal K(\g)$, called the {\it Kostant cascade} of $\g$. The subset $\mathcal K(\g)$ admits a partial order $\le$ through $K\le L$ if $K=L$ or if $L=\{K,\,l_1,\,\ldots,\,l_t\}$ with $l_i\in\mathbb N^*$. In type A or C, this order is actually a total order, since the sets $(\Delta_K)_{\beta_K}$ are already irreducible. So one can index the  subset $\beta_{\pi}$ of $\Delta^+$ simply by $\mathbb N$ in these types, so that the roots in $\beta_{\pi}$ are simply denoted by $\beta_i$, $1\le i\le {\rm Card}\,\mathcal K(\g)$.  In type B or D, the order is not total. In type ${\rm B}_n$ or ${\rm D}_{2n+1}$, resp. ${\rm D}_{2n}$, for the elements $\beta_K$, $K\in\mathcal K(\g)$, we use the notation $\beta_i,\,\beta_{i'}$, resp. $\beta_i,\,\beta_{i'},\,\beta_{i''}$ with order relation $i<i'$, resp. $i<i'$ and $i<i''$.
  For more details, see for example \cite[Table]{FJ2}, \cite[Table I]{FL}, \cite[Sect. 7]{FL1} or \cite[Tables I, II, III]{J1}.

 Let $\beta_K$ be an element of the Kostant cascade $\beta_{\pi}$ of $\g$ and set
$$H_{\beta_K}=\{\alpha\in\Delta_K\mid (\alpha,\,\beta_K)>0\}.$$
Then $H_{\beta_K}$ is a Heisenberg set with centre $\beta_K$ : it is the largest Heisenberg set with centre $\beta_K$ which is included in $\Delta^+$ by $ii)$ of Lemma below. Moreover the vector subspace $\g_{H_{\beta_K}}$ of $\g$ associated with $H_{\beta_K}$ (with the notation in Sect. \ref{not}) is a Heisenberg Lie subalgebra of $\g$  by $iv)$ of Lemma below.
Of course all the Heisenberg sets are not necessarily associated with Heisenberg Lie subalgebras and even not with Lie subalgebras of $\g$, since $iv)$ of Lemma below need not be true for a Heisenberg set in general.

By \cite[Lem. 2.2]{J1} (see also \cite[Lem. 3]{FL}) we have the following Lemma, which is very useful to construct adapted pairs
thanks to the Kostant cascade $\beta_{\pi}$ and to the largest Heisenberg sets $H_{\beta}$, $\beta\in\beta_{\pi}$, which are defined above.

\begin{lm}{\cite[Lem. 2.2]{J1}}

Let $\beta_{\pi}$ denote the Kostant cascade of $\g$. Then we have that :
\begin{enumerate}

\item[i)]
$\Delta^+=\bigsqcup_{\beta\in\beta_{\pi}}H_{\beta}$ (disjoint union).

\item[ii)]
If $\gamma,\,\delta\in\Delta^+$ are such that $\gamma+\delta=\beta\in\beta_{\pi}$ then $\gamma,\,\delta\in H_{\beta}\setminus\{\beta\}$.

\item[iii)] 
If $\gamma\in H_{\beta_K}$ and $\delta\in H_{\beta_L}$ are such that $\gamma+\delta\in H_{\beta_M}$ with $K,\,L,\,M\in\mathcal K(\g)$, then $K\le L$ (resp. $L\le K$) and $M=K$ (resp. $M=L$).

\item[iv)] 
If $\gamma,\,\delta\in H_{\beta},\,\beta\in\beta_{\pi}$, and $\gamma+\delta\in\Delta$ then $\gamma+\delta=\beta$.

\end{enumerate}

\end{lm}

For an explicit description of Kostant cascades, see for example \cite{FL}, \cite{FL1} or \cite{J1}.
The Heisenberg sets (not only the largest Heisenberg sets $H_{\beta}$, $\beta\in\beta_{\pi}$) are very helpful for the construction of an adapted pair. They were used in \cite{J5}, resp. in \cite{F1},\,\cite{FL} and \cite{FL1}, to build adapted pairs for every truncated biparabolic subalgebra in a simple Lie algebra of type ${\rm A}$, resp. for truncated maximal parabolic subalgebras. Below is a proposition where Heisenberg sets appear to be crucial for constructing an adapted pair.

\subsection{A  proposition of regularity.}\label{REG}

The following proposition (see \cite[Prop. 9]{F1}) is a generalization of \cite[Thm. 8.6]{J5}.
We keep the notation of Sect. \ref{not} and consider $S$, $T$ and $T^*$ disjoint subsets of $\Delta^+\sqcup\Delta^-_{\pi'}$ and set $y=\sum_{\gamma\in S} x_{\gamma}$.

\begin{prop}{\cite[Prop. 9]{F1}}\label{propAP}

We assume that, for each $\gamma\in S$, there exists $\Gamma_{\gamma}\subset\Delta^+\sqcup\Delta^-_{\pi'}$ a Heisenberg set with centre $\gamma$ and that all the sets $\Gamma_{\gamma}$, for $\gamma\in S$, together with $T$ and $T^*$ are disjoint.  

We also assume that we can decompose $S$ into $S^+\sqcup S^-$ where $S^+$, resp. $S^-$, is the subset of $S$ containing those $\gamma\in S$ with $\Gamma_{\gamma}\subset\Delta^+$, resp. $\Gamma_{\gamma}\subset\Delta^-_{\pi'}$. 

For all $\gamma\in S$, set $\Gamma_{\gamma}^0=\Gamma_{\gamma}\setminus\{\gamma\}$, $O=\bigsqcup_{\gamma\in S}\Gamma_{\gamma}^0$ and $O^{\pm}=\bigsqcup_{\gamma\in S^{\pm}}\Gamma_{\gamma}^0$. 

We assume further that :
\begin{enumerate}
\item[(i)] $S_{\mid\h_{\Lambda}}$ is a basis for $\h_{\Lambda}^*$.

\item[(ii)] If $\alpha\in\Gamma_{\gamma}^0$ with $\gamma\in S^+$, is such that there exists $\beta\in O^+$, with $\alpha+\beta\in S$, then $\beta\in\Gamma_{\gamma}^0$ and $\alpha+\beta=\gamma$.

\item[(iii)] If $\alpha\in\Gamma_{\gamma}^0$ with $\gamma\in S^-$, is such that there exists $\beta\in O^-$, with $\alpha+\beta\in S$, then $\beta\in\Gamma_{\gamma}^0$ and $\alpha+\beta=\gamma$.

\item[(iv)] $\Delta^+\sqcup\Delta^-_{\pi'}=\bigsqcup_{\gamma\in S}\Gamma_{\gamma}\sqcup T\sqcup T^*$.

\item[(v)] For all $\alpha\in T^*$, $\g_{\alpha}\subset ad\,\p_{\Lambda}(y)+\g_T$.

\item[(vi)] $\lvert T\rvert=\ind\p_{\Lambda}$.
\end{enumerate}
Then $y$  is regular in $\p^*_{\Lambda}$ and $$ad\,\p_{\Lambda}(y)\oplus\g_T=\p_{\Lambda}^*.$$ Moreover we can uniquely define
$h\in\h_{\Lambda}$ by $\gamma(h)=-1$ for all $\gamma\in S$, and then $(h,\,y)$ is an adapted pair for $\p_{\Lambda}$.
\end{prop}

We give below the proof of the above proposition for the reader's convenience.

\begin{proof}

Condition $(iv)$ implies that $\p_{\Lambda}=\h_{\Lambda}\oplus\g_{-O}\oplus\g_{-S}\oplus\g_{-T^*}\oplus\g_{-T}$ and that $\p^*_{\Lambda}=\h_{\Lambda}\oplus\g_O\oplus\g_{S}\oplus\g_{T^*}\oplus\g_{T}$.

Let $\Phi_y$ denote the skew-symmetric bilinear form defined by $\Phi_y(x,\,x')=K(y,\,[x,\,x'])$ for all $x,\,x'\in \g$ where recall $K$ is the Killing form on $\g$.

Conditions $(ii)$ and $(iii)$ imply  by~\cite[Lem. 8.5]{J5} that the restriction of $\Phi_y$ to $\g_{-O} \times\g_{-O}$ is non-degenerate. Then $\g_O\subset ad\,\g_{-O}(y)+\h_{\Lambda}+\g_S+\g_T+\g_{T^*}$.

But since  $O\cap S=\emptyset$ one has that for all $x\in \g_O$ and $x'\in\g_{-O}$, the element $x-ad\,x'( y)$ belongs to the orthogonal of $\h_{\Lambda}$ for the Killing form. Then $\g_O\subset ad\,\g_{-O}(y)+\g_S+\g_T+\g_{T^*}$.

Condition $(i)$ implies that $\g_S=ad\,\h_{\Lambda}(y)$ and that $\h_{\Lambda}\subset ad\,\g_{-S}(y)+\g_O+\g_S+\g_T+\g_{T^*}$.
Condition $(v)$ implies that $\g_{T^*}\subset ad\,\p_{\Lambda}(y)+\g_T$.
Hence $\p^*_{\Lambda}=\h_{\Lambda}\oplus\g_O\oplus\g_{S}\oplus\g_{T^*}\oplus\g_{T}\subset ad\,\p_{\Lambda}(y)+\g_T$. Finally condition $(vi)$ implies that the latter sum is direct, since $\dim\g_T=\ind\p_{\Lambda}\le{\rm codim}\,ad\,\p_{\Lambda}(y)$.
\end{proof}

\begin{Rqs}
\rm

\begin{enumerate}
\item Notice that \cite[Thm. 8.6]{J5} is a special case of the above Proposition, with $T^*=\emptyset$. Here we need to take sometimes  a set $T^*\neq\emptyset$ as in \cite{F1}.\par
\item In \cite[Lem. 3.2 and Lem. 6.1]{FL1} lemmas were given to insure  condition $(v)$ in the above Proposition. In this paper, as in \cite{F1}, we verify by hand that condition $(v)$ of the above Proposition is satisfied, using if necessary Lemma and Prop. \ref{condition(v)} below.\par
\item Assume that there exists an adapted pair $(h,\,y)$ for $\p_{\Lambda}$ and denote by $\g_T$ a complement of $ad\,\p_{\Lambda}(y)$ in $\p_{\Lambda}^*$, with $T\subset\Delta^+\sqcup\Delta^-_{\pi'}$. 
\begin{enumerate}
\item \label{rqeqbounds} Assume further that $\varepsilon_{\Gamma}=1$ for all $\Gamma\in E(\pi')$ (as defined in \eqref{epsilon} of Sect. \ref{bounds}). Then $Y(\p_{\Lambda})$ is a polynomial algebra and by what we said in subsection \ref{WS} one has that $y+\g_T$ is a Weierstrass section for coadjoint action of $\p_{\Lambda}$ (since $\g_T$ is $ad\,h$-stable).\par
\item\label{rqeqIUB} Assume now that there exists $\Gamma\in E(\pi')$ such that $\varepsilon_{\Gamma}=1/2$. Assume further that \eqref{eqIUB} of Sect. \ref{IUB} holds. Then by what we said in Sect. \ref{IUB}, $y+\g_T$ is a Weierstrass section for coadjoint action of $\p_{\Lambda}$.
\end{enumerate}
\end{enumerate}
\end{Rqs}

\subsection{Condition {\it (v)} of Prop. \ref{propAP}.}\label{condition(v)}
Keeping the notation of Sect. \ref{not}, we consider $S,\,T,\,T^*\subset\Delta^+\sqcup\Delta^-_{\pi'}$ three disjoint subsets and $y=\sum_{\alpha\in S}x_{\alpha}$. 
We give in the Proposition below a sufficient condition which implies condition $(v)$ of Prop. \ref{propAP} for some roots $\alpha\in T^*$.
Recall Sect. \ref{not} that for all $\alpha\in\Delta$, we have fixed a nonzero root vector $x_{\alpha}\in\g_{\alpha}$, that we will rescale if necessary, except those associated with the roots $\alpha\in S$, since $y=\sum_{\alpha\in S}x_{\alpha}$ is fixed.

\begin{lm}

Let $\gamma_1,\,\gamma_2,\,\gamma_3\in S$ and $\gamma'_1,\,\gamma'_2,\,\gamma'_3\in(\Delta^-\sqcup\Delta^+_{\pi'})\setminus S$ such that 
\begin{enumerate}
\item $\gamma_i+\gamma'_i\in(\Delta^+\sqcup\Delta^-_{\pi'})\setminus S$ for all $1\le i\le 3$ 
\item $\gamma_2+\gamma'_2=\gamma_1+\gamma'_3$
\item $\gamma_3+\gamma'_3=\gamma_2+\gamma'_1$ 
\item $\gamma_1+\gamma'_1+\gamma_2\in\Delta$
\item $\gamma_1+\gamma_2\not\in\Delta$, $\gamma_2+\gamma_3\not\in\Delta$, $\gamma_1+\gamma_3\not\in\Delta$.
\end{enumerate}

Then $\gamma_1+\gamma'_1=\gamma_3+\gamma'_2$ and up to rescaling the nonzero root vectors $x_{\gamma'_i}\in\p_{\Lambda}$ for all $1\le i\le 3$ and the nonzero root vectors $x_{\gamma_i+\gamma'_i}\in\p_{\Lambda}^*$ for all $1\le i\le 3$, we have that

$$(\Sigma)\begin{cases}
[x_{\gamma'_1},\,x_{\gamma_1}]=[x_{\gamma'_2},\,x_{\gamma_3}]=x_{\gamma_1+\gamma'_1}\\
[x_{\gamma'_2},\,x_{\gamma_2}]=[x_{\gamma'_3},\,x_{\gamma_1}]=x_{\gamma_2+\gamma'_2}\\
[x_{\gamma'_3},\,x_{\gamma_3}]=[x_{\gamma'_1},\,x_{\gamma_2}]=x_{\gamma_3+\gamma'_3}\\
\end{cases}$$

\end{lm}

\begin{proof} The equality $\gamma_1+\gamma'_1=\gamma_3+\gamma'_2$ comes directly from the equalities (2) and (3). Moreover
the rescaling of the nonzero root vectors $x_{\gamma'_i}$ and $x_{\gamma_i+\gamma'_i}$ (which is possible since the roots $\gamma'_i$ and $\gamma_i+\gamma'_i$ do not belong to $S$) gives for example the last two equalities of $(\Sigma)$. Then we obtain the first one, since we prove easily that $[x_{\gamma'_1},\,x_{\gamma_1}]=[x_{\gamma'_2},\, x_{\gamma_3}]$. Indeed by applying Jacobi identity several times, 
it is easy to prove, under the assumptions, that $[[x_{\gamma'_1},\,x_{\gamma_1}],\,x_{\gamma_2}]=[[x_{\gamma'_2},\, x_{\gamma_3}],\,x_{\gamma_2}]$ and using (4) one can conclude. 
\end{proof}

We then have directly the following proposition.
\begin{prop}

Let $\gamma_i$ and $\gamma'_i$, for $1\le i\le 3$, be roots satisfying the hypotheses of previous lemma. Recall that  $y=\sum_{\gamma\in S} x_{\gamma}$ and let $X,\,X',\,X''$ be vectors in $\p_{\Lambda}^*$ such that, after a possible rescaling of some suitable root vectors, we have

$$\begin{cases}
ad\,x_{\gamma'_1}(y)=x_{\gamma_1+\gamma'_1}+x_{\gamma_3+\gamma'_3}+X\\
 ad\,x_{\gamma'_2}(y)=x_{\gamma_1+\gamma'_1}+x_{\gamma_2+\gamma'_2}+X'\\
ad\,x_{\gamma'_3}(y)=x_{\gamma_2+\gamma'_2}+x_{\gamma_3+\gamma'_3}+X''\\
\end{cases}$$
with
$$\begin{cases}X\not\in Vect(x_{\gamma_1+\gamma'_1},\,x_{\gamma_3+\gamma'_3})\setminus\{0\}\\
X'\not\in Vect(x_{\gamma_1+\gamma'_1},\,x_{\gamma_2+\gamma'_2})\setminus\{0\}\\
X''\not\in Vect(x_{\gamma_2+\gamma'_2},\,x_{\gamma_3+\gamma'_3})\setminus\{0\}.\\
\end{cases}$$  If $X,\,X',\,X''\in ad\,\p_{\Lambda}(y)+\g_T$, 
then  $x_{\gamma_i+\gamma'_i}\in ad\,\p_{\Lambda}(y)+\g_T$ for all $1\le i\le 3$.

\end{prop}
Actually we will apply the previous proposition with $X,\,X',\,X''$ being vectors for which it will be immediate to verify that they belong to $ad\,\p_{\Lambda}(y)+\g_T$ by direct computation. Moreover one of the $\gamma_i+\gamma'_i$ will belong to the subset $T^*$. See for example proof of Lemma \ref{condition(v)caseeven}.

\subsection{The Kostant cascade in type A}\label{KCA}

Keep the notation of Sect. \ref{not} and assume that $\g$ is a simple Lie algebra of type ${\rm B}_n$, ${\rm C}_n$ or ${\rm D}_n$. We consider $\p=\n^-\oplus\h\oplus\n^+_{\pi'}$ the standard parabolic subalgebra of $\g$ containing the negative Borel subalgebra ${\mathfrak b}^-=\n^-\oplus\h$ and associated to the subset $\pi'\subset\pi$. 
Recall  that we are interested in studying $\p$ which is equal to $\p_{s,\,\ell}$, resp. $\p_{\ell}$, resp. $\q_{s,\,\ell}$ with $s\in\mathbb N^*$ and $\ell\in\mathbb N$, as defined in subsection \ref{Cases}. 
Then the subset $\pi'$ associated to $\p$ is $\pi'=\pi\setminus\{\alpha_s,\,\alpha_{s+2},\,\ldots,\,\alpha_{s+2\ell}\}$ with $1\le s\le n-2\ell$, resp. $\pi'=\pi\setminus\{\alpha_{n-1-2\ell},\,\ldots,\,\alpha_{n-1},\,\alpha_n\}$ and $\g$ of type ${\rm D}_n$, resp. $\pi'=\pi\setminus\{\alpha_s,\,\alpha_{s+2},\,\ldots,\,\alpha_{s+2\ell},\,\alpha_{n-1},\,\alpha_n\}$ with $s+2\ell\le n-2$ and $\g$ of type ${\rm D}_n$.

If $s\ge 2$ and in the  cases of $\p_{s,\,\ell}$ or of $\q_{s,\,\ell}$, $\pi'_1=\{\alpha_1,\,\alpha_2,\,\ldots,\,\alpha_{s-1}\}$ will  denote the connected component of $\pi'$ of type ${\rm A}_{s-1}$ and in the case of $\p_{\ell}$, $\pi'_1=\{\alpha_1,\,\alpha_2,\,\ldots,\,\alpha_{n-2-2\ell}\}$ will denote the  connected component of $\pi'$ of type ${\rm A}_{n-2-2\ell}$, if $n-2-2\ell\ge 1$. To keep homogeneous notation we will set in this subsection $s=n-1-2\ell$ when we are in the latter case.
We denote by $\beta_{\pi'_1}$ the Kostant cascade (see  \ref{HS}) of the simple Lie subalgebra $\g_{\pi'_1}$ of the Levi subalgebra $\g'$ of $\p$ which is of type ${\rm A}_{s-1}$. We also denote by ${\pi'_1}^\vee$ the subset of $\h'$ formed by the coroots $\alpha^\vee$ with $\alpha\in\pi'_1$ and by $\Delta^+_{\pi'_1}:=\Delta^+\cap\mathbb N\pi'_1$.
We have that $$\beta_{\pi'_1}=\{\beta'_i=\ep_i-\ep_{s+1-i}\mid 1\le i\le [s/2]\}\subset\Delta^+_{\pi'_1}.$$ Set $\beta^0_{\pi'_1}:=\beta_{\pi'_1}\setminus(\beta_{\pi'_1}\cap\pi'_1)$. If $s$ is odd, then $\beta^0_{\pi'_1}=\beta_{\pi'_1}$ and if $s$ is even, then $\beta^0_{\pi'_1}=\{\beta'_i\mid 1\le i\le (s-2)/2\}$.
The following lemma will be useful for the next sections, notably to prove that, for a suitable subset $S\subset\Delta^+\sqcup\Delta^-_{\pi'}$, one has that $S_{\mid\h_{\Lambda}}$ is a basis for $\h_{\Lambda}^*$ (see Lemma \ref{baseFC} or Lemma \ref{BaseC}).
If $s$ is even, set $t:=[s/4]$ and if $s$ is odd, set $t:=[(s+1)/4]$.
We consider the subset $\{h'_j\}_{1\le j\le [(s-1)/2]}\subset{\pi'_1}^\vee$, with the following order.
If $t=s/4$ with $s$ even, resp. $t=(s+1)/4$ with $s$ odd, then 
$$\{h'_j\}=\{{h'_{2j-1}=\alpha^\vee_{2j-1}},\,h'_{2j}={\alpha^\vee_{s-2j}}\,; 1\le j\le t-1,\,h'_{2t-1}={\alpha^\vee_{2t-1}}\}.$$
If $t=(s-2)/4$ with $s$ even, resp. $t=(s-1)/4$ with $s$ odd, then  
$$\{h'_j\}=\{{h'_{2j-1}=\alpha^\vee_{2j-1}},\,h'_{2j}={\alpha^\vee_{s-2j}}\,; 1\le j\le t\}.$$

\begin{lm}

Let $A$ be the square matrix of size $[(s-1)/2]$ which entries are $-\beta'_i(h'_j)$ with $1\le i,\,j\le [(s-1)/2]$.
Then $A$ is a lower triangular matrix with $-1$ on the diagonal. Hence $\det A=(-1)^{[(s-1)/2]}$.

\end{lm}

\begin{proof}
Recall the construction of the Kostant cascade of $\g_{\pi'_1}$ (see \ref{HS}).
 Set $\Delta_1^+
=\Delta^+_{\pi'_1}$, then set $\Delta^+_2=\{\alpha\in\Delta_1^+;\;(\alpha,\beta'_1)=0\}$. Here $\beta'_1$ is the highest root of $\g_{\pi'_1}$ and $\beta'_1=\varpi'_1+\varpi'_{s-1}$. Then $\Delta^+_2=\Delta^+_1\cap\mathbb N\pi'_2$ where $\pi'_2=\pi'_1\setminus\{\alpha_1,\,\alpha_{s-1}\}$. Continuing we set $\Delta^+_{i+1}=\{\alpha\in\Delta^+_i;\;(\alpha,\,\beta'_i)=0\}$ where $\beta'_i$ is the highest root of $\Delta^+_i$. Then we have that
$\Delta^+_{i+1}\subset\Delta^+_i\subset\cdots\subset\Delta^+_1$ and then $(\beta'_i,\,\alpha)=0$ for all $\alpha\in\Delta^+_j$ with $j>i$. Finally observe that, for all $1\le j\le [(s+1)/4]$, $\alpha_{2j-1}\in\Delta^+_{2j-1}$, $\alpha_{s-2j}\in\Delta^+_{2j}$, $\beta'_{2j-1}(\alpha_{2j-1}^\vee)=(\beta'_{2j-1},\,\alpha_{2j-1})=1$ and that $\beta'_{2j}(\alpha_{s-2j}^\vee)=(\beta'_{2j},\,\alpha_{s-2j})=1$ while $2j\le (s-1)/2$. Hence the lemma.
\end{proof}

\section{Some examples.}\label{Ex}

Before stating the main result (see subsection \ref{Mainres}), we give below two examples 
which will enlighten our construction of  a Weierstrass section, each of these examples using a different method to obtain the latter from the adapted pair we construct. 
Thus  case \ref{cas3} of subsection \ref{Mainres} (see also Sect. \ref{FC}) is illustrated
 by the first example and  case \ref{cas9} of subsection \ref{Mainres} (see also Sect. \ref{Cas9}) is illustrated by the second example.
We keep the notation of Sect. \ref{not}.

\subsection{Comparison of multiplicities.}\label{Gr}

Assume that we have constructed an adapted pair $(h,\,y)\in\h_{\Lambda}\times\p_{\Lambda}^*$ for $\p_{\Lambda}$ via Prop. \ref{propAP}.
Let $\lambda\in \Bbbk$ and set $\mathfrak r=\h_{\Lambda}\oplus\g_O\oplus\g_S\oplus\g_{T^*}\subset\p_{\Lambda}^*$ (one has that $\mathfrak r\oplus\g_T=\p_{\Lambda}^*$). Recall that the endomorphism $ad\,h$ of $\p_{\Lambda}^*$, resp. of $\p_{\Lambda}$ (with $ad$ the coadjoint action, resp. the adjoint action)  is semisimple. Then $\lambda$ is an eigenvalue of $ad\,h$ on $\p_{\Lambda}^*$ if and only if $-\lambda$ is an eigenvalue of $ad\,h$ on $\p_{\Lambda}$. Write $m'_{\lambda}$ for the multiplicity of $\lambda$ in $\mathfrak r$, $m_{\lambda}$ for the multiplicity of $\lambda$ in $\p_{\Lambda}$ and $m^*_{\lambda}$ for the multiplicity of $\lambda$ in $\p_{\Lambda}^*$. Then by the above $m_{-\lambda}=m^*_{\lambda}$ and obviously $m'_{\lambda}\le m_{-\lambda}$. Moreover since $ad\,h(y)=-y$ and that $\p_{\Lambda}^*=ad\,\p_{\Lambda}(y)\oplus\g_T$, we must have that 
\begin{equation}\label{mult}
m'_{\lambda}\le m_{\lambda+1}\end{equation}(see also \cite[7.1]{F1}).

In the examples below, we will check that inequality (\ref{mult}) is satisfied.

\subsection{First example.} 

We assume that the Lie algebra $\g$ is simple of type ${\rm B}_6$ and we set $\pi'=\pi\setminus\{\alpha_2,\,\alpha_4\}$. Then we consider the parabolic subalgebra
$\p=\p^-_{\pi'}$ as defined in Sect. \ref{not}. We are then in case \ref{cas3} of subsection \ref{Mainres}.
We take 
$S=S^+\sqcup S^-$ with $$S^+=\{\ep_1+\ep_3,\,\ep_2,\,\ep_4+\ep_5\},\,
 S^-=\{-\beta''_1=-\ep_5-\ep_6\}$$ where $\beta''_1$ is an element of the Kostant cascade of $\g'$,
 $$T=\{\ep_1+\ep_2,\,\ep_1-\ep_3,\,\ep_2+\ep_4,\,\ep_4-\ep_5,\,\ep_4-\ep_3,\,\ep_6-\ep_5\},$$
 $$T^*=\{\ep_2+\ep_6,\,\ep_2+\ep_5,\,\ep_2-\ep_1,\,\ep_2-\ep_5,\,\ep_2-\ep_4,\,\ep_2-\ep_3,\,\ep_6,\,\ep_2-\ep_6\}.$$
We set 
$$\begin{array}{cl}\Gamma_{\ep_1+\ep_3}=&\{\ep_1+\ep_3,\,\ep_1+\ep_4,\,\ep_3-\ep_4,\,\ep_1+\ep_5,\,\ep_3-\ep_5,\\
&\ep_1+\ep_6,\,\ep_3-\ep_6,\,\ep_1,\,\ep_3,\,\ep_1-\ep_6,\,\ep_3+\ep_6,\\
&\ep_1-\ep_5,\,\ep_3+\ep_5,\,\ep_1-\ep_4,\,\ep_3+\ep_4,\,\ep_1-\ep_2,\,\ep_2+\ep_3\},\end{array}$$
$$\Gamma_{\ep_2}=\{\ep_2\},$$
$$\Gamma_{\ep_4+\ep_5}=\{\ep_4+\ep_5,\,\ep_4+\ep_6,\,\ep_5-\ep_6,\,\ep_4,\,\ep_5,\,\ep_4-\ep_6,\,\ep_5+\ep_6\},$$
$$\Gamma_{-\ep_5-\ep_6}=\{-\ep_5-\ep_6,\,-\ep_5,\,-\ep_6\}=-H_{\beta''_1}$$
where $H_{\beta''_1}$ is the largest Heisenberg set  with centre $\beta''_1$ included in $\Delta^+_{\pi'}$, as defined in  \ref{HS}.

By setting $y=\sum_{\gamma\in S}x_{\gamma}$ one verifies (see for more details Sect. \ref{FC}) that all conditions of Proposition \ref{propAP} are satisfied (indeed it is more complicated than what we have to do in the second example). 
Then $h\in\h'$ such that $\gamma(h)=-1$ for all $\gamma\in S$ is :
 $$h=\ep_1-\ep_2-2\ep_3+2\ep_4-3\ep_5+4\ep_6=\alpha_1^\vee-2\alpha_3^\vee-3\alpha_5^\vee+1/2\alpha_6^\vee.$$ 
Hence the pair $(h,\,y)$ is an adapted pair for $\p_{\Lambda}$. This adapted pair is not sufficient a priori to give a Weierstrass section for coadjoint action of $\p_{\Lambda}$,
since there is one $\Gamma\in E(\pi')$ such that $\varepsilon_{\Gamma}=1/2$. But we can easily check that \eqref{eqIUB}  in Sect. \ref{IUB} holds. Hence by Remark \ref{rqeqIUB} of subsection \ref{REG}
 one has that $y+\g_T$ is a Weierstrass section for coadjoint action of $\p_{\Lambda}$, and then $Y(\p_{\Lambda})$ is a polynomial algebra over $\Bbbk$ (result which was not yet known since the criterion that $\ep_{\Gamma}=1$ for all $\Gamma\in E(\pi')$ is here not satisfied).

To convince oneself that $(h,\,y)$ given above is indeed an adapted pair for $\p_{\Lambda}$ (although the inequality \eqref{mult} of \ref{Gr} is just a necessary condition), one gives
in the table below  the multiplicities $m'_{\lambda}$ and $m^*_{\lambda}=m_{-\lambda}$ for all eigenvalue $\lambda\in \Bbbk$ of $\ad h$ on $\p_{\Lambda}^*$ and one easily checks that inequality \eqref{mult} of \ref{Gr} holds.

\newcolumntype{R}[1]{>{\raggedleft\arraybackslash}b{#1}}
\newcolumntype{C}[1]{>{\centering\arraybackslash}b{#1}}
\newcolumntype{L}[1]{>{\raggedright\arraybackslash}b{#1}}

\begin{center}
\begin{tabular}{|C{2cm}|C{0.6cm}|C{0.6cm}|C{0.6cm}|C{0.6cm}|C{0.6cm}|C{0.6cm}|C{0.6cm}|C{0.6cm}|C{0.6cm}|}

\hline $\lambda$&-7&-6&-5&-4&-3&-2&-1&0\\
\hline $m'_{\lambda}$&1&1&2&3&4&4&5&5\\
\hline $m_{-\lambda}$&1&1&2&3&4&4&5&6\\
\hline

\end{tabular}

\begin{tabular}{|C{2cm}|C{0.6cm}|C{0.6cm}|C{0.6cm}|C{0.6cm}|C{0.6cm}|C{0.6cm}|C{0.6cm}|C{0.6cm}|C{0.6cm}|}

\hline $\lambda$&1&2&3&4&5&6&7\\
\hline $m'_{\lambda}$&4&4&3&2&1&1&0\\
\hline $m_{-\lambda}$&5&4&4&3&2&1&1\\
\hline

\end{tabular}
\end{center}

\subsection{Second example.}

Here we assume that $\g$ is simple of type ${\rm D}_9$ and consider $\pi'=\pi\setminus\{\alpha_1,\,\alpha_3,\,\alpha_5,\,\alpha_8,\,\alpha_9\}$ and the parabolic subalgebra $\p=\p^-_{\pi'}$ associated with $\pi'$. Here we are in case \ref{cas9} of subsection \ref{Mainres}.
We take $S=S^+\sqcup S^-$ with
$$S^+=\{\beta_1=\ep_1+\ep_2,\,\beta_2=\ep_3+\ep_4,\,\beta_3=\ep_5+\ep_6,\,\tilde\beta_4=\beta_4-\alpha_8=\ep_7+\ep_9\}$$
$$\hbox{\rm and}\;\;\;S^-=\{-\beta''_1=\ep_8-\ep_6\}.$$

Here $\beta_i=\ep_{2i-1}+\ep_{2i}$ ($1\le i\le 4$) are elements of the Kostant cascade $\beta_{\pi}$ of $\g$ and $\beta''_1$ is an element of the Kostant cascade $\beta_{\pi'}$ of $\g'$. More precisely setting $\beta^0_{\pi'}=\beta_{\pi'}\setminus(\beta_{\pi'}\cap\pi')$, we have that
$S^-=-\beta^0_{\pi'}$.
We also set
$$T=\{\ep_1-\ep_2,\,\ep_3-\ep_4,\,\ep_5-\ep_6,\,\ep_7+\ep_8,\,\ep_7-\ep_9,\,\ep_8-\ep_9,\,\ep_3-\ep_2,\,\ep_5-\ep_4\}$$
and $T^*=\emptyset$.

For all $1\le i\le 3$, we take $\Gamma_{\beta_i}=H_{\beta_i}$ where $H_{\beta}$ is the largest Heisenberg set with centre $\beta\in\beta_{\pi}$ which is included in $\Delta^+$ as defined in subsection \ref{HS}.
We set $\Gamma_{\tilde\beta_4}=\{\tilde\beta_4,\,\ep_7-\ep_8,\,\ep_8+\ep_9\}$ and $\Gamma_{-\beta''_1}=-H_{\beta''_1}$ where $H_{\beta''_1}\subset\Delta^+_{\pi'}$ is the largest Heisenberg set with centre $\beta''_1$ which is included in $\Delta^+_{\pi'}$.
Since $H_{\beta_4}\cup H_{\alpha_7}=\Gamma_{\tilde\beta_4}\sqcup(T\cap H_{\beta_4})$,  Lemma \ref{HS} $i)$ gives condition $(iv)$ of Prop. \ref{propAP}. Moreover  Lemma \ref{HS} $ii)$ and $iii)$ gives conditions $(ii)$ and $(iii)$ of Prop. \ref{propAP}. Finally we verify by hand that conditions $(i)$ and $(vi)$ of Prop. \ref{propAP} are satisfied, noting that $\h_{\Lambda}=\h'\oplus\Bbbk(\alpha_9^\vee-\alpha_8^\vee)$. Setting
$$h=-\alpha_2^\vee-2\alpha_4^\vee-3\alpha_6^\vee+4\alpha_7^\vee-4(\alpha_9^\vee-\alpha_8^\vee)\in\h_{\Lambda}$$ and $y=\sum_{\alpha\in S} x_{\alpha}$, one checks that $(h,\,y)$ is an adapted pair for $\p_{\Lambda}$. Moreover one checks easily that both bounds in \eqref{ch} of Sect. \ref{bounds} coincide, then $Y(\p_{\Lambda})$ is a polynomial algebra and by what we said in Remark \ref{rqeqbounds} of subsection \ref{REG}, $y+\g_T$ is a Weierstrass section for coadjoint action of $\p_{\Lambda}$. In the table below we give the multiplicities $m'_{\lambda}$ and $m^*_{\lambda}=m_{-\lambda}$ for all eigenvalue $\lambda\in \Bbbk$ of $\ad h$ on $\p_{\Lambda}^*$ and one easily checks that inequality \eqref{mult} of \ref{Gr} holds.

\begin{center}
\begin{tabular}{|C{2cm}|C{0.6cm}|C{0.6cm}|C{0.6cm}|C{0.6cm}|C{0.6cm}|C{0.6cm}|C{0.6cm}|C{0.6cm}|C{0.6cm}|C{0.6cm}|}
\hline $\lambda$&-12&-11&-10&-9&-8&-7&-6&-5&-4&-3\\
\hline $m'_{\lambda}$&1&2&2&2&2&3&3&3&3&4\\
\hline$m_{-\lambda}$&1&2&2&2&2&3&3&3&3&4\\
\hline
\end{tabular}

\begin{tabular}{|C{2cm}|C{0.6cm}|C{0.6cm}|C{0.6cm}|C{0.6cm}|C{0.6cm}|C{0.6cm}|C{0.6cm}|C{0.6cm}|C{0.6cm}|C{0.6cm}|}
\hline$\lambda$&-2&-1&0&1&2&3&4&5&6&7\\
\hline$m'_{\lambda}$&5&7&7&5&4&3&3&3&3&2\\
\hline $m_{-\lambda}$&5&7&7&6&5&5&5&4&3&2\\
\hline
\end{tabular}

\begin{tabular}{|C{2cm}|C{0.6cm}|C{0.6cm}|C{0.6cm}|C{0.6cm}|C{0.6cm}|C{0.6cm}|C{0.6cm}|C{0.6cm}|C{0.6cm}|C{0.6cm}|}
\hline$\lambda$&8&9&10&11&12&13&14&15\\
\hline$m'_{\lambda}$&2&2&2&1&0&0&0&0\\
\hline$m_{-\lambda}$&2&2&2&1&0&0&0&1\\
\hline

\end{tabular}
\end{center}
\section{Cases {\ref{cas1}} and {\ref{cas2}}  for type B or D.}\label{SC}

In this Section we consider truncated parabolic subalgebras described  in  \ref{cas1} and in \ref{cas2} of subsection \ref{Mainres}.
More precisely (with the notation of Sect. \ref{not} and of subsection \ref{Cases}) let $\p=\p_{s,\,\ell}=\n^-\oplus\h\oplus\n^+_{\pi'}$ be a parabolic subalgebra associated to the subset
$\pi'=\pi\setminus\{\alpha_s,\,\alpha_{s+2},\ldots,\alpha_{s+2\ell}\}$ with $\ell\in\mathbb N$ and $s$ an {\it odd} integer,
$1\le s\le n-2\ell$, in a simple Lie algebra $\g$ of type ${\rm B}_n$, resp. ${\rm D}_n$, with $n\ge 2$, resp. $n\ge 4$.

If $\ell=0$, then the parabolic subalgebra $\p$ is maximal and this case was already treated in \cite{FL}. Thus we will assume
from now on that $\ell\ge 1$. 
Note that $\h_{\Lambda}=\h'$, in type ${\rm B}_n$ but also in type ${\rm D}_n$ with the above hypotheses,
by what we said in Sect. \ref{not} (since it is not true that $\alpha_{n-1}$ and $\alpha_n$ are  both deleted  from $\pi'$).

Here we will show (see lemma \ref{coincidencebounds}) that the lower and upper bounds for ${\ch}\,(Y(\p_{\Lambda}))$ in \eqref{ch} of Sect. \ref{bounds} coincide, and then the algebra of symmetric invariants $Y(\p_{\Lambda})$ is  polynomial. 
By Remark \ref{rqeqbounds} of subsection \ref{REG} the existence of an adapted pair for $\p_{\Lambda}$ is sufficient to give a Weierstrass section for coadjoint action of $\p_{\Lambda}$.
Our construction of an adapted pair for $\p_{\Lambda}$ generalizes the construction of an adapted pair in \cite[Sect. 4 and 5]{FL} in case of a maximal parabolic subalgebra.

We will use Proposition \ref{propAP}, which here is quite easy  to apply. Indeed it suffices to take $S\cup T$ to be the union of the Kostant cascade in $\g$ and the opposite of the Kostant cascade in $\g'$. Moreover for each $\gamma\in S^+$, resp. $\gamma\in S^-$, we take the  Heisenberg set $\Gamma_{\gamma}$  to be  equal, resp. to be the opposite, to $H_{\gamma}$, resp. of $H_{-\gamma}$, where $H_{\gamma}$, resp. $H_{-\gamma}$, is the largest Heisenberg set with centre $\gamma$, resp. $-\gamma$, included in $\Delta^+$, resp. $\Delta^+_{\pi'}$, as defined in \ref{HS}. Here moreover we set $T^*=\emptyset$. Then Lemma \ref{HS} will give most of conditions of Prop. \ref{propAP}.

\subsection{The Kostant cascades.}\label{KcasBD}\par\noindent

Recall \ref{HS} the Kostant cascade $\beta_{\pi}$ of $\g$  and set $\beta_{\pi}^0=\beta_{\pi}\setminus(\beta_{\pi}\cap\pi)$.
If $\g$ is of type ${\rm B}_n$ then we have that $$\beta^0_{\pi}=\{\beta_i=\ep_{2i-1}+\ep_{2i}\mid 1\le i\le[n/2]\},$$ and
if $\g$ is of type ${\rm D}_n$ then $$\beta^0_{\pi}=\{\beta_i=\ep_{2i-1}+\ep_{2i}\mid 1\le i\le [(n-1)/2]\}.$$
Moreover if $\g$ is of type ${\rm B}_n$, then we have that $$\beta_{\pi}\cap\pi=\{\alpha_{2i-1}\mid 1\le i\le [(n+1)/2]\}.$$
 If $\g$ is of type ${\rm D}_n$ and $n$ odd, then we have that $$\beta_{\pi}\cap\pi=\{\alpha_{2i-1}\mid 1\le i\le (n-1)/2\},$$ 
and if $\g$ is of type ${\rm D}_n$ and $n$ even, then we have that $$\beta_{\pi}\cap\pi=\{\alpha_n,\,\alpha_{2i-1}\mid 1\le i\le n/2\}.$$ 
Now for the Kostant cascade $\beta_{\pi'}$ of $\g'$, set similarly $\beta^0_{\pi'}=\beta_{\pi'}\setminus(\beta_{\pi'}\cap\pi')$.
If $\g$ is of type ${\rm B}_n$, then we have that 
$$\begin{array}{cl}
\beta^0_{\pi'}=&\bigl\{\beta'_i=\ep_i-\ep_{s+1-i}\mid 1\le i\le (s-1)/2,\\
&\;\beta''_j=\ep_{s+2\ell+2j-1}+\ep_{s+2\ell+2j}\mid \,
1\le j\le [(n-s-2\ell)/2]\bigr\}\end{array}$$ and $$\beta_{\pi'}\cap\pi'=\bigl\{\alpha_{s+2\ell+2i-1},\;\alpha_{s+2j-1}\mid 1\le i\le [(n-s-2\ell+1)/2],\,1\le j\le \ell\bigr\}.$$
Now suppose that $\g$ is of type ${\rm D}_n$ and that $s+2\ell\le n-2$.
 Then we have that $$\begin{array}{cl}
 \beta^0_{\pi'}=&\bigl\{\beta'_i=\ep_i-\ep_{s+1-i} \mid 1\le i\le (s-1)/2,\\
& \beta''_j=\ep_{s+2\ell+2j-1}+\ep_{s+2\ell+2j}\mid 
1\le j\le [(n-s-1-2\ell)/2]\bigr\}.\end{array}$$
If moreover $n$ is odd then $$\beta_{\pi'}\cap\pi'=\bigl\{\alpha_n,\;\alpha_{s+2\ell+2i-1},\;\alpha_{s+2j-1}\mid 1\le i\le (n-s-2\ell)/2,\,1\le j\le \ell\bigr\}$$ and
if  $n$ is even then $$\beta_{\pi'}\cap\pi'=\bigl\{\alpha_{s+2\ell+2i-1},\;\alpha_{s+2j-1}\mid 1\le i\le (n-s-2\ell-1)/2,\,1\le j\le \ell\bigr\}.$$
Now assume that $\g$ is of type ${\rm D}_n$ and that $s+2\ell\in\{n-1,\,n\}$. Since the case $\pi'=\pi\setminus\{\alpha_s,\,\alpha_{s+2},\ldots,\alpha_{s+2\ell-2},\,\alpha_{n-1}\}$ and the case
$\pi'=\pi\setminus\{\alpha_s,\,\alpha_{s+2},\ldots,\alpha_{s+2\ell-2},\,\alpha_{n}\}$ are symmetric, one may suppose that we are in the latter case.
 More precisely if $n$ is odd then we assume that $\pi'=\pi\setminus\{\alpha_s,\,\alpha_{s+2},\,\ldots,\,\alpha_{n-2},\,\alpha_n\}$ and if $n$ is even then we assume that $\pi'=\pi\setminus\{\alpha_s,\,\alpha_{s+2},\,\ldots,\,\alpha_{n-3},\,\alpha_n\}$.
If $n$ is odd then 
$$\beta^0_{\pi'}=\bigl\{\beta'_i=\ep_i-\ep_{s+1-i}\,\;;\;1\le i\le (s-1)/2\bigr\}$$
$$\hbox{\rm and}\;\;\beta_{\pi'}\cap\pi'=\bigl\{\alpha_{s+2i-1}=\ep_{s+2i-1}-\ep_{s+2i}\;;\;1\le i\le\ell=(n-s)/2\bigr\}.$$
If $n$ is even then $$\beta^0_{\pi'}=\bigl\{\beta'_i=\ep_i-\ep_{s+1-i},\,\beta''_1=\ep_{n-2}-\ep_n\;;\;1\le i\le (s-1)/2\bigr\}$$
$$\hbox{\rm and}\;\;\beta_{\pi'}\cap\pi'=\bigl\{\alpha_{s+2i-1}=\ep_{s+2i-1}-\ep_{s+2i}\;;\;1\le i\le\ell-1=(n-3-s)/2\bigr\}.$$

\subsection{Conditions {\it (i)} to {\it (v)} of Proposition \ref{propAP}.}\label{APodd}\par\noindent

For $\g$ of type ${\rm D}_n$ with $n$ even and $s+2\ell\le n-2$, we set $$S^+=\beta^0_{\pi}\sqcup
\{\beta_{n/2}:=\alpha_n=\ep_{n-1}+\ep_n\}.$$
Otherwise we set $S^+=\beta^0_{\pi}$.

For $\g$ of type ${\rm D}_n$ with $n$ odd and $s+2\ell\le n-2$, we set $$S^-=-\beta^0_{\pi'}\sqcup\{-\beta''_{(n-2\ell-s)/2}:=-\alpha_n=-(\ep_{n-1}+\ep_n)\}.$$
Otherwise we set  $S^-=-\beta^0_{\pi'}$.

For $\g$ of type ${\rm D}_n$ with $n$ even and $s+2\ell\le n-2$, we set $$T^+=(\beta_{\pi}\cap\pi)\setminus\{\alpha_n\}.$$
For $\g$ of type ${\rm D}_n$ with $n$ odd and $s+2\ell\le n-2$, we set $$T^-=-(\beta_{\pi'}\cap\pi')\setminus\{-\alpha_n\}.$$
Otherwise we set $T^+=\beta_{\pi}\cap\pi$ and $T^-=-(\beta_{\pi'}\cap\pi')$. 
Finally we set $S=S^+\sqcup S^-$, $T=T^+\sqcup T^-$ and $T^*=\emptyset$.
Then $S^+,\,T^+\subset\Delta^+$ and $S^-,\,T^-\subset\Delta^-_{\pi'}$. 
In all cases we have  that $\beta_{\pi}=S^+\sqcup T^+$ and $-\beta_{\pi'}=S^-\sqcup T^-$.

Then for all $\gamma\in S^+$, resp. $\gamma\in S^-$, we choose $\Gamma_{\gamma}=H_{\gamma}$, resp. $\Gamma_{\gamma}=-H_{-\gamma}$, where $H_{\gamma}$, resp. $H_{-\gamma}$, is the largest Heisenberg set with centre $\gamma\in\beta_{\pi}$, resp. $-\gamma\in\beta_{\pi'}$, included in $\Delta^+$, resp. $\Delta^+_{\pi'}$, as defined in  \ref{HS}.
Observe also that, if $\alpha\in\beta_{\pi}\cap\pi$, resp. $\alpha\in\beta_{\pi'}\cap\pi'$, then $H_{\alpha}=\{\alpha\}$.

By Lemma \ref{HS} $i)$ we have that $\Delta^+=\bigsqcup_{\gamma\in S^+}\Gamma_{\gamma}\sqcup T^+$, resp. $\Delta^-_{\pi'}=\bigsqcup_{\gamma\in S^-}\Gamma_{\gamma}\sqcup T^-$, hence condition $(iv)$ of Proposition \ref{propAP}  is satisfied. 
 Conditions $(ii)$ and $(iii)$ of Proposition \ref{propAP} follow from Lemma \ref{HS} $ii)$. Moreover condition $(v)$ of Prop. \ref{propAP} is empty since $T^*=\emptyset$. 
Below we check condition $(i)$ of Proposition \ref{propAP}.

\begin{lm}\label{baseFC}

$S_{\mid\h_{\Lambda}}$ is a basis for $\h_{\Lambda}^*$.

\end{lm}

\begin{proof}

Recall that $\h_{\Lambda}=\h'$ and remark that $\lvert S\rvert=\dim\h'=n-\ell-1$.

Assume first that $\g$ is of type ${\rm D}_n$ with $n$ odd and that $s+2\ell=n$. Then $\lvert S\rvert=(n-1)/2+(s-1)/2=n-\ell-1$
and one may order the elements $s_u$ of $S$ as 
$$\beta_1,\,\beta_2,\,\ldots,\,\beta_{(n-1)/2},\,-\beta'_1,\,-\beta'_2,\,\ldots,\,-\beta'_{(s-1)/2}$$
and choose the following (ordered) basis $h_v$ of $\h'$ :
$$\begin{array}{cc}
h_i=\alpha_{2i}^\vee,\,\,1\le i\le (n-1)/2,\\
h_{(n-1)/2+2j-1}=h'_{2j-1}=\alpha_{2j-1}^\vee,\,h_{(n-1)/2+2j}=h'_{2j}=\alpha_{s-2j}^\vee,\\
\,1\le j\le [(s+1)/4]\end{array}$$
without repetition for the $h'_j$'s.
Then observe that, for all $1\le i\le (n-3)/2$, one has $\beta_i=\ep_{2i-1}+\ep_{2i}=\varpi_{2i}-\varpi_{2i-2}$
if we set $\varpi_0=0$ and $\beta_{(n-1)/2}=\ep_{n-2}+\ep_{n-1}=\varpi_{n-1}+\varpi_n-\varpi_{n-3}$.
It follows that the matrix $(s_u(h_v))_{1\le u,\,v\le (n-\ell-1)}$ has the form
$$\begin{pmatrix} A&0\\
*&B\\
\end{pmatrix}$$
where $A=(\beta_i(h_j))_{1\le i,\,j\le(n-1)/2}$ is a $(n-1)/2\times(n-1)/2$ lower triangular matrix with $1$ on the diagonal, and $B=(-\beta'_i(h'_j))_{1\le i,\,j\le (s-1)/2}$ is a $(s-1)/2\times(s-1)/2$ which by Lemma \ref{KCA} is a lower triangular matrix with $-1$ on the diagonal. Hence $\det(s_u(h_v))_{1\le u,\,v\le (n-\ell-1)}\neq 0$ and we are done in this case.\par\noindent
Assume now that $\g$ is of type ${\rm D}_n$ with $n$ even and that $s+2\ell=n-1$. Then consider the parabolic subalgebra $\p$ of $\g$ associated to $\pi'=\pi\setminus\{\alpha_s,\,\alpha_{s+2},\,\ldots,\,\alpha_{n-3},\,\alpha_n\}$ ($1\le s\le n-3$ is still an odd integer). Then $\lvert S\rvert=(n-2)/2+(s-1)/2+1=n-\ell-1$ and one may order the elements $s_u$ of $S$ as 
$$\beta_1,\,\beta_2,\,\ldots,\,\beta_{(n-2)/2},\,-\beta'_1,\,-\beta'_2,\,\ldots,\,-\beta'_{(s-1)/2},\,\ep_n-\ep_{n-2}$$
and choose the following (ordered) basis $h_v$ of $\h'$ :
$$\begin{array}{cc}
h_i=\alpha_{2i}^\vee,\,\,1\le i\le (n-2)/2,\\
\,h_{(n-2)/2+2j-1}=h'_{2j-1}=\alpha_{2j-1}^\vee,\,h_{(n-2)/2+2j}=h'_{2j}=\alpha_{s-2j}^\vee,\\
\,1\le j\le [(s+1)/4],\\
\,h_{n-\ell-1}=\alpha_{n-1}^\vee\\
\end{array}$$
without repetition for the $h'_j$'s.
Similarly as above one obtains that the matrix $(s_u(h_v))_{1\le u,\,v\le (n-\ell-1)}$ has the form
$$\begin{pmatrix} A&0\\
*&B\\
\end{pmatrix}$$
where $A=(\beta_i(h_j))_{1\le i,\,j\le(n-2)/2}$ is a $(n-2)/2\times(n-2)/2$ lower triangular matrix with $1$ on the diagonal, and $B$ is a $(s+1)/2\times(s+1)/2$ lower triangular matrix with $-1$ on the diagonal by Lemma \ref{KCA}. Hence $\det(s_u(h_v))_{1\le u,\,v\le (n-\ell-1)}\neq 0$ and we are done in this case.\par\noindent
Assume that $\g$ is of type ${\rm B}_n$.
Then one may order the elements $s_u$ of $S$ as follows :
$$\begin{array}{cc}
s_i=\beta_i,\,1\le i\le [n/2],\\
s_{[n/2]+j}=-\beta'_j,\,1\le j\le (s-1)/2,\\
s_{[n/2]+(s-1)/2+k}=-\beta''_k,\;1\le k\le [(n-s-2\ell)/2)]\\
\end{array}$$
and choose the following (ordered) basis $h_v$ of $\h'$ :
$$\begin{array}{cc}
h_i=\alpha_{2i}^\vee,\,1\le i\le [n/2]\\
h_{[n/2]+2j-1}=h'_{2j-1}=\alpha_{2j-1}^\vee,\,h_{[n/2]+2j}=h'_{2j}=\alpha_{s-2j}^\vee,\,1\le j\le [(s+1)/4]\\
h_{[n/2]+(s-1)/2+k}=\alpha_{s+2\ell+2k}^\vee,\;\, 1\le k\le [(n-2\ell-s)/2]\\
\end{array}$$
without repetition for the $h'_j$'s.
Now if $\g$ is of type ${\rm D}_n$ with $s+2\ell\le n-2$, we take the same set $S$ ordered as above and the same basis of $\h'$, up to replacing $\alpha_n^\vee$ by $2\ep_n$.
Then by what we explained before, the matrix $(s_u(h_v))_{1\le u,\,v\le (n-\ell-1)}$ has the form
$$\begin{pmatrix} A&0&0\\
*&B&0\\
*&*&C\\
\end{pmatrix}$$
where $A$ is a $[n/2]\times[n/2]$ lower triangular matrix with one on the diagonal (except for the case $n$ even where $2$ is the last entry of the diagonal), $B$ is a $(s-1)/2\times(s-1)/2$ lower triangular matrix with $-1$ on the diagonal (by Lemma \ref{KCA}) and $C$ is a $[(n-s-2\ell)/2]\times[(n-s-2\ell)/2]$ lower triangular matrix with $-1$ on the diagonal (except for the case $n$ odd where $-2$ is the last entry of the diagonal). Hence $\det(s_u(h_v))_{1\le u,\,v\le (n-\ell-1)}\neq 0$ and the proof is complete.
\end{proof}

\subsection{Condition {\it (vi)} of Proposition \ref{propAP}}

\begin{lm}\label{TS}

One has that $\lvert T\rvert=\ind\p_{\Lambda}$.

\end{lm}

\begin{proof}

Recall \eqref{index} of Sect. \ref{bounds}, that $\ind\p_{\Lambda}=\lvert E(\pi')\rvert$ where $E(\pi')$ is the set of $\langle ij\rangle$-orbits in $\pi$.
Assume first that $\g$ is of type ${\rm D}_n$ and that $s+2\ell\in\{n-1,\,n\}$.
If $n$ is odd, then $$\begin{array}{cl}
E(\pi')=&\Bigl\{\Gamma_u=\{\alpha_u,\,\alpha_{s-u}\}, \,1\le u\le (s-1)/2,\\
&\Gamma_v=\{\alpha_{v}\},\,\, s\le v\le n-2,\,
\Gamma_{n-1}=\{\alpha_n,\,\alpha_{n-1}\}\Bigr\}.\\
\end{array}$$
If $n$ is even, assuming that $\pi'=\pi\setminus\{\alpha_s,\,\alpha_{s+2},\ldots,\,\alpha_{s+2\ell-2},\,\alpha_{n}\}$, with $s+2\ell-2=n-3$,
then
$$\begin{array}{cl}
E(\pi')=&\Bigl\{\Gamma_u=\{\alpha_u,\,\alpha_{s-u}\},\; 1\le u\le (s-1)/2,\\
&\Gamma_v=\{\alpha_{v}\},\, s\le v\le n-3,\,
\Gamma_{n-1}=\{\alpha_{n-2},\,\alpha_{n-1}\},\,\Gamma_n=\{\alpha_n\}\Bigr\}.
\end{array}$$
Hence $\ind\p_{\Lambda}=\lvert E(\pi')\rvert =n-(s+1)/2$.\par\noindent
On the other hand, one has that :
for $n$ even, $\lvert T^+\rvert=n/2+1$ and $\lvert T^-\rvert=\ell-1=(n-3-s)/2$,
and for $n$ odd, $\lvert T^+\rvert=(n-1)/2$, $\lvert T^-\rvert=\ell=(n-s)/2$. 
Then $\lvert T\rvert=\ind\p_{\Lambda}$ in both cases.\par\noindent
Now assume that $\g$ is of type ${\rm B}_n$. Then
$$\begin{array}{cc}
E(\pi')=\Bigl\{\Gamma_u=\{\alpha_u,\,\alpha_{s-u}\}, \,1\le u\le (s-1)/2,\\
\Gamma_v=\{\alpha_{v}\},\; s\le v\le n\Bigr\}.\end{array}$$
Hence $\ind\p_{\Lambda}=\lvert E(\pi')\rvert =n-(s-1)/2$ and one checks that this is also equal to $\lvert T\rvert$.\par\noindent
Finally assume that $\g$ is of type ${\rm D}_n$ and that $s+2\ell\le n-2$. Then
$$\begin{array}{cl}
E(\pi')=&\Bigl\{\Gamma_u=\{\alpha_u,\,\alpha_{s-u}\}, \,1\le u\le (s-1)/2,\\
&\Gamma_v=\{\alpha_{v}\},\,\, s\le v\le n-2,\,\Gamma_{n-1}=\{\alpha_{n-1},\,\alpha_n\}\Bigr\}.\\
\end{array}$$
Hence $\ind\p_{\Lambda}=\lvert E(\pi')\rvert =n-(s+1)/2$ and one checks that this is also equal to $\lvert T\rvert$.
\end{proof}

\subsection{}\label{ADPFC}

All conditions of Proposition \ref{propAP} are satisfied, thus one can deduce the following corollary.

\begin{cor}

Set $y=\sum_{\alpha\in S} x_{\alpha}$. Then $y$ is regular in $\p_{\Lambda}^*$ and more precisely
$ad\,\p_{\Lambda}(y)\oplus\g_T=\p_{\Lambda}^*$. Moreover since $S_{\mid\h_{\Lambda}}$ is a basis for $\h_{\Lambda}^*$,
there exists a uniquely defined element $h\in\h_{\Lambda}$ such that $\alpha(h)=-1$ for all $\alpha\in S$. Thus the pair $(h,\,y)$ is an adapted pair for $\p_{\Lambda}$.

\end{cor}

\subsection{Coincidence of bounds of Sect. \ref{bounds}.}\label{coincidencebounds}

Recall Remark \ref{rqeqbounds} of subsection \ref{REG} that it suffices to show that both bounds in \eqref{ch} of Sect. \ref{bounds} coincide to obtain a Weierstrass section for coadjoint action of $\p_{\Lambda}$.
This is the following lemma.

\begin{lm}

For all $\Gamma\in E(\pi')$, one has that $\ep_{\Gamma}=1$. Then $Y(\p_{\Lambda})$ is a polynomial algebra over $\Bbbk$.

\end{lm}

\begin{proof}

Recall subsection \ref{compepsilon}.
We will show, for all $\Gamma\in E(\pi')$ such that $j(\Gamma)=\Gamma$, that  $d_{\Gamma}\not\in\mathcal B_{\pi}$ or $d'_{\Gamma}\not\in\mathcal B_{\pi'}$, hence $\ep_{\Gamma}=1$.
Recall the $\langle ij\rangle$-orbits in $E(\pi')$ given in the proof of Lemma \ref{TS}.

For all $1\le u\le (s-1)/2$, one has that $d_{\Gamma_u}=\varpi_u+\varpi_{s-u}\not\in \mathcal B_{\pi}$ since $u$ and $s-u$ are of different parity.
Hence $\ep_{\Gamma_u}=1$.\par\noindent
For $s\le v\le n-2$, (with the restriction that $v\le n-3$ if $n$ even and $\g$ of type ${\rm D}_n$ with $s+2\ell=n-1$) one has that $d_{\Gamma_v}=\varpi_v\not\in\mathcal B_{\pi}$ if $v$ is odd and $d'_{\Gamma_v}=\varpi'_v\not\in\mathcal B_{\pi'}$ if $v$ is even. Hence $\ep_{\Gamma_v}=1$.\par\noindent
Now if $\g$ is of type ${\rm D}_n$, $s+2\ell=n-1$  and $n$ even, one has that $d_{\Gamma_{n-1}}=\varpi_{n-2}+\varpi_{n-1}\not\in\mathcal B_{\pi}$ and $d_{\Gamma_n}=\varpi_n\not\in\mathcal B_{\pi}$.  Hence $\ep_{\Gamma_{n-1}}=
\ep_{\Gamma_n}=1$.\par\noindent
If  $\g$ is of type ${\rm D}_n$, $s+2\ell=n$  and $n$ odd, then $d_{\Gamma_{n-1}}=\varpi_{n-1}+\varpi_n\in\mathcal B_{\pi}$,
but $d'_{\Gamma_{n-1}}=\varpi'_{n-1}\not\in\mathcal B_{\pi'}$. Hence $\ep_{\Gamma_{n-1}}=1$.\par\noindent
If $\g$ is of type ${\rm D}_n$, $s+2\ell\le n-2$, then $d_{\Gamma_{n-1}}=\varpi_{n-1}+\varpi_n$,
and $d'_{\Gamma_{n-1}}=\varpi'_{n-1}+\varpi'_n$. One of them does not belong to $\mathcal B_{\pi}$, resp. $\mathcal B_{\pi'}$, since  $\g'$  is of type ${\rm D}_{n-s-2\ell}$, and since $n$ and $n-s-2\ell$ are of different parity. Hence $\ep_{\Gamma_{n-1}}=1$.\par\noindent
Finally assume that $\g$ is of type ${\rm B}_n$ and take $v\in\{n-1,\,n\}$. Then $d_{\Gamma_v}=\varpi_v\not\in\mathcal B_{\pi}$ if $n$ is even.
 If now $n$ is odd then $d_{\Gamma_n}=\varpi_n\not\in\mathcal B_{\pi}$ while $d_{\Gamma_{n-1}}=\varpi_{n-1}\in\mathcal B_{\pi}$. But since $n$ is odd, $\alpha_{n-1}\in\pi'$ and $d'_{\Gamma_{n-1}}=\varpi'_{n-1}\not\in\mathcal B_{\pi'}$. Hence $\ep_{\Gamma_v}=1$. This completes the proof.
\end{proof}

\subsection{A Weierstrass section.}

Summarizing the above results, we obtain by Remark \ref{rqeqbounds} of subsection \ref{REG} the following Theorem.
\begin{thm}\label{ThmFC}

Let $\g$ be a complex simple Lie algebra of type ${\rm B}_n$, resp. ${\rm D}_n$ with $n\ge 2$, resp. $n\ge 4$, and let $\p=\n^-\oplus\h\oplus\n^+_{\pi'}$ be a parabolic subalgebra associated to $\pi'=\pi\setminus\{\alpha_s,\,\alpha_{s+2},\,\ldots,\,\alpha_{s+2\ell}\}$ where $s,\,\ell\in\mathbb N^*$ and $1\le s\le n-2\ell$, $s$ odd. Then there exists a Weierstrass section  for coadjoint action of $\p_{\Lambda}$.

\end{thm}

\subsection{Weights and Degrees.}

For completeness we give below the weights and degrees of a set of homogeneous and $\h$-weight algebraically independent generators of $Y(\p_{\Lambda})$.
Since both bounds in \eqref{ch} of Sect. \ref{bounds} coincide then, for all $\Gamma\in E(\pi')$, each homogeneous and $\h$-weight generator has $\delta_{\Gamma}$ as a weight given by \eqref{poids} of Sect. \ref{bounds} and a degree $\partial_{\Gamma}$  given by \eqref{degre} of Sect. \ref{bounds} (since here for all $\Gamma\in E(\pi')$, we have that $j(\Gamma)=\Gamma$). 

Below  are weights and degrees of a set of homogeneous and $\h$-weight algebraically independent generators of $Y(\p_{\Lambda})$, each of them corresponding to an $\langle ij\rangle$-orbit $\Gamma_r$ in $E(\pi')$.\smallskip

Assume  that $\g$ is of type ${\rm B}_n$ and that $s+2\ell<n$  : \medskip

\begin{tabular}{|c|c|c|}
\hline
$\langle ij\rangle$-orbit in $E(\pi')$& Weight & Degree\\
\hline
$\Gamma_u=\{\alpha_u,\,\alpha_{s-u}\}$, &$-2\varpi_s$&$s+1+2u$\\
\;$1\le u\le (s-1)/2$&&\\
\hline
$\Gamma_v=\{\alpha_v\}$,&$-2\varpi_v$&$v+1$\\
$v=s+2k,\,0\le k\le\ell$ &&\\
\hline
$\Gamma_v=\{\alpha_v\}$,& $-\varpi_{v-1}-\varpi_{v+1}$&$v+1$\\
$v=s+2k-1,\,1\le k\le\ell$&& \\
\hline
$\Gamma_v=\{\alpha_v\}$,&$-2\varpi_{s+2\ell}$&$2v+1-s-2\ell$\\
$s+2\ell+1\le v\le n-1$ && \\
\hline
$\Gamma_n=\{\alpha_n\}$ & $-\varpi_{s+2\ell}$ & $n-\ell+(1-s)/2$\\
\hline
\end{tabular}
\medskip

Assume that  $\g$ is of type ${\rm B}_n$ and that $s+2\ell=n$ (hence $n$ is odd)  :\medskip

\begin{tabular}{|c|c|c|}
\hline
$\langle ij\rangle$-orbit in $E(\pi')$& Weight & Degree\\
\hline
$\Gamma_u=\{\alpha_u,\,\alpha_{s-u}\}$,&$-2\varpi_s$&$s+1+2u$\\
$1\le u\le (s-1)/2$&&\\
\hline
$\Gamma_v=\{\alpha_v\}$,\; & $-2\varpi_v$&$v+1$\\
$v=s+2k,\,0\le k\le\ell-1$&&\\
\hline
$\Gamma_v=\{\alpha_v\}$,\; & $-\varpi_{v-1}-\varpi_{v+1}$&$v+1$\\
$v=s+2k-1,\,1\le k\le\ell-1$&&\\
\hline
$\Gamma_{n-1}=\{\alpha_{n-1}\}$ & $-\varpi_{n-2}-2\varpi_n$ & $n$\\
\hline
$\Gamma_n=\{\alpha_n\}$ & $-2\varpi_n$ & $(n+1)/2$\\
\hline
\end{tabular}
\medskip

Assume that $\g$ is of type ${\rm D}_n$ and that $s+2\ell\le n-2$ :\medskip

\begin{tabular}{|c|c|c|}
\hline
$\langle ij\rangle$-orbit in $E(\pi')$& Weight & Degree\\
\hline
$\Gamma_u=\{\alpha_u,\,\alpha_{s-u}\}$,\;&$-2\varpi_s$&$s+1+2u$\\
$1\le u\le (s-1)/2$&&\\
\hline
$\Gamma_v=\{\alpha_v\}$,& $-2\varpi_v$&$v+1$\\
$v=s+2k,\,0\le k\le\ell$ &&\\
\hline
$\Gamma_v=\{\alpha_v\}$,\;& $-\varpi_{v-1}-\varpi_{v+1}$&$v+1$\\
$v=s+2k-1,\,1\le k\le\ell$ &&\\
\hline
$\Gamma_v=\{\alpha_v\}$,\; & $-2\varpi_{s+2\ell}$&$2v+1-s-2\ell$\\
$s+2\ell+1\le v\le n-2$&&\\
\hline
$\Gamma_{n-1}=\{\alpha_{n-1},\,\alpha_n\}$ & $-2\varpi_{s+2\ell}$ & $2n-s-2\ell-1$\\
\hline
\end{tabular}
\medskip

Assume that $\g$ is of type ${\rm D}_n$ and that $s+2\ell=n$ (hence $n$ is odd) :\medskip

\begin{tabular}{|c|c|c|}
\hline
$\langle ij\rangle$-orbit in $E(\pi')$& Weight & Degree\\
\hline
$\Gamma_u=\{\alpha_u,\,\alpha_{s-u}\}$,&$-2\varpi_s$&$s+1+2u$\\
$1\le u\le (s-1)/2$&&\\
\hline
$\Gamma_v=\{\alpha_v\}$,& $-2\varpi_v$&$v+1$\\
$v=s+2k,\,0\le k\le\ell-1$&&\\
\hline
$\Gamma_v=\{\alpha_v\}$,& $-\varpi_{v-1}-\varpi_{v+1}$&$v+1$\\
$v=s+2k-1,\,1\le k\le\ell-1$&&\\
\hline
$\Gamma_{n-1}=\{\alpha_{n-1},\,\alpha_n\}$ & $-\varpi_{n-2}-2\varpi_n$ & $n$\\
\hline
\end{tabular}
\medskip

Assume that $\g$ is of type ${\rm D}_n$ and that $s+2\ell=n-1$. Hence $n$ is even 
and  we assume that $\pi'=\pi\setminus\{\alpha_s,\,\alpha_{s+2},\ldots,\alpha_{n-3},\,\alpha_n\}$ :\medskip

\begin{tabular}{|c|c|c|}
\hline
$\langle ij\rangle$-orbit in $E(\pi')$& Weight & Degree\\
\hline
$\Gamma_u=\{\alpha_u,\,\alpha_{s-u}\}$,&$-2\varpi_s$&$s+1+2u$\\
$1\le u\le (s-1)/2$&&\\
\hline
$\Gamma_v=\{\alpha_v\}$,& $-2\varpi_v$&$v+1$\\
$v=s+2k,\,0\le k\le\ell-1$&&\\
\hline
$\Gamma_v=\{\alpha_v\}$,& $-\varpi_{v-1}-\varpi_{v+1}$&$v+1$\\
$v=s+2k-1,\,1\le k\le\ell-1$ &&\\
\hline
$\Gamma_{n-1}=\{\alpha_{n-1},\,\alpha_{n-2}\}$ & $-2(\varpi_{n-3}+\varpi_n)$ & $3n/2$\\
\hline
$\Gamma_n=\{\alpha_n\}$ & $-2\varpi_n$ & $n/2$\\
\hline
\end{tabular}
\subsection{}
\begin{Rq}\rm
Assume that the simple Lie algebra $\g$ is of type ${\rm B}_n$ or ${\rm D}_n$ and that $\pi'=\pi\setminus\{\alpha_s,\,\alpha_{s+4}\}$, with $s$ odd and consider the truncated parabolic subalgebra $\p_{\Lambda}$ associated  to $\pi'$.
In this case the lower and upper bounds for ${\ch}\,(Y(\p_{\Lambda}))$ in (\ref{ch}) of Sect. \ref{bounds}  do not coincide in general and then we do not know for the moment whether polynomiality of $Y(\p_{\Lambda})$ holds or not. However the adapted pair that we have constructed in subsection  \ref{APodd}  using the set $S=\beta^0_{\pi}\cup(-\beta^0_{\pi'})$ (at least for type ${\rm B}_n$) does no more work in this case. Indeed one may notice that for all $\beta\in\beta^0_{\pi}$,
and for all $\beta'\in\beta^0_{\pi'}$, one has $\beta(\alpha_{s+2}^\vee)=\beta'(\alpha_{s+2}^\vee)=0$ while $\alpha_{s+2}^\vee\in\h_{\Lambda}$. It follows that the restriction of $\beta^0_{\pi}\cup(-\beta^0_{\pi'})$ to $\h_{\Lambda}$ cannot give a basis for $\h_{\Lambda}^*$.
\end{Rq}

\section{Cases \ref{cas3} and \ref{cas4} for type B or D.}\label{FC}

Recall the notation of subsection \ref{Cases} and Sect. \ref{not}.
In this Section the  Lie algebra $\g$ is simple of type ${\rm B}_n$, $n\ge 4$, resp. ${\rm D}_n$, $n\ge 6$, and we consider the parabolic subalgebra $\p=\p_{s,\,1}=\p_{\pi'}^-$ of $\g$ associated to the subset $\pi'=\pi\setminus\{\alpha_s,\,\alpha_{s+2}\}$
of simple roots, with $s$ an {\it even} integer, $2\le s\le n-2$, resp. $2\le s\le n-4$. We are then in the cases \ref{cas3} and \ref{cas4} of subsection \ref{Mainres}.

The Levi subalgebra $\g'$  of $\p$ is isomorphic to the product $\s\mathfrak l_s\times\s\mathfrak l_2\times\s\mathfrak o_m$, with $m\in\mathbb N^*$, and  $m\ge 4$ if $\g$ is of type ${\rm D}_n$. More precisely if $\g$ is of type ${\rm B}_n$ one has that $m=2n-2s-3$, and when $\g$ is of type ${\rm D}_n$ one has that $m=2n-2s-4$. We adopt the convention that $\s\mathfrak o_1=\{0\}$, $\s\mathfrak o_3=\s\mathfrak l_2$, $\s\mathfrak o_4=\s\mathfrak l_2\times\s\mathfrak l_2$ and $\s\mathfrak o_6=\s\mathfrak l_4$.

In these cases the lower and upper bounds given by \eqref{ch} of Sect. \ref{bounds} do not coincide, hence we cannot conclude with this criterion that the algebra  $Sy(\p)=Y(\p_{\Lambda})$ is or not polynomial. However we will construct an adapted pair for the truncated parabolic subalgebra $\p_{\Lambda}$ associated to $\p$. We will then prove that the improved upper bound defined in Sect. \ref{IUB} is equal to the lower bound (namely that equality \eqref{eqIUB} of Sect. \ref{IUB} holds). This implies by Remark \ref{rqeqIUB} of subsection \ref{REG} that there is a Weierstrass section for coadjoint action of $\p_{\Lambda}$ and then that
the algebra of symmetric invariants $Y(\p_{\Lambda})=Sy(\p)$ is a polynomial algebra over $\Bbbk$ for which  the weights and degrees of homogeneous and $\h$-weight generators may also
be computed.

We will still use Proposition \ref{propAP} but here the set $S$ cannot be taken to contain $\beta^0_{\pi}\cup(-\beta^0_{\pi'})$ as in Sect. \ref{SC}. Indeed assume that $S$ contains the elements $\beta_1,\,\ldots,\,\beta_{s/2}$ of the Kostant cascade of $\g$. Then
the semisimple element $h$ of the adapted pair should verify both equalities $\varpi_s(h)=((\ep_1+\ep_2)+\ldots+(\ep_{s-1}+\ep_s))(h)=(-1)\times s/2$
and $\varpi_s(h)=0$ by definition of $\h_{\Lambda}$ (see Sect. \ref{not}) and since $h\in\h_{\Lambda}$ and $-2\varpi_s\in\Lambda(\p)$ by \eqref{ch} of Sect. \ref{bounds}. Hence we obtain a contradiction. Also  for each $\gamma\in S$, a more complicated Heisenberg set $\Gamma_{\gamma}$ with centre $\gamma$ than the set $H_{\gamma}$ used in previous section will be taken in general. We will also take $T^*\neq\emptyset$.

\subsection{Condition {\it (i)} of Proposition \ref{propAP}.}\label{APP}

For type  ${\rm B}_n$, we set 
$$\begin{array}{cl}
S^+=&\bigl\{\ep_s,\,\beta_i=\ep_{2i-1}+\ep_{2i},\,\ep_{s-1}+\ep_{s+1},\,\ep_{2j}+\ep_{2j+1};\\
 &1\le i\le s/2-1,\,s/2+1\le j\le [(n-1)/2]\bigr\}\end{array}$$
 and 
 $$\begin{array}{cl}
S^-=&\bigl\{\ep_{s-i}-\ep_i,\,-\beta''_{j-(s+2)/2}=-\ep_{2j-1}-\ep_{2j};\\
 &1\le i\le s/2-1,\,s/2+2\le j\le [n/2]\bigr\}\end{array}$$
 
 For type  ${\rm D}_n$, we set 
$$\begin{array}{cl}
S^+=&\bigl\{\ep_s-\ep_n,\,\ep_s+\ep_n,\,\beta_i=\ep_{2i-1}+\ep_{2i},\,\ep_{s-1}+\ep_{s+1},\,\ep_{2j}+\ep_{2j+1};\\
 &1\le i\le s/2-1,\,s/2+1\le j\le [(n-2)/2]\bigr\}\end{array}$$
 and 
 $$\begin{array}{cl}
S^-=&\bigl\{\ep_{s-i}-\ep_i,\,-\beta''_{j-(s+2)/2}=-\ep_{2j-1}-\ep_{2j};\\
& 1\le i\le s/2-1,\,s/2+2\le j\le [(n-1)/2]\bigr\}\end{array}$$
 
 Remark that the above sets $S^\pm$ contain the same elements as those defined in \cite{FL1} or in \cite{F1} for maximal parabolic
 subalgebras, except for one of them which is missing, namely the element $-\ep_{s+1}-\ep_{s+2}$, since it does no more belong to $\Delta^-_{\pi'}$.
 
 As we already noticed in Sect. \ref{not} for type ${\rm B}_n$, and also for type ${\rm D}_n$ (since $s\le n-4$) we have that $\h_{\Lambda}=\h'$. 
 As in \cite[Lem. 7.1]{FL1}, we prove the following lemma.
 
 \begin{lm}
 
 Set $S=S^+\sqcup S^-$ as above. Then $S_{\mid\h_{\Lambda}}$ is a basis for $\h_{\Lambda}^*$.

 \end{lm}

\begin{proof}

The proof is quite similar to that of \cite[Lem. 7.1]{FL1}. We give it below for the reader's convenience. First observe that $\lvert S\rvert=n-2$.
The elements of $S$ will be denoted by $s_i$, with $1\le i\le n-2$.
When $\g$ is of type ${\rm B}_n$, we set $s_{n-3}=\ep_s$ and $s_{n-2}=\ep_{n-1}+\ep_n$ if $n$ is odd, resp. $s_{n-2}=-\ep_{n-1}-\ep_n$ if $n$ is even.
When $\g$ is of type ${\rm D}_n$, we set $s_{n-3}=\ep_s-\ep_n$ and $s_{n-2}=\ep_s+\ep_n$.
Then we set $s'_i=s_i$ for all $1\le i\le n-2$ if $\g$ is of type ${\rm B}_n$. If $\g$ is of type ${\rm D}_n$, we set $s'_i=s_i$ for all $1\le i\le n-4$, $s'_{n-3}=\ep_s$ and $s'_{n-2}=\ep_n$.
It suffices to verify that, if $\{h_j\}_{1\le j\le n-2}$ is a basis of $\h_{\Lambda}=\h'$, then $\det(s'_i(h_j))_{1\le i,\,j\le n-2}\neq 0$.

To prove this, we order the basis $\{h_j\}_{1\le j\le n-2}$ of $\h_{\Lambda}$ as 
$$\begin{array}{cc}\bigl\{\alpha^\vee_{2i},\,1\le i\le s/2-1,\\
\alpha_{s-1}^\vee,\,\alpha^\vee_{2j-1},\,\alpha^\vee_{s-2j-1},\,1\le j\le [s/4],\\
\alpha^\vee_k,\,\,s+1\le k\le n,\,k\neq s+2\bigr\}\end{array}$$
without repetitions. The elements $s'_i$, $1\le i\le n-2$, are ordered as
$$\begin{array}{cc}\bigl\{\ep_{2i-1}+\ep_{2i},\, 1\le i\le s/2-1,\\
\ep_s,\,\ep_{s-j}-\ep_j,\,1\le j\le s/2-1,\,\ep_{s-1}+\ep_{s+1},\\
\ep_{2k}+\ep_{2k+1},\,-\ep_{2k+1}-\ep_{2k+2},\,s/2+1\le k\le [(n-3)/2],\\
(-1)^n(\ep_{n-2}+\ep_{n-1}),\,s'_{n-2}\bigr\}\end{array}$$
without repetitions.

Then one verifies that $(s'_i(h_j))_{1\le i,\,j\le n-2}=\begin{pmatrix}A&0&0&0\\
*&B&0&0\\
*&*&C&0\\
*&*&*&D\end{pmatrix}$
with $A$, resp. $B$, a $(s/2-1)\times(s/2-1)$, resp. $s/2\times s/2$, lower triangular matrix with $1$, resp. $-1$, on its diagonal.

Moreover $C=(1)$ and $D=\begin{pmatrix} D'&0\\
*&D''\end{pmatrix}$ with $D'$ a $(n-s-4)\times(n-s-4)$ lower triangular matrix with alternating $1$ and $-1$ on its diagonal, and $D''$ an invertible $2\times 2$ matrix.
\end{proof}

\subsection{Conditions {\it (ii)}, {\it (iii)} and {\it (vi)} of Proposition \ref{propAP}.}\label{T}

To each $\gamma\in S$, we need now to associate a Heisenberg set $\Gamma_{\gamma}$ with centre $\gamma$.

Recall that $\beta_i:=\ep_{2i-1}+\ep_{2i}$, for all $1\le i\le s/2-1$, is a positive root which belongs to the Kostant cascade of $\g$.
We then set, for all $1\le i\le s/2-1$, $\Gamma_{\beta_i}=H_{\beta_i}$ where $H_{\beta_i}$ is the largest Heisenberg set with centre $\beta_i$ included in $\Delta^+$ as defined in \ref{HS}.

For $\g$ of type ${\rm B}_n$ we set $$\begin{array}{cl}\Gamma_{\ep_{s-1}+\ep_{s+1}}=&\{\ep_{s-1}+\ep_{s+1},\,\ep_{s-1}\pm\ep_i,\,\ep_{s+1}\mp\ep_i\,\;;\; s+2\le i\le n,\\
&\ep_{s-1},\,\ep_{s+1},\,\ep_{s-1}-\ep_s,\,\ep_s+\ep_{s+1}\}.\end{array}$$  For $\g$ of type ${\rm D}_n$, $\Gamma_{\ep_{s-1}+\ep_{s+1}}$ is taken to be the same set as above but without $\ep_{s-1}$ and $\ep_{s+1}$ which are not roots in this type.

For $s/2+1\le j\le [(n-1)/2]$ for type ${\rm B}_n$, resp. $s/2+1\le j\le [(n-2)/2]$ for type ${\rm D}_n$, we set $$\begin{array}{cl}\Gamma_{\ep_{2j}+\ep_{2j+1}}=&\{\ep_{2j}+\ep_{2j+1},\,\ep_{2j},\,\ep_{2j+1},\,\\
&\ep_{2j}\pm\ep_k,\,\ep_{2j+1}\mp\ep_k\,;\,\,2j+2\le k\le n\},\end{array}$$ resp. the same set as above but without $\ep_{2j}$ and $\ep_{2j+1}$.

For all $1\le i\le s/2-1$, we set $$\Gamma_{\ep_{s-i}-\ep_i}=\{\ep_{s-i}-\ep_i,\,\ep_{s-i}-\ep_j,\,\ep_j-\ep_i\,;\, i+1\le j\le s-i-1\}.$$

For all $s/2+2\le j\le [n/2]$ for type ${\rm B}_n$, resp. $s/2+2\le j\le [(n-1)/2]$ for type ${\rm D}_n$, we set
$\Gamma_{-\ep_{2j-1}-\ep_{2j}}=-H_{\ep_{2j-1}+\ep_{2j}}$ where $H_{\ep_{2j-1}+\ep_{2j}}$ is the largest Heisenberg set with centre $\beta''_{j-(s+2)/2}:=\ep_{2j-1}+\ep_{2j}\in\beta_{\pi'}$ included in $\Delta^+_{\pi'}$, as defined in \ref{HS}.

Finally for $\g$ of type ${\rm B}_n$, we set $\Gamma_{\ep_s}=\{\ep_s\}$
and for $\g$ of type ${\rm D}_n$, we set $\Gamma_{\ep_s+\ep_n}=\{\ep_s+\ep_n\}$ and $\Gamma_{\ep_s-\ep_n}=\{\ep_s-\ep_n\}$.

By construction all the above sets $\Gamma_{\gamma}$, $\gamma\in S$, are Heisenberg sets with centre $\gamma$ and they are pairwise disjoint.

Moreover the above sets $\Gamma_{\gamma}$, $\gamma\in S$, are chosen to be the same as in \cite{F1} (for type ${\rm B}_n$), except for $\Gamma_{\ep_{s-1}+\ep_{s+1}}$ where here the roots $\ep_{s-1}-\ep_s$ and $\ep_s+\ep_{s+1}$ are added.
However  the proofs of \cite[Lem. 14 and 15]{F1}, themselves based on Lemma \ref{HS} $ii)$ and $iii)$, can still be applied to show that conditions $(ii)$ and $(iii)$ of Proposition \ref{propAP} are satisfied. 

Now for the set $T$ we take
$$\begin{array}{cl}
T=&\{\ep_{s-1}+\ep_s,\,\ep_{s-1}-\ep_{s+1},\,\ep_s+\ep_{s+2},\\
&\ep_{2i-1}-\ep_{2i},\,\ep_{s+2j}-\ep_{s+2j+1},\,-\ep_{s+2k-1}+\ep_{s+2k}\;;\\
&1\le i\le s/2-1,\,1\le j\le [(n-s-1)/2],\,1\le k\le [(n-s)/2]\}.\end{array}$$

One checks that $T\subset\Delta^+\sqcup\Delta^-_{\pi'}$ and that $T$ is disjoint from $\Gamma=\bigsqcup_{\gamma\in S}\Gamma_{\gamma}$.
Note also that this set $T$ has the same elements as the set $T$ in \cite{F1},  except that $\alpha_{s-1}=\ep_{s-1}-\ep_s$ now belongs to $\Gamma_{\ep_{s-1}+\ep_{s+1}}$, and  is replaced  by $\ep_s+\ep_{s+2}$. We check below that condition $(vi)$ of Proposition \ref{propAP} is satisfied.

\begin{lm}

We have that $\lvert T\rvert=\ind\p_{\Lambda}.$

\end{lm}

\begin{proof}

One checks that $\lvert T\rvert=n-s/2+1$.
Recall  \eqref{index} of Sect. \ref{bounds}, that $\ind\p_{\Lambda}=\lvert E(\pi')\rvert$ where $E(\pi')$ is the set of $\langle ij\rangle$-orbits in $\pi$.

Denote by $\pi'_1$, $\pi'_2$, $\pi'_3$ the three irreducible components of $\pi'$. Then $\pi'_1$ is of type ${\rm A}_{s-1}$, $\pi'_2$ is of type ${\rm A}_1$ and $\pi'_3$ is of type ${\rm B}_{n-s-2}$, resp. ${\rm D}_{n-s-2}$ if $\g$ is of type ${\rm B}_n$, resp. ${\rm D}_n$.

Then $i_{\mid\pi'_1}$ exchanges $\alpha_t$ and $\alpha_{s-t}$ for all $1\le t\le s/2-1$ and fixes $\alpha_{s/2}$, $i_{\mid\pi'_2}=Id_{\pi'_2}$ and $(ij)_{\mid\pi'_3}=Id_{\pi'_3}$ since $n$ and $n-s-2$ are of the same parity (and $n-s-2\ge 2$ if $\g$ is of type ${\rm D}_n$). Moreover for all $\alpha\in\pi\setminus\pi'$, $i(\alpha)=j(\alpha)=\alpha$.
Then the set $E(\pi')$ of $\langle ij\rangle$-orbits in $\pi$ is
$$E(\pi')=\{\{\alpha_t,\,\alpha_{s-t}\},\,\{\alpha_{s/2}\},\,\{\alpha_u\}\; ; \; 1\le t\le s/2-1,\,s\le u\le n\}.$$
They are $n-s/2+1$ in number. Hence the lemma.
\end{proof}

\subsection{Condition {\it (iv)} and {\it (v)} of Proposition \ref{propAP}.}\label{condition(v)caseeven}

If $\g$ is of type ${\rm B}_n$, we take :
$$T^*=\{\ep_s-\ep_i,\,\ep_s+\ep_j,\,(-1)^n\ep_n\;;\; 1\le i\le n,\; i\neq s,\,s+3\le j\le n\}.$$ 
If $\g$ is of type ${\rm D}_n$, we take :
 $$T^*=\{\ep_s-\ep_i,\,\ep_s+\ep_j,\,(-1)^{n-1}\alpha_n\;;\; 1\le i\le n-1,\; i\neq s,\,s+3\le j\le n-1\}.$$ 
In type ${\rm B}_n$, note that this set $T^*$ is  the same as $T^*$ in \cite{F1}, except that two elements here are missing :
$\ep_s+\ep_{s+1}$ which now belongs to $\Gamma_{\ep_{s-1}+\ep_{s+1}}$ and $\ep_s+\ep_{s+2}$ which now belongs to $T$.
By construction $T^*$ is disjoint from $\Gamma\sqcup T$.

Denote by $\Delta_{\tilde\pi'}^-$ the set of negative roots in the case when $\tilde\pi'=\pi\setminus\{\alpha_s\}$ (that is, the set of negative roots for the parabolic subalgebra $\p_{\tilde\pi'}$ as considered in \cite{F1}) 
and recall that we denote by $\Delta_{\pi'}^-$ the set of negative roots for $\p_{\pi'}$ in our present case when $\pi'=\pi\setminus\{\alpha_s,\,\alpha_{s+2}\}$. Then one has that $\Delta_{\tilde\pi'}^-=\Delta_{\pi'}^-\sqcup -H_{\ep_{s+1}+\ep_{s+2}}$, where $H_{\ep_{s+1}+\ep_{s+2}}$ is the largest Heisenberg set with centre $\ep_{s+1}+\ep_{s+2}$ which is included in $\Delta^+_{\tilde\pi'}$ as defined in \ref{HS}.
By a similar proof as in \cite[Lem. 13]{F1} and using Lemma \ref{HS} $i)$, one checks that $\Delta^+\sqcup\Delta^-_{\pi'}=\Gamma\sqcup T\sqcup T^*$. Hence condition $(iv)$ of Proposition \ref{propAP} is satisfied. It remains to verify condition $(v)$ of Proposition \ref{propAP}.
The proofs of \cite[Lem. 16, 17, 18, 19]{F1} can still be applied in type ${\rm B}_n$. In type ${\rm D}_n$ they have to be adapted. For completeness, we give a proof below. Set $y=\sum_{\gamma\in S}x_{\gamma}$.

\begin{lm}

Let $\gamma\in T^*$. Then $\g_{\gamma}\subset ad\,\p_{\Lambda}(y)+\g_T$.

\end{lm}

\begin{proof}
Recall (Sect. \ref{not}) that $\p_{\Lambda}=\n^-\oplus\h'\oplus\n^+_{\pi'}$ and that we have chosen, for each $\alpha\in\Delta$, a nonzero root vector $x_{\alpha}\in\g_{\alpha}$.
Given $\gamma,\,\delta\in\Delta^\pm$ such that $\gamma+\delta\in\Delta^\pm$, one has that $ad\,x_{\gamma}(x_{\delta})=[x_{\gamma},\,x_{\delta}]\in\g_{\gamma+\delta}\setminus\{0\}$ by say \cite[1.10.7]{D}, then it is a nonzero multiple of $x_{\gamma+\delta}$.

Assume that $\g$ is of type ${\rm B}_n$ and rescale if necessary the nonzero root vectors $x_{\gamma}$, $\gamma\in\Delta\setminus S$.\par\noindent
Let $s+3\le j\le n-1$ and $j$ odd. Then 
one has that
$$\begin{cases} ad\,x_{\ep_j}(y)=x_{\ep_s+\ep_j}+x_{-\ep_{j+1}}\\
ad\,x_{-\ep_s-\ep_{j+1}}(y)=x_{-\ep_{j+1}}.\\
\end{cases}$$ Hence 
$x_{\ep_s+\ep_j}=ad\,(x_{\ep_j}-x_{-\ep_s-\ep_{j+1}})(y)\in ad\,\p_{\Lambda}(y)$.
 If $j=n$ is odd, then $x_{\ep_s+\ep_n}=ad\,x_{\ep_n}(y)\in ad\,\p_{\Lambda}(y)$.
Let $s+4\le j\le n$ and $j$ even. Then $x_{\ep_s+\ep_j}=ad\,(x_{\ep_j}-x_{-\ep_s-\ep_{j-1}})(y)\in ad\,\p_{\Lambda}(y).$\par\noindent
Let $1\le i\le s-3$ and $i$ odd, or $s+2\le i\le n-1$ and $i$ even.
Then
$$\begin{cases}x_{\ep_s-\ep_i}=ad\,(x_{-\ep_i}-x_{\ep_{i+1}-\ep_s}+x_{-\ep_{s-i-2}-\ep_s})(y)&{\rm if}\;i\le s/2-2\\
x_{\ep_s-\ep_i}=ad\,(x_{-\ep_i}-x_{\ep_{i+1}-\ep_s})(y)&{\rm otherwise}.\\
\end{cases}$$
Hence
$x_{\ep_s-\ep_i}\in ad\,\p_{\Lambda}(y).$
Let $2\le i\le s-2$ and $i$ even, or $s+3\le i\le n$ and $i$ odd. Then
$$\begin{cases} x_{\ep_s-\ep_i}=ad\,(x_{-\ep_i}-x_{\ep_{i-1}-\ep_s}+x_{-\ep_{s-i+2}-\ep_s})(y)&{\rm if}\;4\le i\le s/2\\
x_{\ep_s-\ep_i}=ad\,(x_{-\ep_2}-x_{\ep_{1}-\ep_s}+x_{-\ep_{s}-\ep_{s+1}})(y)&{\rm if}\;i=2\\
x_{\ep_s-\ep_i}=ad\,(x_{-\ep_i}-x_{\ep_{i-1}-\ep_s})(y)&{\rm otherwise}.\\
\end{cases}$$
Hence
$x_{\ep_s-\ep_i}\in ad\,\p_{\Lambda}(y).$
If $i=s-1$ then
$$x_{\ep_s-\ep_{s-1}}=ad\,(x_{-\ep_{s-1}}-x_{\ep_{s+1}-\ep_s})(y)\in ad\,\p_{\Lambda}(y).$$
If $i=s+1$ then
$$x_{\ep_s-\ep_{s+1}}=ad\,(x_{-\ep_{s+1}}-x_{\ep_{s-1}-\ep_s})(y)\in ad\,\p_{\Lambda}(y).$$
If $i=n$ is even, then
$$x_{\ep_s-\ep_n}=ad\,x_{-\ep_n}(y)\in ad\,\p_{\Lambda}(y).$$
Finally, if $n$ is odd, then $x_{-\ep_n}=ad\,x_{-\ep_s-\ep_n}(y)\in ad\,\p_{\Lambda}(y)$ and if $n$ is even, then
$x_{\ep_n}=ad\,x_{-\ep_s+\ep_n}(y)\in ad\,\p_{\Lambda}(y)$.
Hence the lemma for $\g$ of type ${\rm B}_n$.\smallskip

Assume now that $\g$ is of type ${\rm D}_n$.\par\noindent
Let $s+3\le j\le n-2$ and $j$ odd. 
One may apply Lemma \ref{condition(v)} with $\gamma_1=\ep_s+\ep_n\in S$, $\gamma'_1=\ep_j-\ep_n\not\in S$, $\gamma_2=-\ep_j-\ep_{j+1}\in S$,
$\gamma'_2=\ep_j+\ep_n\not\in S$, $\gamma_3=\ep_s-\ep_n\in S$, $\gamma'_3=-\ep_s-\ep_{j+1}\not\in S$.
Then up to rescaling the nonzero root vectors $x_{\ep_j-\ep_n}$, $x_{-\ep_s-\ep_{j+1}}$, $x_{\ep_j+\ep_n}$ in $\p_{\Lambda}$ and 
$x_{-\ep_{j+1}-\ep_n}$, $x_{\ep_n-\ep_{j+1}}$, $x_{\ep_s+\ep_j}$ in $\p_{\Lambda}^*$, one has that, by Lemma \ref{condition(v)}
$$\begin{cases}[x_{\ep_j-\ep_n},\,x_{\ep_s+\ep_n}]=[x_{\ep_j+\ep_n},\,x_{\ep_s-\ep_n}]=x_{\ep_s+\ep_j}\\
[x_{\ep_j-\ep_n},\,x_{-\ep_j-\ep_{j+1}}]=[x_{-\ep_s-\ep_{j+1}},\,x_{\ep_s-\ep_n}]=x_{-\ep_{j+1}-\ep_n}\\
[x_{-\ep_s-\ep_{j+1}},\,x_{\ep_s+\ep_n}]=[x_{\ep_j+\ep_n},\,x_{-\ep_j-\ep_{j+1}}]=x_{\ep_n-\ep_{j+1}}.\end{cases}$$

It follows that 
$$\begin{cases} x_{\ep_s+\ep_j}+x_{-\ep_{j+1}-\ep_n}=ad\,x_{\ep_j-\ep_n}(y)\in ad\,\p_{\Lambda}(y)\\
x_{-\ep_{j+1}-\ep_n}+x_{\ep_n-\ep_{j+1}}=ad\,x_{-\ep_s-\ep_{j+1}}(y)\in ad\,\p_{\Lambda}(y)\\
x_{\ep_n-\ep_{j+1}}+x_{\ep_s+\ep_j}=ad\,x_{\ep_j+\ep_n}(y)\in ad\,\p_{\Lambda}(y)\\
\end{cases}$$
Hence $x_{\ep_s+\ep_j}\in ad\,\p_{\Lambda}(y)$.\par\noindent
Now if $j=n-1$ is odd, then $x_{\ep_s+\ep_{n-1}}=ad\,x_{\ep_{n-1}-\ep_n}(y)\in ad\,\p_{\Lambda}(y)$.
Let $s+4\le j\le n-1$ and $j$ even. Then similarly as above (by Lemma and Prop. \ref{condition(v)}), one has that
$$\begin{cases}x_{\ep_s+\ep_j}+x_{-\ep_{j-1}-\ep_n}=ad\,x_{\ep_j-\ep_n}(y)\in ad\,\p_{\Lambda}(y)\\
x_{-\ep_{j-1}-\ep_n}+x_{\ep_n-\ep_{j-1}}=ad\,x_{-\ep_s-\ep_{j-1}}(y)\in ad\,\p_{\Lambda}(y)\\
x_{\ep_n-\ep_{j-1}}+x_{\ep_s+\ep_j}=ad\,x_{\ep_j+\ep_n}(y)\in ad\,\p_{\Lambda}(y)\\
\end{cases}$$
Hence $x_{\ep_s+\ep_j}\in ad\,\p_{\Lambda}(y)$.\par\noindent
Let $1\le i\le s-3$ and $i$ odd, or $s+2\le i\le n-2$ and $i$ even.
Again, up to rescaling some nonzero root vectors, Lemma and Prop.  \ref{condition(v)} imply  that
$$\begin{cases}x_{\ep_s-\ep_i}+x_{\ep_{i+1}-\ep_n}=ad\,x_{-\ep_i-\ep_n}(y)\in ad\,\p_{\Lambda}(y)\\
x_{\ep_{i+1}+\ep_n}+x_{\ep_s-\ep_i}=ad\,x_{\ep_n-\ep_i}(y)\in ad\,\p_{\Lambda}(y)\\
x_{\ep_{i+1}-\ep_n}+x_{\ep_{i+1}+\ep_n}\in ad\,\p_{\Lambda}(y)\\
\end{cases}$$ since
$$\begin{cases}
x_{\ep_{i+1}-\ep_n}+x_{\ep_{i+1}+\ep_n}=ad\,(x_{\ep_{i+1}-\ep_s}-x_{-\ep_{s-i-2}-\ep_s})(y)&{\rm if}\; i\le s/2-2\\
x_{\ep_{i+1}-\ep_n}+x_{\ep_{i+1}+\ep_n}=ad\,x_{\ep_{i+1}-\ep_s}(y)&{\rm otherwise}.\\
\end{cases}$$
Hence $x_{\ep_s-\ep_i}\in ad\,\p_{\Lambda}(y)$.
For $i=n-1$ even, one has that
$x_{\ep_s-\ep_{n-1}}=ad\,x_{\ep_n-\ep_{n-1}}(y)\in ad\,\p_{\Lambda}(y)$.
For $2\le i\le s-2$ and $i$ even, or  $s+3\le i\le n-1$ and $i$ odd, a similar computation shows that one also has that $x_{\ep_s-\ep_i}\in ad\,\p_{\Lambda}(y)$ in these cases.
Let $i=s-1$. Then Lemma \ref{condition(v)} implies that
$$\begin{cases} x_{\ep_s-\ep_{s-1}}+x_{\ep_{s+1}-\ep_n}=ad\,x_{-\ep_{s-1}-\ep_n}(y)\\
x_{\ep_{s+1}-\ep_n}+x_{\ep_{s+1}+\ep_n}=ad\,x_{\ep_{s+1}-\ep_s}(y)\\
x_{\ep_{s+1}+\ep_n}+x_{\ep_s-\ep_{s-1}}=ad\,x_{\ep_n-\ep_{s-1}}(y).\\
\end{cases}$$
Hence $x_{\ep_s-\ep_{s-1}}\in ad\,\p_{\Lambda}(y)$.
A similar computation shows that $x_{\ep_s-\ep_{s+1}}\in ad\,\p_{\Lambda}(y)$.\par\noindent
Finally assume that $n$ is even. Then $ad\,x_{-\ep_s-\ep_{n-1}}(y)=x_{-\alpha_n}+x_{-\alpha_{n-1}}\in ad\,\p_{\Lambda}(y)$
and $x_{-\alpha_{n-1}}\in\g_T$. Thus $x_{-\alpha_n}\in ad\,\p_{\Lambda}(y)+\g_T$.
If  $n$ is odd, then $ad\,x_{\ep_{n-1}-\ep_s}(y)=x_{\alpha_n}+x_{\alpha_{n-1}}\in ad\,\p_{\Lambda}(y)$
and $x_{\alpha_{n-1}}\in\g_T$. Thus $x_{\alpha_n}\in ad\,\p_{\Lambda}(y)+\g_T$.
The proof is complete.
\end{proof}

\subsection{}\label{CSC}

All conditions of Proposition \ref{propAP} are satisfied. Thus one has the following corollary.

\begin{cor} Keep the above notation.
One has that
 $$ad\,\p_{\Lambda}(y)\oplus\g_T=\p_{\Lambda}^*$$
with $\dim(\g_T)=\ind(\p_{\Lambda})$
that is, $y$ is regular in $\p_{\Lambda}^*$.
Moreover, by Lemma \ref{APP}, there exists a uniquely defined element $h\in\h_{\Lambda}$ such that $\gamma(h)=-1$ for all $\gamma\in S$. Then $(h,\,y)$ is an adapted pair for $\p_{\Lambda}$.
\end{cor}

\subsection{The semisimple element of the adapted pair.}\label{ssel}

By direct computation, one may give the expansion of the semisimple element $h$ of the adapted pair for $\p_{\Lambda}$ obtained in Corollary \ref{CSC}.

\begin{lm}

In terms of the elements $\ep_i$, $1\le i\le n$, the semisimple element $h\in\h_{\Lambda}$ of the adapted pair $(h,\,y)$ obtained
in Corollary \ref{CSC} has the following expansion. Set $u=0$ in type ${\rm D}_n$, resp. $u=1$ in type ${\rm B}_n$.
$$\begin{array}{ll}
h=&\sum_{k=1}^{[s/4]}(s/2+2k-1)\ep_{2k-1}+\sum_{k=[s/4]+1}^{s/2-1}(3s/2-2k)\ep_{2k-1}\\
&-\sum_{k=1}^{[s/4]} (s/2+2k)\ep_{2k}
-\sum_{k=[s/4]+1}^{s/2-1}(3s/2+1-2k)\ep_{2k}\\
&+(s/2)\ep_{s-1}-\ep_s-(s/2+1)\ep_{s+1}\\
&+\sum_{k=1}^{[(n-s-1+u)/2]}(2k-1+s/2)\ep_{s+2k}\\
&-\sum_{k=2}^{[(n-s+u)/2]}(2k-2+s/2)\ep_{s+2k-1}.\\
\end{array}$$
In terms of the coroots $\alpha_k^\vee$, $1\le k\le n$, $k\not\in\{ s,\,s+2\}$, the element $h$ has the 
following expansion.
Set 
$$\begin{array}{ll}
H=&-\sum_{k=1}^{s/2-1}k\alpha_{2k}^\vee+\sum_{k=1}^{[s/4]}(s/2+k)\alpha_{2k-1}^\vee\\
&+\sum_{k=[s/4]+1}^{s/2}(3s/2+1-3k)\alpha_{2k-1}^\vee\\
&+\sum_{k=s/2+1}^{[(n-2+u)/2]}(k-1-s/2)\alpha_{2k}^\vee\\
&-\sum_{k=s/2+1}^{[(n-1+u)/2]}k\alpha_{2k-1}^\vee.\\
\end{array}$$
Then 
$$\begin{cases}
h=H+(n-s-2)/4\alpha_n^\vee&\;{\rm if }\;\g\;{\rm is\; of\; type}\; {\rm B}_n\;{\rm with}\; n\;{\rm even}\\
h=H-(n+1)/4\alpha_n^\vee&\;{\rm if }\;\g\;{\rm is\; of\; type}\; {\rm B}_n\;{\rm with}\; n\;{\rm odd}\\
h=H-n/4(\alpha_{n-1}^\vee+\alpha_n^\vee)&\;{\rm if }\;\g\;{\rm is\; of\; type}\; {\rm D}_n\;{\rm with}\; n\;{\rm even}\\
h=H+(n-s-3)/4(\alpha_{n-1}^\vee+\alpha_n^\vee)&\;{\rm if }\;\g\;{\rm is\; of\; type}\; {\rm D}_n\;{\rm with}\; n\;{\rm odd}.\\
\end{cases}$$

\end{lm}

\subsection{Computation of the improved upper bound and the lower bound.}\label{comp}

Recall the notation of Sect. \ref{bounds} and \ref{IUB}. One obtains the following Lemma.

\begin{lm}

If $\g$ is of type ${\rm B}_n$, resp. ${\rm D}_n$, then

 $$\prod_{\Gamma\in E(\pi')}(1-e^{\delta_{\Gamma}})^{-1}=\prod_{\gamma\in T}(1-e^{-(\gamma+s(\gamma))})^{-1}$$
More precisely one has the following.
 \begin{enumerate}
\item[{ (i)}] If $\g$ is of type ${\rm B}_n$ and $s+2<n$, then
$$\begin{array}{cl}{\ch}\,(Y(\p_{\Lambda}))=&(1-e^{-2\varpi_s})^{-s/2}(1-e^{-\varpi_s})^{-1}(1-e^{-2\varpi_{s+2}})^{-(n-s-2)}\\
 &(1-e^{-\varpi_{s+2}})^{-1}(1-e^{-(\varpi_s+\varpi_{s+2})})^{-1}.\\
 \end{array}$$
\item[{ (ii)}] If $\g$ is of type ${\rm B}_n$ and $n=s+2$, then
$$\begin{array}{cl}{\ch}\,(Y(\p_{\Lambda}))=&(1-e^{-2\varpi_s})^{-s/2}(1-e^{-\varpi_s})^{-1}\\
 &(1-e^{-2\varpi_{s+2}})^{-1}(1-e^{-(\varpi_s+2\varpi_{s+2})})^{-1}.\end{array}$$
\item[{ (iii)}]  If $\g$ is of type ${\rm D}_n$, then
 $$\begin{array}{cl}{\ch}\,(Y(\p_{\Lambda}))=&(1-e^{-2\varpi_s})^{-s/2}(1-e^{-\varpi_s})^{-1}(1-e^{-2\varpi_{s+2}})^{-(n-s-3)}\\
 &(1-e^{-\varpi_{s+2}})^{-2}(1-e^{-(\varpi_s+\varpi_{s+2})})^{-1}.\\
 \end{array}$$
 \end{enumerate}
\end{lm}

\begin{proof}

Recall the set $E(\pi')$ given in the proof of Lemma \ref{T} and set for all $1\le t\le s/2-1$,
$\Gamma_t=\{\alpha_t,\,\alpha_{s-t}\}$, $\Gamma_{s/2}=\{\alpha_{s/2}\}$ and $\Gamma_u=\{\alpha_u\}$ for all
$s\le u\le n$.
Observe that  $j(\Gamma)=\Gamma$ (and then $i(\Gamma\cap\pi')=j(\Gamma)\cap\pi'=\Gamma\cap\pi'$) for all $\Gamma\in E(\pi')$, except in type ${\rm D}_n$, with $n$ odd, for $\Gamma=\{\alpha_{n-1}\}$ or $\Gamma=\{\alpha_n\}$. Recall for all $\Gamma\in E(\pi')$ the weight $\delta_{\Gamma}$ defined in \eqref{poids} of Sect. \ref{bounds}. 
One checks that : 
$$\forall\,1\le t\le s/2-1,\,\delta_{\Gamma_t}=2(\varpi'_t-\varpi_t+\varpi'_{s-t}-\varpi_{s-t})=-2\varpi_s.$$
Moreover $\delta_{\Gamma_{s/2}}=2(\varpi'_{s/2}-\varpi_{s/2})=-\varpi_s,\,\delta_{\Gamma_s}=-2\varpi_s$
 and $\delta_{\Gamma_{s+2}}=-2\varpi_{s+2}.$
Finally $\delta_{\Gamma_{s+1}}=2(\varpi'_{s+1}-\varpi_{s+1})=-(\varpi_s+\varpi_{s+2}),$ for type ${\rm D}_n$
and for type ${\rm B}_n$ if $s+2<n$.
For type ${\rm B}_n$ with $s+2=n$ one checks that 
$\delta_{\Gamma_{s+1}}=-(\varpi_s+2\varpi_{s+2})$.\par\noindent
If $s+3\le u\le n-1$ for type ${\rm B}_n$, resp. $s+3\le u\le n-2$ for type ${\rm D}_n$, then one checks that
$\delta_{\Gamma_{u}}=-2\varpi_{s+2}$.
If $u=n$ for $\g$ of type ${\rm B}_n$  (and $s+2<n$), resp. $u=n-1$ or $u=n$ for $\g$ of type ${\rm D}_n$, then
one checks that $\delta_{\Gamma_{u}}=-\varpi_{s+2}$.
Thus $\prod_{\Gamma\in E(\pi')}(1-e^{\delta_{\Gamma}})^{-1}$ is equal to the right hand side of $(i)$, $(ii)$ or $(iii)$.\par\noindent
It remains to check that $\prod_{\Gamma\in E(\pi')}(1-e^{\delta_{\Gamma}})^{-1}=\prod_{\gamma\in T}(1-e^{-(\gamma+s(\gamma))})^{-1}$.
Recall the set $T$ given before Lemma \ref{T}.\par\noindent
For $\gamma=\ep_{s-1}+\ep_s$, one checks that $$s(\gamma)=(\ep_1+\ep_2)+\ldots+(\ep_{s-3}+\ep_{s-2})$$
so that $\gamma+s(\gamma)=\varpi_s$.\par\noindent
For $\gamma=\ep_{s-1}-\ep_{s+1}$, one checks that, if $\g$ is of type ${\rm B}_n$, $$s(\gamma)=2((\ep_1+\ep_2)+\ldots+(\ep_{s-3}+\ep_{s-2}))+(\ep_{s-1}+\ep_{s+1})+2\ep_s$$ and  if $\g$ of type ${\rm D}_n$, then
$$s(\gamma)=2((\ep_1+\ep_2)+\ldots+(\ep_{s-3}+\ep_{s-2}))+(\ep_{s-1}+\ep_{s+1})+(\ep_s+\ep_n)+(\ep_s-\ep_n)$$ and for both types that $\gamma+s(\gamma)=2\varpi_s$.\par\noindent
For $\gamma=\ep_s+\ep_{s+2}$, one checks that $$s(\gamma)=(\ep_1+\ep_2)+\ldots+(\ep_{s-3}+\ep_{s-2})+(\ep_{s-1}+\ep_{s+1})$$ so that $\gamma+s(\gamma)=\varpi_{s+2}$ for $\g$ of type ${\rm D}_n$ or $\g$ of type ${\rm B}_n$ with $s+2<n$,
and that $\gamma+s(\gamma)=2\varpi_{s+2}$ for $\g$ of type ${\rm B}_n$ with $s+2=n$.\par\noindent
Let $1\le i\le s/2-1$ and set $\gamma=\ep_{2i-1}-\ep_{2i}$. As in \cite[Proof of Lem. 7.9]{FL1}, one checks that :\par\noindent
-- If $s\le 4i-2$, then 
$$\begin{array}{cc} s(\ep_{2i-1}-\ep_{2i})=&\displaystyle 2\sum_{j=1}^{s-2i}(\ep_{s-j}-\ep_j)+4\sum_{j=1}^{s/2-i}(\ep_{2j-1}+\ep_{2j})\\
&\displaystyle+2\sum_{j=s/2-i+1}^{i-1}(\ep_{2j-1}+\ep_{2j})+(\ep_{2i-1}+\ep_{2i})+2\ep_s\\
\end{array}$$
in type ${\rm B}_n$ and the same as above in type ${\rm D}_n$ but with $2\ep_s$ replaced by $(\ep_s+\ep_n)+(\ep_s-\ep_n)$.\par\noindent
-- If $s>4i-2$, then
$$\begin{array}{cc} s(\ep_{2i-1}-\ep_{2i})=&\displaystyle 2\sum_{j=1}^{2i-1}(\ep_{s-j}-\ep_j)+4\sum_{j=1}^{i-1}(\ep_{2j-1}+\ep_{2j})\\
&\displaystyle+2\sum_{j=i+1}^{s/2-i}(\ep_{2j-1}+\ep_{2j})+3(\ep_{2i-1}+\ep_{2i})+2\ep_s\\
\end{array}$$
in type ${\rm B}_n$ and the same as above in type ${\rm D}_n$ but with $2\ep_s$ replaced by $(\ep_s+\ep_n)+(\ep_s-\ep_n)$.
In both cases one obtains that $\gamma+s(\gamma)=2\varpi_s$.\par\noindent
Let $1\le j\le [(n-s-1)/2]$ and set $\gamma=\ep_{s+2j}-\ep_{s+2j+1}$. One checks that :
$$\begin{array}{ll} s(\gamma)&=2((\ep_1+\ep_2)+\ldots+(\ep_{s-3}+\ep_{s-2}))+2(\ep_{s-1}+\ep_{s+1})\\
&-2\sum_{k=2}^j(\ep_{s+2k-1}+\ep_{s+2k})
+2\sum_{k=2}^j((\ep_{s+2k-2}+\ep_{s+2k-1})\\
&+(\ep_{s+2j}+\ep_{s+2j+1})+2\ep_s\\
\end{array}$$
in type ${\rm B}_n$, resp. in type ${\rm D}_n$ with $n$ even (with $2\ep_s$ replaced by $(\ep_s-\ep_n)+(\ep_s-\ep_n)$), so that
$\gamma+s(\gamma)=2\varpi_{s+2}$.\par\noindent
In  type ${\rm D}_n$ with $n$ odd, for all $1\le j\le [(n-s-1)/2]-1$, one also obtains that $\gamma+s(\gamma)=2\varpi_{s+2}$.
If $\g$ is of type ${\rm D}_n$, with $n$  odd, then for $\gamma=\ep_{s+2j}-\ep_{s+2j+1}$ with $j=[(n-s-1)/2]=(n-s-1)/2$, one has that
$$\begin{array}{ll}
s(\gamma)&=(\ep_1+\ep_2)+\ldots+(\ep_{s-3}+\ep_{s-2})+(\ep_{s-1}+\ep_{s+1})\\
&+(\ep_{s+2}+\ep_{s+3})+\ldots(\ep_{n-3}+\ep_{n-2})\\
&-((\ep_{s+3}+\ep_{s+4})+\ldots+(\ep_{n-2}+\ep_{n-1}))+(\ep_s+\ep_n)\\
\end{array}$$
so that $\gamma+s(\gamma)=\varpi_{s+2}$.\par\noindent
Let $2\le k\le [(n-s)/2]$ and set $\gamma=-\ep_{s+2k-1}+\ep_{s+2k}$.
One checks that
$$\begin{array}{ll} s(\gamma)&=2((\ep_1+\ep_2)+\ldots+(\ep_{s-3}+\ep_{s-2}))+2(\ep_{s-1}+\ep_{s+1})\\
&-2\sum_{\ell=3}^k(\ep_{s+2\ell-3}+\ep_{s+2\ell-2})
+2\sum_{\ell=2}^k((\ep_{s+2\ell-2}+\ep_{s+2\ell-1})\\
&-(\ep_{s+2k-1}+\ep_{s+2k})+2\ep_s\\
\end{array}$$
in type ${\rm B}_n$, resp. in type ${\rm D}_n$ with $n$ odd (with $2\ep_s$ replaced by $(\ep_s-\ep_n)+(\ep_s-\ep_n)$), so that $\gamma+s(\gamma)=2\varpi_{s+2}$.
If  $\g$ is of type ${\rm D}_n$, with $n$ even, then for all $2\le k\le [(n-s)/2]-1$, one also obtains   that
$\gamma+s(\gamma)=2\varpi_{s+2}$.\par\noindent
Now for $\g$  of type ${\rm D}_n$, with $n$ even and for $\gamma=-\ep_{s+2k-1}+\ep_{s+2k}$, with $k=[(n-s)/2]=(n-s)/2$, one has that
$$\begin{array}{ll}
s(\gamma)&=(\ep_1+\ep_2)+\ldots+(\ep_{s-3}+\ep_{s-2})+(\ep_{s-1}+\ep_{s+1})\\
&+(\ep_{s+2}+\ep_{s+3})+\ldots(\ep_{n-2}+\ep_{n-1})\\
&-((\ep_{s+3}+\ep_{s+4})+\ldots+(\ep_{n-3}+\ep_{n-2}))+(\ep_s-\ep_n)\\
\end{array}$$
so that $\gamma+s(\gamma)=\varpi_{s+2}$.\par\noindent
Finally set $\gamma=\ep_{s+2}-\ep_{s+1}=-\alpha_{s+1}\in T$.
Then one has that
$$s(\gamma)=2((\ep_1+\ep_2)+\ldots+(\ep_{s-3}+\ep_{s-2}))+2(\ep_{s-1}+\ep_{s+1})+2\ep_s$$
so that $\gamma+s(\gamma)=\varpi_s+\varpi_{s+2}$ if $s+2<n$ and
if $s+2=n$ (necessarily in type ${\rm B}_n$) then 
$\gamma+s(\gamma)=\varpi_s+2\varpi_{s+2}$.\par\noindent
It follows that the lower and the improved upper bounds for ${\rm ch}\,(Y(\p_{\Lambda}))$ coincide, then equalities in $(i)$, $(ii)$ and $(iii)$ hold.\end{proof}

\subsection{A Weierstrass section}

By Sect. \ref{IUB} we deduce from Corollary \ref{CSC} and Lemma \ref{comp} that $y+\g_T$ is a Weierstrass section  for coadjoint action of $\p_{\Lambda}$. On can then write the following theorem.

\begin{thm}\label{eigen}

Let $\g$ be a complex simple Lie algebra of type ${\rm B}_n$, resp. ${\rm D}_n$, with $n\ge 4$, resp. $n\ge 6$ and let $\p=\n^-\oplus\h\oplus\n^+_{\pi'}$ be a parabolic subalgebra associated with $\pi'=\pi\setminus\{\alpha_s,\,\alpha_{s+2}\}$, $s$ even,
$2\le s\le n-2$, resp. $2\le s\le n-4$. Then there exists a Weierstrass section $y+\g_T$ for coadjoint action of the canonical truncation $\p_{\Lambda}$ of $\p$.
It follows that the algebra of symmetric invariants $Y(\p_{\Lambda})=Sy(\p)$ is a polynomial algebra over $\Bbbk$ on $n-s/2+1$ algebraically independent homogeneous and $\h$-weight generators. \end{thm}

\subsection{Weights and Degrees.}

One may associate with each $\gamma\in T$ an homogeneous generator $p_{\gamma}\in Y(\p_{\Lambda})$ so that the set $\{p_{\gamma};\;\gamma\in T\}$ is a set of algebraically independent and $\h$-weight generators of the polynomial algebra $Y(\p_{\Lambda})$.
By what we said in Sect. \ref{IUB}, for each $\gamma\in T$, the weight of $p_{\gamma}$ is $wt(p_{\gamma})=-(\gamma+s(\gamma))$ and the degree of $p_{\gamma}$ is $\deg(p_{\gamma})=1+\lvert s(\gamma)\rvert$ (also equal to the eigenvalue plus one of $x_{\gamma}$ with respect to $ad\,h$, where recall $h$ is the semisimple element of the adapted pair for $\p_{\Lambda}$ that we have constructed, see \ref{ssel}). It suffices then to use the proof of Lemma \ref{comp} to compute all these weights and degrees, since there all the $s(\gamma)$, for $\gamma\in T$, have been computed.
Set, for all $1\le i\le s/2-1$, $\gamma_i=\ep_{2i-1}-\ep_{2i}$.\smallskip

Assume first that $\g$ is simple of type ${\rm B}_n$, with $s+2<n$ or $\g$ is of type ${\rm D}_n$, where recall $s+2\le n-2$.\smallskip

\begin{tabular}{|c|c|c|}
\hline
$\gamma\in T$&$wt(p_{\gamma})$&$\deg(p_{\gamma})$\\
\hline
$\gamma_i$, $1\le i\le[s/4]$&$-2\varpi_s$&$s+4i$\\
\hline
$\gamma_i$, $[s/4]+1\le i\le s/2-1$&$-2\varpi_s$&$3s-4i+2$\\
\hline
$\gamma=\ep_{s-1}-\ep_{s+1}$&$-2\varpi_s$&$s+2$\\
\hline
$\gamma=\ep_{s-1}+\ep_s$&$-\varpi_s$&$s/2$\\
\hline
$\gamma=\ep_s+\ep_{s+2}$&$-\varpi_{s+2}$&$s/2+1$\\
\hline
$\gamma=(-1)^n(\ep_n-\ep_{n-1})$, $\g$ of type $D_n$&$-\varpi_{s+2}$&$n-s/2-1$ \\
$\gamma=(-1)^n(\ep_n-\ep_{n-1})$, $\g$ of type $B_n$&$-2\varpi_{s+2}$&$2n-s-2$\\
\hline
$\ep_{s+2j}-\ep_{s+2j+1}$, $1\le j\le[(n-s-2)/2]$&$-2\varpi_{s+2}$&$s+4j$\\
\hline
$\ep_{s+2k}-\ep_{s+2k-1}$, $2\le k\le [(n-s-1)/2]$&$-2\varpi_{s+2}$&$s+4k-2$\\
\hline
$\gamma=\ep_{s+2}-\ep_{s+1}$&$-(\varpi_s+\varpi_{s+2})$&$s+3$\\
\hline
\end{tabular}
\medskip

Assume now that $\g$ is of type ${\rm B}_n$, with $n=s+2$.\smallskip

\begin{tabular}{|c|c|c|}
\hline
$\gamma\in T$&$wt(p_{\gamma})$&$\deg(p_{\gamma})$\\
\hline
$\gamma_i$, $1\le i\le[s/4]$&$-2\varpi_s$&$s+4i$\\
\hline
$\gamma_i$, $[s/4]+1\le i\le s/2-1$&$-2\varpi_s$&$3s-4i+2$\\
\hline
$\gamma=\ep_{s-1}-\ep_{s+1}$&$-2\varpi_s$&$s+2$\\
\hline
$\gamma=\ep_{s-1}+\ep_s$&$-\varpi_s$&$s/2$\\
\hline
$\gamma=\ep_s+\ep_{s+2}$&$-\varpi_{s+2}$&$s/2+1$\\
\hline
$\gamma=\ep_{s+2}-\ep_{s+1}$&$-(\varpi_s+2\varpi_{s+2})$&$s+3=n+1$\\
\hline
\end{tabular}

\subsection{}\label{notwork}
\begin{Rq}\rm
Assume that $\g$ is simple of type ${\rm B}_n$  with $n\ge 6$ and that $\pi'=\pi\setminus\{\alpha_s,\,\alpha_{s+2},\,\alpha_{s+4}\}$
with $s$ an even integer and consider the parabolic subalgebra $\p=\n^-\oplus\h\oplus\n^+_{\pi'}$ associated to $\pi'$. Then one may easily check as in the proof of Lemma \ref{T}  that $\ind\p_{\Lambda}=n-s/2+1$. Take  the same set $S^+$ as this chosen for the case $\pi'=\pi\setminus\{\alpha_s,\,\alpha_{s+2}\}$ in subsection \ref{APP} and the same set $S^-$ but without the element $-\ep_{s+3}-\ep_{s+4}$ which does no more belong to $\Delta^-_{\pi'}$. Then restriction to $\h'=\h_{\Lambda}$ of $S=S^+\sqcup S^-$ is still a basis for $\h_{\Lambda}^*$. Take also the same sets $T$ and $T^*$ as in subsections \ref{T} and \ref{condition(v)caseeven}, which still lie in $\Delta^+\sqcup\Delta^-_{\pi'}$. Unfortunately condition $(v)$ of Proposition \ref{propAP} is no more satisfied since $x_{\ep_s+\ep_{s+3}}$ and $x_{\ep_s+\ep_{s+4}}$ in $T^*$ do no more belong to $ad\,\p_{\Lambda}(y)+\g_T$. Thus our construction cannot be generalized to the more general case of $\p_{s,\,\ell}$ with $s$ even and $\ell\ge 2$.
\end{Rq}

\section{Case \ref{cas5} for type C}\label{TC}

In this Section, we consider a parabolic subalgebra $\p=\p_{s,\,\ell}=\p^-_{\pi'}=\n^-\oplus\h\oplus\n^+_{\pi'}$
associated to the subset
$\pi'=\pi\setminus\{\alpha_s,\,\alpha_{s+2},\ldots,\alpha_{s+2\ell}\}$ with $\ell\in\mathbb N$ and $s$ an even or an odd integer,
$1\le s\le n-2\ell$, in a simple Lie algebra $\g$ of type ${\rm C}_n$, with $n\ge 3$. Hence we are in the case \ref{cas5} of subsection \ref{Mainres}.

If $\ell=0$, such a parabolic subalgebra is maximal and this case was already treated in \cite{FL}.
Thus we will assume that $\ell\ge 1$.
By subsection \ref{equalitybounds}, the lower and upper bounds for ${\ch}\,(Y(\p_{\Lambda}))$ in \eqref{ch} of Sect. \ref{bounds} always coincide and then $Y(\p_{\Lambda})$ is a polynomial algebra. However Weierstrass sections were not yet exhibited.
As we said in Remark \ref{rqeqbounds} of subsection \ref{REG}, it suffices to construct an adapted pair to obtain a Weierstrass section for coadjoint action of $\p_{\Lambda}$ in the present case.
Our construction generalizes the construction of an adapted pair in case of a maximal parabolic subalgebra in type C
(see \cite[Sect. 6]{FL}).

\subsection{The Kostant cascades.}

The Kostant cascade $\beta_{\pi}$ for $\g$ simple of type ${\rm C}_n$ is given by $$\beta_{\pi}=\{\beta_i=2\varepsilon_i\,;\, 1\le i\le n\}.$$ 
The Kostant cascade $\beta_{\pi'}$ of $\g'$ is given by 
$$\begin{array}{lc}
\beta_{\pi'}=&\{\beta'_i=\ep_i-\ep_{s+1-i},\,\alpha_{s+2j-1},\,\beta''_k=2\ep_{s+2\ell+k}\,;\\
&1\le i\le[s/2],\,1\le j\le \ell,\,1\le k\le n-s-2\ell\}.\end{array}$$
We want to construct an adapted pair for $\p_{\Lambda}=\n^-\oplus\h'\oplus\n^+_{\pi'}$. (Recall that here $\h_{\Lambda}=\h'$). For this purpose we will use Prop. \ref{propAP}. First we give a set $S=S^+\sqcup S^-\subset\Delta^+\sqcup\Delta^-_{\pi'}$ such that $S_{\mid\h_{\Lambda}}$ is a basis for $\h_{\Lambda}^*$. 

\subsection{Condition {\it (i)} of Proposition \ref{propAP}.}

Since, for all $1\le k\le n-s-2\ell$, one has that $\beta''_k=\beta_{s+2\ell+k}$, we will not be able to take $S=\beta^0_{\pi}\cup(-\beta^0_{\pi'})$ as we did for type ${\rm B}_n$ in Sect. \ref{SC}.

Instead we will take elements which are a kind of deformation of roots of the Kostant cascade, by 
setting $\gamma_i=\beta_i-\alpha_i=\ep_i+\ep_{i+1}$ for all $1\le i\le n-1$.
Assume first that $s$ is odd.
We set 
$$S^+=\{\gamma_{2i-1}=\ep_{2i-1}+\ep_{2i}\,;\, 1\le i\le [n/2]\}$$
and
$$\begin{array}{lc}
S^-=&\{-\beta'_i=\ep_{s+1-i}-\ep_i,\,-\gamma_{2j}=-(\ep_{2j}+\ep_{2j+1})\,;\\
& 1\le i\le (s-1)/2,\,(s+2\ell+1)/2\le j\le[(n-1)/2]\}.
\end{array}$$
Assume now that $s$ is even and set $t:=[s/4]$.
We set 
$$S^+=\{\beta_{2t+1},\,\gamma_{2i-1},\,\gamma_{2j}\,;\, 1\le i\le t,\,t+1\le j\le [(n-1)/2]\}$$
and
$$\begin{array}{lc}
S^-=&\{-\beta'_i=\ep_{s+1-i}-\ep_i,\,-\gamma_{2j+1}=-(\ep_{2j+1}+\ep_{2j+2})\,;\\
& 1\le i\le (s-2)/2,\,(s+2\ell)/2\le j\le[(n-2)/2]\}.
\end{array}$$
In both cases for $S=S^+\sqcup S^-$, one easily checks  that $\lvert S\rvert=n-\ell-1=\dim\h'=\dim\h_{\Lambda}$.
The following lemma establishes condition $(i)$ of Proposition \ref{propAP}.

\begin{lm}\label{BaseC}

$S_{\mid\h_{\Lambda}}$ is a basis for $\h_{\Lambda}^*$.

\end{lm}

\begin{proof}

Assume first that $s$ is odd.
Then we order the elements $s_u$ of $S$ as follows :
$$\begin{array}{cc}
s_i=\gamma_{2i-1},\,1\le i\le [n/2]\\
s_{[n/2]+j}=-\gamma_{s+2\ell+2j-1},\,1\le j\le [(n-1)/2]-(s+2\ell-1)/2\\
s_{n-\ell-1-(s-1)/2+k}=-\beta'_k,\,1\le k\le (s-1)/2\\
\end{array}$$
and order the elements $h_v$ of a basis of $\h'$ as follows :
$$\begin{array}{cc}
h_i=\alpha^\vee_{2i},\,1\le i\le[n/2]\\
h_{[n/2]+j}=\alpha^\vee_{s+2\ell+2j},\,1\le j\le [(n-1)/2]-(s+2\ell-1)/2\\
h_{n-\ell-1-(s-1)/2+2k-1}=h'_{2k-1}=\alpha^\vee_{2k-1},\,h_{n-\ell-1-(s-1)/2+2k}=h'_{2k}=\alpha^\vee_{s-2k},\\
\,1\le k\le [(s+1)/4]\\
\end{array}$$
without repetitions for the $h'_j$'s.\par\noindent
Set $$\begin{array}{cc}
A=(s_u(h_v))_{1\le u,\,v\le [n/2]},\\
B=(s_{[n/2]+u}(h_{[n/2]+v}))_{1\le u,\,v\le [(n-1)/2]-(s+2\ell-1)/2},\\
C=(-\beta'_i(h'_j))_{1\le i,\,j\le (s-1)/2}.\end{array}$$
By observing that $\gamma_i=\varpi_{i+1}-\varpi_{i-1}$ for all $1\le i\le n-1$ (with $\varpi_0=0$) one obtains that
$A$, resp. $B$, is a lower triangular matrix with 1, resp. $-1$, on its diagonal. Moreover $C$ is a lower triangular matrix with $-1$ on its diagonal  by Lemma \ref{KCA}.
Then one obtains that the matrix $$(s_u(h_v))_{1\le u,\,v\le n-\ell-1}=\begin{pmatrix}
A&0&0\\
*&B&0\\
*&*&C\\
\end{pmatrix}$$
is such that $\det(s_u(h_v))_{1\le u,\,v\le n-\ell-1}\neq 0$.\par\noindent
Assume now that $s$ is even. Recall that $t=[s/4]$.
We order the elements $s_u$ of $S$ as follows :
$$\begin{array}{cc}
s_i=\gamma_{2i-1},\,1\le i\le t,\\
s_{t+1}=\beta_{2t+1}=2\ep_{2t+1},\\
s_{t+1+j}=\gamma_{2(t+j)},\,1\le j\le[(n-1)/2]-t,\\
s_{[(n+1)/2]+k}=-\gamma_{s+2\ell+2k-1},\,1\le k\le [(n-2)/2]-(s+2\ell-2)/2,\\
s_{n-\ell-1-(s-2)/2+r}=-\beta'_r,\,1\le r\le (s-2)/2.\\
\end{array}$$
We order the elements $h_v$ of a basis of $\h'$ as follows :
$$\begin{array}{cc}
h_i=\alpha^\vee_{2i},\,1\le i\le t,\\
h_{t+1}=\alpha^\vee_{2t+1},\\
h_{t+1+j}=\alpha^\vee_{2(t+j)+1},\,1\le j\le[(n-1)/2]-t,\\
h_{[(n+1)/2]+k}=\alpha^\vee_{s+2\ell+2k},\,1\le k\le [(n-2)/2]-(s+2\ell-2)/2,\\
h_{n-\ell-1-(s-2)/2+2r-1}=h'_{2r-1}=\alpha^\vee_{2r-1},\,h_{n-\ell-1-(s-2)/2+2r}=h'_{2r}=\alpha^\vee_{s-2r},\\
1\le r\le t\\
\end{array}$$
without repetitions. More precisely : if $t=(s-2)/4$ then $h_t=\alpha_{2t}^\vee=\alpha_{s/2-1}^\vee\neq h'_{2t}=\alpha^\vee_{s/2+1}$, then both are taken and if $t=s/4$ then $h'_{2t}=\alpha^\vee_{s/2}=h_t$
then one takes $h_t$ but not $h'_{2t}$.
Then by the above  one obtains that the matrix $$(s_u(h_v))_{1\le u,\,v\le n-\ell-1}=\begin{pmatrix}
A&0&0&0&0\\
*&2&0&0&0\\
*&*&B&0&0\\
*&*&*&C&0\\
*&*&*&*&D\\
\end{pmatrix}$$
with $$\begin{array}{cc}
A=(s_u(h_v))_{1\le u,\,v\le t},\\
B=(s_{t+1+u}(h_{t+1+v}))_{1\le u,\,v\le[(n-1)/2]-t},\\
C=(s_{[(n+1)/2]+u}(h_{[n+1)/2]+v}))_{1\le u,\,v\le  [(n-2)/2]-(s+2\ell-2)/2},\\
D=(-\beta'_i(h'_j))_{1\le i,\,j\le (s-2)/2}\end{array}$$
is such that $\det(s_u(h_v))_{1\le u,\,v\le n-\ell-1}\neq 0$.
Indeed the above matrices are lower triangular matrices 
with 1 (for $A$ and $B$), resp. $-1$ (for $C$ and $D$), on their diagonal.
\end{proof}

\subsection{Conditions {\it (ii)} and {\it (iii)} of Proposition \ref{propAP}.}\label{condition(ii)typeC}

For each $\gamma\in S^+$, resp. $\gamma\in S^-$, we will take a Heisenberg set $\Gamma_{\gamma}$ with centre $\gamma$ included in $\Delta^+$, resp. $\Delta^-_{\pi'}$ such that $\lvert(\Delta^+\sqcup\Delta^-_{\pi'})\setminus( \bigsqcup_{\gamma\in S}\Gamma_{\gamma})\rvert=\ind\p_{\Lambda}$ (here we will take $T^*=\emptyset$ in Prop.
\ref{propAP}) and such that conditions $(ii)$ and $(iii)$ of Prop. \ref{propAP} are satisfied.

For this purpose we use (see \ref{HS}) the largest Heisenberg set $H_{\beta_i}$ with centre $\beta_i\in\beta_{\pi}$ included in $\Delta^+$, and $-H_{\beta'_i}$, resp. $-H_{\beta''_i}$ where $H_{\beta'_i}$, resp. $H_{\beta''_i}$, is the largest Heisenberg set in $\Delta^+_{\pi'}$ with centre $\beta'_i$, resp. $\beta''_i$, belonging to the Kostant cascade $\beta_{\pi'}$ of $\g'$.

For each $\beta_i\in\beta_{\pi}$, set $H_{\beta_i}^0=H_{\beta_i}\setminus\{\beta_i\}$ and each $\beta''_i\in\beta_{\pi'}$, $H^0_{-\beta''_i}=-H_{\beta''_i}\setminus\{-\beta''_i\}$.
As in \cite[Sect. 6]{FL}, for each $\gamma\in S^+\cap\beta_{\pi}$, we set $\Gamma_{\gamma}=H_{\gamma}$ and for each $\gamma\in S^-\cap(-\beta_{\pi'})$, we set $\Gamma_{\gamma}=-H_{-\gamma}$.
Moreover for the roots $\gamma_i=\beta_i-\alpha_i\in S^+$, we set $\Gamma_{\gamma_i}=H^0_{\beta_i}\sqcup H_{\beta_{i+1}}$ and for the roots $-\gamma''_i=-\gamma_{s+2\ell+i}=-(\beta''_i-\alpha_{s+2\ell+i})\in S^-$, we set $\Gamma_{-\gamma''_i}=H^0_{-\beta''_i}\sqcup(-H_{\beta''_{i+1}})$.\par\noindent
As in  \cite[Sect. 6]{FL} one easily checks  that, for each $\gamma\in S$, $\Gamma_{\gamma}$ is a Heisenberg set with centre $\gamma$ and these sets $\Gamma_{\gamma}$, $\gamma\in S$, are pairwise disjoint by Lemma \ref{HS} $i)$.
Below we will show that conditions $(ii)$ and $(iii)$ of Prop. \ref{propAP} are satisfied.
Recall that, for each $\gamma\in S$, we set  $\Gamma^0_{\gamma}=\Gamma_{\gamma}\setminus\{\gamma\}$ and $O^\pm=\bigsqcup_{\gamma\in S^\pm}\Gamma^0_{\gamma}$.

\begin{lm}
Let $\gamma\in S^+$ and $\alpha\in\Gamma^0_{\gamma}$ such that there exists $\beta\in O^+$ with $\alpha+\beta\in S$. Then $\beta\in\Gamma^0_{\gamma}$ and $\alpha+\beta=\gamma$.

\end{lm}

\begin{proof}

Assume first that $\alpha+\beta\in S^+\cap\beta_{\pi}$.
Then the assertion follows from Lemma \ref{HS} $ii)$.
Assume now that $\alpha+\beta=\gamma_i=\beta_i-\alpha_i\in S^+$. We will show that $\gamma=\gamma_i$.
We have that $\gamma_i\in H^0_{\beta_i}$ then by Lemma \ref{HS} $iii)$, one has that $\alpha\in H_{\beta_i}$ or $\beta\in H_{\beta_i}$.
Suppose that $\alpha\in H_{\beta_i}$. Then $\alpha\neq \beta_i$ since $\beta$ is a positive root and
we have $\alpha\in H^0_{\beta_i}\subset\Gamma_{\gamma_i}$. Since the Heisenberg sets $\Gamma_{\gamma}$, $\gamma\in S$, are pairwise disjoint, one deduces that $\gamma=\gamma_i$. Since $\Gamma_{\gamma}$ is a Heisenberg set, it follows that $\beta\in\Gamma^0_{\gamma}$. 
Now if $\beta\in H_{\beta_i}$ then for the same reason as before $\beta\in H^0_{\beta_i}\subset\Gamma_{\gamma_i}$.
But $\beta\in O^+$ then there exists $\gamma'\in S^+$ such that $\beta\in\Gamma^0_{\gamma'}$. As before one deduces that $\gamma'=\gamma_i$ and then that $\alpha\in\Gamma^0_{\gamma'}$ hence $\gamma=\gamma'=\gamma_i$.
Since all roots in $S^+$ are of the above form, we are done.
\end{proof}

Condition $(iii)$ of Prop. \ref{propAP} follows similarly. 

\subsection{Condition {\it (vi)} of Proposition \ref{propAP}.}

Here we will show that condition $(vi)$ is satisfied, with $T=(\Delta^+\sqcup\Delta^-_{\pi'})\setminus\bigsqcup_{\gamma\in S}\Gamma_{\gamma}$.
Set $T^+=\Delta^+\setminus\bigsqcup_{\gamma\in S^+}\Gamma_{\gamma}$ and $T^-=\Delta^-_{\pi'}\setminus\bigsqcup_{\gamma\in S^-}\Gamma_{\gamma}$. Then $T=T^+\sqcup T^-$.
Recall  Lemma \ref{HS} $i)$ that $\Delta^+=\bigsqcup_{\beta\in\beta_{\pi}}H_{\beta}$ and $\Delta^-_{\pi'}=\bigsqcup_{\beta\in\beta_{\pi'}}(-H_{\beta})$. \par\noindent
Assume first that $s$ is odd.
Then $$T^+=\{\beta_{2i-1}\mid 1\le i\le [(n+1)/2]\}$$ and 
$$T^-=\{-\alpha_{s+2i-1},\,-\beta_{s+2\ell+2j-1},\,\mid 1\le i\le \ell,\,1\le j\le[(n+1-s)/2]-\ell\}.$$
Assume that $s$ is even.
Then $$T^+=\{\beta_{2i-1},\,\beta_{2j}\mid 1\le i\le t,\,t+1\le j\le [n/2]\}$$
and
$$T^-=\{-\alpha_{s/2},\,-\alpha_{s+2i-1},\,-\beta_{s+2\ell+2j-1}\mid 1\le i\le \ell,\,1\le j\le [(n+1-s)/2]-\ell\}.$$
Below we establish condition $(vi)$ of Prop. \ref{propAP}.

\begin{lm}

One has that $\lvert T\rvert=\ind\p_{\Lambda}$.

\end{lm}

\begin{proof}

Recall that $\ind\p_{\Lambda}$ is equal to the number $\lvert E(\pi')\rvert$ of $\langle ij\rangle$-orbits in $\pi$ (see Sect. \ref{bounds}).
Here, since $j=Id_{\pi}$, the set $E(\pi')$ of $\langle ij\rangle$-orbits in $\pi$ is the following.
If $s$ is odd, then $$E(\pi')=\{\{\alpha_i,\,\alpha_{s-i}\},\,\{\alpha_{s-1+j}\}\mid 1\le i\le (s-1)/2,\,1\le j\le n-s+1\}.$$
If $s$ is even, then $$E(\pi')=\{\{\alpha_i,\,\alpha_{s-i}\},\,\{\alpha_{s/2}\},\,\{\alpha_{s-1+j}\}\mid 1\le i\le (s-2)/2,\,1\le j\le n-s+1\}.$$
One checks that $\lvert T\rvert=\lvert E(\pi')\rvert$. Hence the lemma.
\end{proof}

Finally if we set $T^*=\emptyset$, then by construction condition $(iv)$ of Prop. \ref{propAP} is also satisfied and condition $(v)$ is empty.

\subsection{A Weierstrass section.}\label{WSC}

By the above, all conditions of Prop. \ref{propAP} are satisfied. Set $y=\sum_{\gamma\in S}x_{\gamma}$.
Since $S_{\mid\h_{\Lambda}}$ is a basis for $\h_{\Lambda}^*$ there exists a unique $h\in\h_{\Lambda}$ such that for all $\gamma\in S$, $\gamma(h)=-1$. Then by Prop. \ref{propAP} $(h,\,y)$ is an adapted pair for $\p_{\Lambda}$. 
Moreover by subsection \ref{equalitybounds}, for all $\Gamma\in E(\pi')$, $\varepsilon_{\Gamma}=1$. Then by Remark \ref{rqeqbounds} of subsection \ref{REG}, one deduces that
 $y+\g_T$ is a Weierstrass section  for coadjoint action of  $\p_{\Lambda}$.
Summarizing we obtain  the following theorem.

\begin{thm}\label{ThmTC}

Let $\g$ be a complex simple Lie algebra of type ${\rm C}_n$ ($n\ge 3$) and let $\p=\n^-\oplus\h\oplus\n^+_{\pi'}$ be
a parabolic subalgebra of $\g$ associated to $\pi'=\pi\setminus\{\alpha_s,\,\alpha_{s+2},\ldots,\,\alpha_{s+2\ell}\}$ ($s,\,\ell\in\mathbb N^*$)
and $1\le s\le n-2\ell$. Then $y+\g_T$ is a Weierstrass section  for coadjoint action of the canonical truncation $\p_{\Lambda}$ of $\p$.

\end{thm}

\subsection{Weights and degrees.}

Here both bounds (see \eqref{ch} in Sect. \ref{bounds}) for ${\ch}(Y(\p_{\Lambda}))$ coincide and then $Y(\p_{\Lambda})$ is a polynomial algebra whose homogeneous and $\h$-weight generators have weights and degrees which can be easily computed.
To each $\Gamma\in E(\pi')$ is associated an homogeneous and $\h$-weight  generator of $Y(\p_{\Lambda})$ which has weight $\delta_{\Gamma}$ given by (\ref{poids}) of Sect. \ref{bounds} and a degree $\partial_{\Gamma}$ given by (\ref{degre}) of Sect. \ref{bounds}.

Below  we give for completeness weights and degrees of a set of homogeneous and $\h$-weight algebraically independent generators of $Y(\p_{\Lambda})$, each of them corresponding to an $\langle ij\rangle$-orbit $\Gamma_r$ in $E(\pi')$.\smallskip

Assume first that $s$ is odd  :\medskip

\begin{tabular}{|c|c|c|}
\hline
$\langle ij\rangle$-orbit in $E(\pi')$& Weight & Degree\\
\hline
$\Gamma_u=\{\alpha_u,\,\alpha_{s-u}\},\,1\le u\le (s-1)/2$&$-2\varpi_s$&$s+2u$\\
\hline
$\Gamma_v=\{\alpha_v\},\;v=s+2k,\,0\le k\le\ell$ & $-2\varpi_v$&$v$\\
\hline
$\Gamma_v=\{\alpha_v\},\;v=s+2k-1,\,1\le k\le\ell$ & $-\varpi_{v-1}-\varpi_{v+1}$&$v+1$\\
\hline
$\Gamma_v=\{\alpha_v\},\;s+2\ell+1\le v\le n$ & $-2\varpi_{s+2\ell}$&$2v-s-2\ell$\\
\hline

\end{tabular}
\bigskip

Assume now that $s$ is even :\medskip

\begin{tabular}{|c|c|c|}
\hline
$\langle ij\rangle$-orbit in $E(\pi')$& Weight & Degree\\
\hline
$\Gamma_u=\{\alpha_u,\,\alpha_{s-u}\},\;1\le u\le (s-2)/2$&$-2\varpi_s$&$s+2u$\\
\hline
$\Gamma_{s/2}=\{\alpha_{s/2}\}$&$-\varpi_s$&$s$\\
\hline
$\Gamma_v=\{\alpha_v\},\;v=s+2k,\,0\le k\le\ell$ & $-2\varpi_v$&$v$\\
\hline
$\Gamma_v=\{\alpha_v\},\;v=s+2k-1,\,1\le k\le\ell$ & $-\varpi_{v-1}-\varpi_{v+1}$&$v+1$\\
\hline
$\Gamma_v=\{\alpha_v\},\;s+2\ell+1\le v\le n$ & $-2\varpi_{s+2\ell}$&$2v-s-2\ell$\\
\hline

\end{tabular}

\section{Case \ref{cas6} for type D.}

In this Section we consider the Lie parabolic subalgebra $\p=\p_{\ell}$ of the simple Lie algebra $\g$ of type ${\rm D}_n$, with $n\ge 4$, $n$ even and $\ell\in\mathbb N$, $0\le \ell\le(n-2)/2$, associated with the subset $\pi'=\pi\setminus\{\alpha_{n-1-2k},\,\alpha_n\mid 0\le k\le\ell\}$. This is the case \ref{cas6} of subsection \ref{Mainres}. Recall \ref{KcasBD} the Kostant cascade $\beta_{\pi}$ for $\g$ of type ${\rm D}_n$.
Recall \ref{KCA}  the Kostant cascade $\beta_{\pi_1'}\subset\beta_{\pi'}$ for the simple Lie subalgebra $\g_{\pi'_1}$ of the Levi subalgebra $\g'$ of $\p$ of type ${\rm A}_{n-2-2\ell}$ if $\ell<(n-2)/2$.  One has
$$\beta_{\pi_1'}=\{\beta'_i=\ep_i-\ep_{n-2\ell-i}\mid 1\le i\le [(n-1-2\ell)/2]\}$$

We denote (as in \ref{KcasBD}) $\beta^0_{\pi}=\beta_{\pi}\setminus(\beta_{\pi}\cap\pi)$ and $\beta^0_{\pi'}=\beta_{\pi'}\setminus(\beta_{\pi'}\cap\pi')$.

Then we have that $\beta^0_{\pi}=\{\beta_i=\ep_{2i-1}+\ep_{2i}\mid 1\le i\le (n-2)/2\}$ and $\beta^0_{\pi'}=\beta_{\pi'_1}$ since $n$ is even.

We set $S^+=\beta^0_{\pi}=\{\beta_i\mid 1\le i\le (n-2)/2\}$ and $S^-=-\beta^0_{\pi'}=\{-\beta'_i\mid 1\le i\le (n-2-2\ell)/2\}$.

For all $\gamma\in S^+$, we set $\Gamma_{\gamma}=H_{\gamma}$ the largest Heisenberg set with centre $\gamma$ which is included in $\Delta^+$ as defined in subsection \ref{HS} and for all $\gamma\in S^-$, we set $\Gamma_{\gamma}=-H_{-\gamma}$ where $H_{-\gamma}$ is the largest Heisenberg set with centre $-\gamma$ which is included in $\Delta_{\pi'}^+$.
Finally we set $T^+=\beta_{\pi}\cap\pi$, $T^-=-(\beta_{\pi'}\cap\pi')$, $T=T^+\cup T^-$  and $T^*=\emptyset$.

\subsection{Conditions {\it (i)}  to {\it(v)} of Proposition \ref{propAP}}\label{basisD}

By $i)$ of Lemma \ref{HS} and since $H_{\beta}=\{\beta\}$ for all $\beta\in\beta_{\pi}\cap\pi$, condition $(iv)$ of Prop. \ref{propAP} is satisfied. Moreover conditions $(ii)$ and $(iii)$ of Prop. \ref{propAP} are satisfied by $ii)$ of Lemma \ref{HS}. Condition $(v)$ is empty since $T^*=\emptyset$. 
Condition $(i)$ follows from the following Lemma.

\begin{lm}

Set $S=S^+\cup S^-$. Then $S_{\mid\h_{\Lambda}}$ is a basis for $\h_{\Lambda}^*$.

\end{lm}

\begin{proof}
Here  $j=Id_{\pi}$ and then (see Sect. \ref{not}) we have that $\h'=\h_{\Lambda}$ and
we observe  that $\lvert S\rvert=n-2-\ell=\dim\h'=\dim\h_{\Lambda}$. The proof is similar to the proof of Lemma \ref{baseFC}.
We order the elements $s_u$ of $S$ as
$$\beta_1,\,\beta_2,\,\ldots,\,\beta_{(n-2)/2},\,-\beta'_1,\,-\beta'_2,\,\ldots,\,-\beta'_{(n-2-2\ell)/2}$$

and we choose the following ordered basis $(h_v)_{1\le v\le n-2-\ell}$ of $\h'$ :

$$\begin{array}{cc} h_i=\alpha_{2i}^\vee,\; 1\le i\le (n-2)/2\\
h_{(n-2)/2+2j-1}=h'_{2j-1}=\alpha_{2j-1}^\vee,\,h_{(n-2)/2+2j}=h'_{2j}=\alpha_{n-1-2\ell-2j}^\vee,\\
1\le j\le [(n-2\ell)/4]\\
\end{array}$$

without repetitions for the $h'_j$'s.

Then the matrix $(s_u(h_v))_{1\le u,\,v\le n-2-\ell}$ has the form

$$\begin{pmatrix} A&0\\
*&B\end{pmatrix}$$
Here $A=(\beta_i(h_j))_{1\le i,\,j\le(n-2)/2}$ is a $(n-2)/2\times(n-2)/2$ lower triangular matrix with $1$ on its diagonal since $\beta_i=\varpi_{2i}-\varpi_{2i-2}$. Moreover $B=(-\beta'_i(h'_j))_{1\le i,\,j\le (n-2-2\ell)/2}$ is a $(n-2-2\ell)/2\times(n-2-2\ell)/2$ lower triangular matrix with $-1$ on its diagonal, by Lemma \ref{KCA}. 

Hence $\det(s_u(h_v))_{1\le u,\,v\le n-2-\ell}\neq 0$.
\end{proof}

\subsection{Condition {\it (vi)} of Proposition \ref{propAP}}

We obtain the following Lemma.

\begin{lm}\label{TD}

We have that $\lvert T\rvert=\ind\p_{\Lambda}$.

\end{lm}

\begin{proof}

Recall \eqref{index} of Sect. \ref{bounds} that $\ind\p_{\Lambda}=\lvert E(\pi')\rvert$. Here the set $E(\pi')$ of $\langle ij\rangle$-orbits in $\pi$ is the following :

$$\begin{array}{cc}
E(\pi')=\{\Gamma_u=\{\alpha_u,\,\alpha_{n-1-2\ell-u}\},\,1\le u\le (n-2-2\ell)/2,\\
\Gamma_v=\{\alpha_v\},\,n-1-2\ell\le v\le n\}.\end{array}$$ 
Hence $\ind\p_{\Lambda}=(n+2+2\ell)/2$.

On the other hand we have that
$T^+=\beta_{\pi}\cap\pi=\{\alpha_n,\,\alpha_{2i-1},\,1\le i\le n/2\}$ by \ref{KcasBD}, and $T^-=-(\beta_{\pi'}\cap\pi')=\{-\alpha_{2i},\,(n-2\ell)/2\le i\le (n-2)/2\}$.  Hence $\lvert T^+\rvert=n/2+1$ and $\lvert T^-\rvert=\ell$. Thus $\lvert T\rvert=\ind\p_{\Lambda}$.
\end{proof}

\subsection{} All conditions of Prop. \ref{propAP} are satisfied. Thus one can deduce the following corollary.

\begin{cor} Keep the above notation and set $y=\sum_{\alpha\in S} x_{\alpha}$. Then $y$ is regular in $\p_{\Lambda}^*$ and more precisely one has that $ad\,\p_{\Lambda}(y)\oplus\g_T=\p_{\Lambda}^*$. Moreover there exists a uniquely defined $h\in\h_{\Lambda}$ such that $\alpha(h)=-1$ for all $\alpha\in S$. Thus the pair $(h,\,y)$ is an adapted pair for $\p_{\Lambda}$.

\end{cor}

\subsection{Existence of a Weierstrass section.}\label{boundsD}

By Remark \ref{rqeqbounds} of subsection \ref{REG}, the existence of an adapted pair for $\p_{\Lambda}$ is sufficient to produce a Weierstrass section for coadjoint action of $\p_{\Lambda}$ provided one has the following Lemma.

\begin{lm}
Keep the above hypotheses and notation. One has that $\varepsilon_{\Gamma}=1$ for all $\Gamma\in E(\pi')$.

\end{lm}

\begin{proof}
Recall subsection \ref{compepsilon}  and the $\langle ij\rangle$-orbits in $E(\pi')$ described in the proof of Lemma \ref{TD}.

For $\Gamma_u=\{\alpha_u,\,\alpha_{n-1-2\ell-u}\}$ for $1\le u\le (n-2-2\ell)/2$, one has that $d_{\Gamma_u}=\varpi_u+\varpi_{n-1-2\ell-u}\not\in\mathcal B_{\pi}$ since $u$ and $n-1-2\ell-u$ are not of the same parity.

Let $n-1-2\ell\le v< n$. If $v$ is even, then $d_{\Gamma_v}=\varpi_v\in\mathcal B_{\pi}$ but $d'_{\Gamma_v}=\varpi'_v\not\in\mathcal B_{\pi'}$ since $\alpha_v$ belongs to a connected component of $\pi'$ of type ${\rm A}_1$. If $v$ is odd, then $d_{\Gamma_v}=\varpi_v\not\in\mathcal B_{\pi}$.
Finally $d_{\Gamma_n}=\varpi_n\not\in\mathcal B_{\pi}$. Hence the lemma.
\end{proof}

One can then deduce the following Theorem.

\begin{thm}
Let $\g$ be a simple Lie algebra of type ${\rm D}_n$, with $n$ even, $n\ge 4$. Let $\ell\in\mathbb N$ be such that $0\le \ell\le(n-2)/2$ and $\p_{\ell}$ be the parabolic subalgebra of $\g$ associated with the subset $\pi'=\pi\setminus\{\alpha_{n-1-2k},\,\alpha_n\mid 0\le k\le \ell\}$. 
Then there exists a Weierstrass section for coadjoint action of the canonical truncation of $\p_{\ell}$.

\end{thm}

\subsection{Weights and degrees}

Here both bounds (see \eqref{ch} in Sect. \ref{bounds}) for ${\ch}(Y(\p_{\Lambda}))$ coincide by Lemma \ref{boundsD} and then $Y(\p_{\Lambda})$ is a polynomial algebra whose homogeneous and $\h$-weight generators have weights and degrees which can be easily computed.
To each $\Gamma\in E(\pi')$ is associated an homogeneous and $\h$-weight  generator of $Y(\p_{\Lambda})$ which has weight $\delta_{\Gamma}$ given by (\ref{poids}) of Sect. \ref{bounds} and a degree $\partial_{\Gamma}$ given by (\ref{degre}) of Sect. \ref{bounds}.

Below  we give for completeness weights and degrees of a set of homogeneous and $\h$-weight algebraically independent generators of $Y(\p_{\Lambda})$, each of them corresponding to an $\langle ij\rangle$-orbit $\Gamma_r$ in $E(\pi')$.\smallskip

Assume first that $\ell\ge 1$  :\medskip

\begin{tabular}{|c|c|c|}
\hline
$\langle ij\rangle$-orbit in $E(\pi')$& Weight & Degree\\
\hline
$\Gamma_u=\{\alpha_u,\,\alpha_{n-1-2\ell-u}\}$&$-2\varpi_{n-1-2\ell}$&$n-2\ell+2u$\\
$1\le u\le (n-2-2\ell)/2$&&\\
\hline
$\Gamma_v=\{\alpha_v\}$& $-2\varpi_v$&$v+1$\\
$v=n-1-2k,\,1\le k\le\ell$&&\\
\hline
$\Gamma_v=\{\alpha_v\}$& $-\varpi_{v-1}-\varpi_{v+1}$&$v+1$\\
$v=n-2k,\,2\le k\le\ell$&&\\
\hline
$\Gamma_{n-2}=\{\alpha_{n-2}\}$&$-\varpi_{n-3}-\varpi_{n-1}-\varpi_n$&$n-1$\\
\hline
$\Gamma_{n-1}=\{\alpha_{n-1}\}$& $-2\varpi_{n-1}$&$n/2$\\
\hline
$\Gamma_{n}=\{\alpha_n\}$&$-2\varpi_n$&$n/2$\\
\hline

\end{tabular}
\bigskip

Finally assume that $\ell=0$, that is, $\pi'=\pi\setminus\{\alpha_{n-1},\,\alpha_n\}$ : \medskip

\begin{tabular}{|c|c|c|}
\hline
$\langle ij\rangle$-orbit in $E(\pi')$& Weight & Degree\\
\hline
$\Gamma_u=\{\alpha_u,\,\alpha_{n-1-u}\},\;1\le u\le (n-2)/2$&$-2(\varpi_{n-1}+\varpi_n)$&$n+2u$\\
\hline
$\Gamma_{n-1}=\{\alpha_{n-1}\}$&$-2\varpi_{n-1}$&$n/2$\\
\hline
$\Gamma_n=\{\alpha_n\}$&$-2\varpi_n$&$n/2$\\
\hline
\end{tabular}

\subsection{}

\begin{Rq}\rm

Assume now that $\g$ is simple of type ${\rm D}_n$ with $n$ odd and consider the parabolic subalgebra $\p=\p_{\ell}$ with $0\le\ell\le(n-2)/2$. Assume that we have found an adapted pair $(h,\,y)$ for $\p_{\Lambda}$ with $y=\sum_{\gamma\in S} x_{\gamma}$, $S\subset\Delta^+\sqcup\Delta^-_{\pi'}$ and $h\in\h_{\Lambda}$.

First assume that $\ell=0$. Then by \eqref{ch} of Sect. \ref{bounds}, $-(\varpi_{n-1}+\varpi_n)$ must be a weight of $Sy(\p)$, hence $(\varpi_{n-1}+\varpi_n)(h)=0$ by definition of the canonical truncation (see \ref{CT}). It follows that the set $S$ cannot contain $\beta_{\pi}^0$, that is cannot contain all $\beta_i$, for $1\le i\le (n-1)/2$, as in the case $n$ even. Indeed one has that $\varpi_{n-1}+\varpi_n=\ep_1+\ldots+\ep_{n-1}=\beta_1+\ldots+\beta_{(n-1)/2}$ and then otherwise we would have both $(\varpi_{n-1}+\varpi_n)(h)=0$ and $(\varpi_{n-1}+\varpi_n)(h)=(-1)\times(n-1)/2$, a contradiction.

Assume now that $\ell\ge 1$. Then by \eqref{ch} of Sect. \ref{bounds}, for all $1\le k\le\ell$, $-2\varpi_{n-1-2k}$ must be a weight of $Sy(\p)$, hence by the same argument as above, we cannot have that $S$ contains $\beta_1,\ldots,\,\beta_{(n-1-2k)/2}$. 

\end{Rq}

\section{Case \ref{cas7} for type D.}\label{case7}

Here we consider the parabolic subalgebra $\p_0$ of $\g$ of type ${\rm D}_n$, with $n\ge 5$, $n$ odd. This is the parabolic subalgebra of $\g$ associated with the subset $\pi'=\pi\setminus\{\alpha_{n-1},\,\alpha_n\}$. Then it is the case \ref{cas7} of subsection \ref{Mainres}.

We set $S=S^+\sqcup S^-$ with
$$\begin{array}{cc}
S^+=\{\beta_i=\ep_{2i-1}+\ep_{2i}\,;\;1\le i\le (n-3)/2,\;\tilde\beta_{(n-1)/2}=\ep_{n-2}+\ep_n,\\
\,\tilde\beta_{(n+1)/2}=\ep_{n-2}-\ep_n\}\end{array}$$
and
$$S^-=\{-\beta'_1=\ep_{n-1}-\ep_1,\,-\tilde\beta'_i=\ep_{n-i-1}-\ep_i\,;\,2\le i\le (n-3)/2\}.$$

For all $1\le i\le (n-3)/2$, we set $\Gamma_{\beta_i}=H_{\beta_i}$ and we set $\Gamma_{-\beta'_1}=-H_{\beta'_1}$, where $H_{\beta_i}$, resp. $H_{\beta'_1}$, is the largest Heisenberg set with centre $\beta_i\in\beta_{\pi}$, resp. $\beta'_1\in\beta_{\pi'}$, which is included in $\Delta^+$, resp. in $\Delta^+_{\pi'}$, as defined in subsection \ref{HS}.

We set $$\Gamma_{\tilde\beta_{(n-1)/2}}=\{\tilde\beta_{(n-1)/2},\,\ep_{n-2}-\ep_{n-1},\,\ep_{n-1}+\ep_n\},\;
\Gamma_{\tilde\beta_{(n+1)/2}}=\{\tilde\beta_{(n+1)/2}\}.$$

For all $2\le i\le (n-3)/2$, we set 
$$\Gamma_{-\tilde\beta'_i}=\{-\tilde\beta'_i,\,\ep_{n-i-1}-\ep_j,\,\ep_j-\ep_i\,;\;i+1\le j\le n-i-2\}.$$
We also set
$$T=\{\ep_{n-2}-\ep_2,\,\ep_{n-2}+\ep_{n-1},\,\ep_{n-1}-\ep_n,\,\ep_{2i-1}-\ep_{2i}\,;\,1\le i\le (n-3)/2\}$$
and
$$T^*=\{\ep_{n-2}-\ep_i\,;\;3\le i\le n-3\}$$
By construction for all $\gamma\in S^+$, resp. $\gamma\in S^-$, we have that $\Gamma_{\gamma}\subset\Delta^+$, resp. $\Gamma_{\gamma}\subset\Delta^-_{\pi'}$, is a Heisenberg set with centre $\gamma$ and all the sets $\Gamma_{\gamma}$, for $\gamma\in S$, together with the sets $T$ and $T^*$ are disjoint. We easily verify that condition $(iv)$ of Prop. \ref{propAP} is satisfied, using $i)$ of Lemma \ref{HS}. Conditions $(ii)$ and $(iii)$ of Prop. \ref{propAP} follow easily from $ii)$, $iii)$ and $iv)$ of Lemma \ref{HS}.

\subsection{Condition {\it (i)} of Proposition \ref{propAP}.}

This condition follows from the lemma below.

\begin{lm}
$S_{\mid\h_{\Lambda}}$ is a basis for $\h_{\Lambda}^*$.

\end{lm}

\begin{proof}
First we observe that $\h_{\Lambda}=\h'\oplus\mathcal H^{-1}(\ep_n)$, where recall $\mathcal H:\h\longrightarrow\h^*$ is the isomorphism induced by the Killing form on $\h\times\h$ by Sect. \ref{not}. 

Then $\dim\h_{\Lambda}=\dim\h'+1=n-1$. We have that $\lvert S\rvert=n-1$.

Now set $s_i=\beta_i$  for $1\le i\le(n-3)/2$, $s_{(n-1)/2}=\ep_n$, $s_{(n+1)/2}=\ep_{n-2}$, $s_{(n+3)/2}=-\beta'_1$ and then $s_{(n+3)/2+j}=-\tilde\beta'_{j+1}$ for all $1\le j\le (n-5)/2$ and we take the elements of $S$ in this order.

For a basis $(h_j)$ of $\h_{\Lambda}$ we take, in this order :
$$\alpha_{2i}^\vee,\,1\le i\le (n-3)/2,\,\mathcal H^{-1}(\ep_n),\alpha_{n-2}^\vee,\,\alpha_{2j-1}^\vee,\,\alpha_{n-2-2j}^\vee,\,1\le j\le[(n-1)/4]$$
without repetitions.
Then it suffices to prove that $\det(s_i(h_j))_{1\le i,\,j\le n-1}\neq 0$.

We easily check that $(s_i(h_j))=\begin{pmatrix} A&0\\
*&B\\
\end{pmatrix}$ where $A$, resp. $B$, is a lower triangular matrix of size $(n+1)/2$, resp. $(n-3)/2$, with 1, resp. $-1$, on the diagonal. Hence the lemma.
\end{proof}

\subsection{Condition {\it (v)} of Proposition \ref{propAP}.}

Set $y=\sum_{\gamma\in S} x_{\gamma}$.
Condition $(v)$ of Prop. \ref{propAP} follows from the lemma below.

\begin{lm}
Let $k\in\mathbb N$ such that $3\le k\le n-3$. Then $x_{\ep_{n-2}-\ep_k}\in ad\,\p_{\Lambda}(y)+\g_T$.

\end{lm}

\begin{proof} Suppose first that $k$ is odd ($3\le k\le n-4$) and set $\gamma_1=\ep_{n-2}+\ep_n\in S$, $\gamma'_1=\ep_{k+1}-\ep_{n-2}\in\Delta^+_{\pi'}\setminus S$, $\gamma_2=\ep_{n-2}-\ep_n\in S$, $\gamma'_2=\ep_n-\ep_k\in\Delta^-\setminus S$, $\gamma_3=\ep_{k+1}+\ep_k\in S$, $\gamma'_3=-\ep_k-\ep_n\in\Delta^-\setminus S$.
We will show that the hypotheses of Lemma and Proposition \ref{condition(v)} are satisfied.
We have that $\gamma_1+\gamma'_1=\ep_{k+1}+\ep_n\in\Delta^+\setminus S$, $\gamma_2+\gamma'_2=\ep_{n-2}-\ep_k\in\Delta^-_{\pi'}\setminus S$,
$\gamma_3+\gamma'_3=\ep_{k+1}-\ep_n\in\Delta^+\setminus S$. Moreover $\gamma_2+\gamma'_2=\gamma_1+\gamma'_3$,
$\gamma_3+\gamma'_3=\gamma_2+\gamma'_1$, $\gamma_1+\gamma'_1+\gamma_2\in\Delta$ and $\gamma_1+\gamma_2\not\in\Delta$, $\gamma_2+\gamma_3\not\in\Delta$, $\gamma_1+\gamma_3\not\in\Delta$. Hence, by Lemma \ref{condition(v)}, up to rescaling some root vectors in a complement of $\g_S$ in $\g$, we have that

$$\begin{cases}ad\,x_{\ep_{k+1}-\ep_{n-2}}(y)=x_{\ep_{k+1}+\ep_n}+x_{\ep_{k+1}-\ep_n}+X\\
ad\,x_{\ep_n-\ep_k}(y)=x_{\ep_{k+1}+\ep_n}+x_{\ep_{n-2}-\ep_k}\\
ad\,x_{-\ep_k-\ep_n}(y)=x_{\ep_{k+1}-\ep_n}+x_{\ep_{n-2}-\ep_k}\\
\end{cases}$$
with $X=x_{\ep_{n-2-k}-\ep_{n-2}}=ad\,x_{-\ep_{n-k-3}-\ep_{n-2}}(y)\in ad\,\p_{\Lambda}(y)+\g_T$ if $3\le k\le(n-5)/2$, and $X=0$ otherwise. Hence $x_{\ep_{n-2}-\ep_k}\in ad\,\p_{\Lambda}(y)+\g_T$ for $k$ odd, $3\le k\le n-4$. A similar computation shows that 
$x_{\ep_{n-2}-\ep_k}\in ad\,\p_{\Lambda}(y)+\g_T$ for $k$ even, $4\le k\le n-3$.
\end{proof}

\subsection{Condition {\it (vi)} of Proposition \ref{propAP}.}\label{Tcase7}

It follows from lemma below.

\begin{lm}
We have that $\lvert T\rvert=\ind\p_{\Lambda}=\lvert E(\pi')\rvert$.

\end{lm}

\begin{proof}
The set $E(\pi')$ of $\langle ij\rangle$-orbits in $\pi$ is the following :
$$\begin{array}{cc}
E(\pi')=\Bigl\{\Gamma_u=\{\alpha_u,\,\alpha_{n-1-u}\},\,1\le u\le (n-3)/2,\,\Gamma_{(n-1)/2}=\{\alpha_{(n-1)/2}\},\\
\Gamma_{n-1}=\{\alpha_{n-1}\},\,\Gamma_n=\{\alpha_n\}\Bigr\}\end{array}$$

Then $\lvert E(\pi')\rvert=(n-3)/2+3=(n+3)/2$ and it is equal to $\lvert T\rvert$.
\end{proof}

\subsection{The semisimple element of the adapted pair.}
All conditions of Proposition \ref{propAP} are satisfied. Hence $y=\sum_{\gamma\in S}x_{\gamma}$ is regular in $\p_{\Lambda}^*$ and there exists a uniquely defined $h\in\h_{\Lambda}$ such that $(h,\,y)$ is an adapted pair for $\p_{\Lambda}$.

Below we give the semisimple element $h$ :

$$\begin{array}{cc}
h=-\sum_{k=1}^{(n-3)/2}k\alpha_{2k}^\vee+\sum_{k=1}^{[(n-1)/4]}((n-1)/2+k)\alpha_{2k-1}^\vee+\\
\sum_{k=[(n-1)/4]+1}^{(n-3)/2}(3(n-1)/2+1-3k)\alpha_{2k-1}^\vee-((n-1)/2)\alpha_{n-2}^\vee\in\h'\subset\h_{\Lambda}\\
\end{array}$$

\subsection{Computation of the improved upper bound.}
However for $n\ge 7$, both bounds in \eqref{ch} of Sect. \ref{bounds} do not coincide since for $\Gamma=\{\alpha_2,\,\alpha_{n-3}\}\in E(\pi')$ one has $d_{\Gamma}=\varpi_2+\varpi_{n-3}\in\mathcal B_{\pi}$ and $d'_{\Gamma}=\varpi'_2+\varpi'_{n-3}\in\mathcal B_{\pi'}$, hence $\ep_{\Gamma}=1/2$. We then need to compute the improved upper bound mentioned in Sect. \ref{IUB}.

\begin{lm}
We have that
$${\rm ch}(Y(\p_{\Lambda}))=\Bigl(1-e^{-2(\varpi_{n-1}+\varpi_n)}\Bigr)^{-(n-3)/2}\Bigl(1-e^{-(\varpi_{n-1}+\varpi_n)}\Bigr)^{-3}$$

\end{lm}

\begin{proof}
It suffices to prove that \eqref{eqIUB} of Sect. \ref{IUB} holds. Recall the $\langle ij\rangle$-orbits computed in the proof of Lemma \ref{Tcase7} and the lower bound for ${\rm ch}(Y(\p_{\Lambda}))$ given by \eqref{ch} in Sect. \ref{bounds}, with the weights $\delta_{\Gamma}$, for all $\Gamma\in E(\pi')$, given by \eqref{poids}. For $1\le u\le(n-3)/2$, we have that $\delta_{\Gamma_u}=-2(\varpi_u+\varpi_{n-1-u})+2(\varpi'_u+\varpi'_{n-1-u})=-2(\varpi_{n-1}+\varpi_n)$. Then $\delta_{\Gamma_{(n-1)/2}}=-2\varpi_{(n-1)/2}+2\varpi'_{(n-1)/2}=-(\varpi_{n-1}+\varpi_n)$. Finally observe that $j(\Gamma_{n-1})=\Gamma_n$ and then $\delta_{\Gamma_{n-1}}=-(\varpi_{n-1}+\varpi_n)=\delta_{\Gamma_n}$. It follows that the lower bound for ${\rm ch}(Y(\p_{\Lambda}))$ is equal to the right hand side of equality in the lemma. Now we have to compute the improved upper bound and for this purpose we have to compute, for all $\gamma\in T$, the $s(\gamma)\in\mathbb Q S$ such that $\gamma+s(\gamma)$ vanishes on $\h_{\Lambda}$, that is, we have to determine $s(\gamma)\in\mathbb Q S$ such that $\gamma+s(\gamma)=k(\varpi_{n-1}+\varpi_n)$ for some $k\in\mathbb Q$ (in fact $k\in\mathbb N$). Recall the sets $S$ and $T$  given in the beginning of this Section. For $1\le i\le(n-3)/2$, set $\gamma_i=\ep_{2i-1}-\ep_{2i}$. Assume first that $1\le i\le[(n-1)/4]$. Then one checks that
$$\begin{array}{cl} 
s(\gamma_i)=&2(\ep_{n-1}-\ep_1)+(\ep_{n-2}-\ep_n)+(\ep_{n-2}+\ep_n)+\\
&2\sum_{j=2}^{2i-1}(\ep_{n-j-1}-\ep_j)+4\sum_{j=1}^{i-1}(\ep_{2j-1}+\ep_{2j})+\\
&2\sum_{j=i+1}^{(n-1)/2-i}(\ep_{2j-1}+\ep_{2j})++3(\ep_{2i-1}+\ep_{2i})\in\mathbb N S\end{array}$$
so that $\gamma_i+s(\gamma_i)=2(\varpi_{n-1}+\varpi_n)$.

Now assume that $[(n-1)/4]<i\le(n-3)/2$. Then one checks that
$$\begin{array}{cl} s(\gamma_i)=&2\sum_{j=2}^{n-1-2i}(\ep_{n-1-j}-\ep_j)+4\sum_{j=1}^{(n-1)/2-i}(\ep_{2j-1}+\ep_{2j})+\\
&2\sum_{j=(n-1)/2-i+1}^{i-1}(\ep_{2j-1}+\ep_{2j})+(\ep_{2i-1}+\ep_{2i})+\\
&2(\ep_{n-1}-\ep_1)+(\ep_{n-2}-\ep_n)+(\ep_{n-2}+\ep_n)\in\mathbb N S\end{array}$$
so that $\gamma_i+s(\gamma_i)=2(\varpi_{n-1}+\varpi_n)$.

For $\gamma=\ep_{n-2}-\ep_2\in T$, one checks that $s(\gamma)=2(\ep_1+\ep_2)+(\ep_3+\ep_4)+\ldots+(\ep_{n-4}+\ep_{n-3})+(\ep_{n-1}-\ep_1)\in\mathbb N S$ so that $\gamma+s(\gamma)=\varpi_{n-1}+\varpi_n$.

For $\gamma=\ep_{n-2}+\ep_{n-1}\in T$, one checks that $s(\gamma)=(\ep_1+\ep_2)+(\ep_3+\ep_4)+\ldots+(\ep_{n-4}+\ep_{n-3})\in\mathbb N S$ so that $\gamma+s(\gamma)=\varpi_{n-1}+\varpi_n$.

Finally for $\gamma=\ep_{n-1}-\ep_n\in T$, one checks that $s(\gamma)=(\ep_1+\ep_2)+(\ep_3+\ep_4)+\ldots+(\ep_{n-4}+\ep_{n-3})+(\ep_{n-2}+\ep_n)\in\mathbb N S$ so that $\gamma+s(\gamma)=\varpi_{n-1}+\varpi_n$.

We deduce that the improved upper bound is equal to the right hand side of equality in the lemma. Hence the lemma, by what we said in Sect. \ref{IUB}. 
\end{proof}

\subsection{Existence of a Weierstrass section for coadjoint action.}

By what we said in Sect. \ref{IUB} (see also Remark \ref{rqeqIUB} of subsection \ref{REG}) we have the following Theorem.

\begin{thm}
Let $\g$ be a simple Lie algebra of type ${\rm D}_n$, with $n\ge 5$, $n$ odd, and $\p$ be the standard parabolic subalgebra of $\g$ associated with the subset $\pi'=\pi\setminus\{\alpha_{n-1},\,\alpha_n\}$ of the set $\pi$ of simple roots of $\g$. Then there exists a Weierstrass section for coadjoint action of the canonical truncation $\p_{\Lambda}$ of $\p$ and it follows that the algebra of symmetric invariants $Y(\p_{\Lambda})$ is a polynomial algebra over $\Bbbk$.
\end{thm}

\subsection{Weights and degrees of a set of generators.}\label{wdcas8}

By what we said in Sect. \ref{IUB} to each $\gamma\in T$ is associated an element $p_{\gamma}$ such that $\{p_{\gamma};\;\gamma\in T\}$ is a set of algebraically independent homogeneous and $\h$-weight generators of the polynomial algebra $Y(\p_{\Lambda})$. Moreover for all $\gamma\in T$, $p_{\gamma}$ has  a weight $wt(p_{\gamma})$ equal to $-(\gamma+s(\gamma))$ and a degree $\deg(p_{\gamma})$ equal to $1+\lvert s(\gamma)\rvert$. Below we give the weight $wt(p_{\gamma})$ and the degree $\deg(p_{\gamma})$ of $p_{\gamma}$, for all $\gamma\in T$. Set, for all $1\le i\le(n-3)/2$, $\gamma_i=\ep_{2i-1}-\ep_{2i}$.\bigskip

\begin{tabular}{|c|c|c|}
\hline
$\gamma\in T$&$wt(p_{\gamma})$&$\deg(p_{\gamma})$\\
\hline
$\gamma_i$, $1\le i\le[(n-1)/4]$&$-2(\varpi_{n-1}+\varpi_n)$&$n-1+4i$\\
\hline
$\gamma_i$, $[(n-1)/4]+1\le i\le(n-3)/2$&$-2(\varpi_{n-1}+\varpi_n)$&$3n-4i-1$\\
\hline
$\gamma=\ep_{n-2}-\ep_2$&$-(\varpi_{n-1}+\varpi_n)$&$(n+3)/2$\\
\hline
$\gamma=\ep_{n-2}+\ep_{n-1}$&$-(\varpi_{n-1}+\varpi_n)$&$(n-1)/2$\\
\hline
$\gamma=\ep_{n-1}-\ep_n$&$-(\varpi_{n-1}+\varpi_n)$&$(n+1)/2$\\
\hline
\end{tabular}

\section{Case \ref{cas8} for type D.}\label{Cas8}

In this Section we consider a simple Lie algebra $\g$ of type ${\rm D}_n$, $n\ge 5$, $n$ odd and the standard parabolic subalgebra $\p=\p_1$  associated with $\pi'=\pi\setminus\{\alpha_{n-3},\,\alpha_{n-1},\,\alpha_n\}$. It corresponds to the case \ref{cas8} of subsection \ref{Mainres}. As in previous case, both bounds in \eqref{ch} of Sect. \ref{bounds} do not coincide. Hence the existence of an adapted pair for $\p_{\Lambda}$ will not produce immediately a Weierstrass section for coadjoint action of $\p_{\Lambda}$. We will have to compute the improved upper bound mentioned in Sect. \ref{IUB} and show that the latter coincides with the lower bound in \eqref{ch}, namely that equality \eqref{eqIUB} holds. Then by Remark \ref{rqeqIUB} of subsection \ref{REG} this will produce a Weierstrass section for coadjoint action of $\p_{\Lambda}$.

Recall the elements $\beta_i=\ep_{2i-1}+\ep_{2i}$ of the Kostant cascade $\beta_{\pi}$ of $\g$.
We set $S=S^+\sqcup S^-$ with

$$S^+=\{\beta_i,\,1\le i\le (n-5)/2,\,\ep_{n-4}+\ep_{n-2},\,\ep_{n-3}+\ep_n,\,\ep_{n-3}-\ep_n\}$$
and
$$S^-=\{\ep_{n-3-k}-\ep_k,\,1\le k\le (n-5)/2\}.$$

For all $1\le i\le (n-5)/2$, we set $\Gamma_{\beta_i}=H_{\beta_i}$ the largest Heisenberg set with centre $\beta_i$ which is included in $\Delta^+$, as defined in subsection \ref{HS}.

We also set 
$$\begin{array}{cl}\Gamma_{\ep_{n-4}+\ep_{n-2}}=&\{\ep_{n-4}+\ep_{n-2},\,\ep_{n-4}+\ep_{n-1},\,\ep_{n-2}-\ep_{n-1},\\
&\ep_{n-4}+\ep_n,\,\ep_{n-2}-\ep_n,\,\ep_{n-4}-\ep_{n-1},\,\ep_{n-2}+\ep_{n-1},\\
&\ep_{n-4}-\ep_n,\,\ep_{n-2}+\ep_n,\,\ep_{n-4}-\ep_{n-3},\,\ep_{n-2}+\ep_{n-3}\},\end{array}$$

$$\Gamma_{\ep_{n-3}+\ep_n}=\{\ep_{n-3}+\ep_n,\,\ep_{n-3}-\ep_{n-1},\,\ep_{n-1}+\ep_n\},$$
$$\Gamma_{\ep_{n-3}-\ep_n}=\{\ep_{n-3}-\ep_n\}$$
and for all $1\le k\le (n-5)/2$,
$$\Gamma_{\ep_{n-3-k}-\ep_k}=\{\ep_{n-3-k}-\ep_k,\,\ep_{n-3-k}-\ep_j,\,\ep_j-\ep_k;\;k+1\le j\le n-4-k\}.$$

Finally we set 

$$T^*=\{\ep_{n-3}-\ep_k;\;1\le k\le n-2,\,k\neq n-3\}$$
and 
$$\begin{array}{cc}T=\{\ep_{2i-1}-\ep_{2i};\;1\le i\le (n-5)/2,\,\ep_{n-4}+\ep_{n-3},\\
\ep_{n-4}-\ep_{n-2},\,\ep_{n-3}+\ep_{n-1},\,\ep_{n-1}-\ep_n,\,\ep_{n-1}-\ep_{n-2}\}.\end{array}$$

By construction for all $\gamma\in S^+$, resp. $\gamma\in S^-$, we have that $\Gamma_{\gamma}\subset\Delta^+$, resp. $\Gamma_{\gamma}\subset\Delta^-_{\pi'}$, is a Heisenberg set with centre $\gamma$ and all the sets $\Gamma_{\gamma}$, for $\gamma\in S$, together with the sets $T$ and $T^*$ are disjoint. We easily verify that condition $(iv)$ of Prop. \ref{propAP} is satisfied, using $i)$ of Lemma \ref{HS}. Conditions $(ii)$ and $(iii)$ of Prop. \ref{propAP} follow easily from $ii)$, $iii)$ and $iv)$ of Lemma \ref{HS}.

\subsection{Conditions {\it (i)} of Proposition \ref{propAP}.}

We have the following lemma.

\begin{lm}
We have that $S_{\mid\h_{\Lambda}}$ is a basis for $\h_{\Lambda}^*$.
\end{lm}

\begin{proof}
First as in previous Section, one has that $\dim\h_{\Lambda}=\dim\h'+1$ since $\h_{\Lambda}=\h'\oplus\mathcal H^{-1}(\varpi_n-\varpi_{n-1})=\h'\oplus\mathcal H^{-1}(\ep_n)$ by Section \ref{not}. We check that $\lvert S\rvert=n-2=\dim\h_{\Lambda}$.

Set $s_i=\beta_i$ for all $1\le i\le (n-5)/2$, $s_{(n-3)/2}=\ep_{n-3}$, $s_{(n-1)/2}=\ep_n$, $s_{(n+1)/2}=\ep_{n-4}+\ep_{n-2}$,
$s_{(n+1)/2+k}=\ep_{n-3-k}-\ep_k$, $1\le k\le (n-5)/2$ and we take the elements of $S$ in this order.

For a basis $(h_v)$ of $\h_{\Lambda}$, we take in this order, 

$$\begin{array}{cc}\alpha_{2j}^\vee;\,1\le j\le (n-5)/2,\,\alpha_{n-4}^\vee,\,\mathcal H^{-1}(\ep_n),\,\alpha_{n-2}^\vee,\\
\,\alpha_{2k-1}^\vee,\,\alpha_{n-4-2k}^\vee;\,1\le k\le [(n-3)/4]\end{array}$$
without repetitions for the last coroots.

Then it suffices to show that $\det(s_u(h_v))_{1\le u,\,v\le n-2}\neq 0$.

One can easily verify that

$$(s_u(h_v))=\begin{pmatrix}A&0&0&0&0\\
*&-1&0&0&0\\
*&*&1&0&0\\
*&*&*&1&0\\
*&*&*&*&B\end{pmatrix}$$

where $A$, resp. $B$, is a lower triangular square matrix of size $(n-5)/2$, with one, resp. $-1$, on its diagonal. Hence the lemma.
\end{proof}

\subsection{Condition {\it (v)} of Proposition \ref{propAP}.}

It follows from lemma below. Set $y=\sum_{\alpha\in S} x_{\alpha}$.

\begin{lm}

For all $1\le k\le n-2$, $k\neq n-3$, we have that $x_{\ep_{n-3}-\ep_k}\in ad\,\p_{\Lambda}(y)+\g_T$.

\end{lm}

\begin{proof}

We will use Lemma and Proposition of subsection \ref{condition(v)}. First assume that $k=n-2$ and set $$\begin{cases}\gamma_1=\ep_{n-4}+\ep_{n-2}\in S,\,\gamma'_1=-\ep_n-\ep_{n-2}\not\in S\\
\gamma_2=\ep_{n-3}+\ep_n\in S,\,\gamma'_2=\ep_{n-4}-\ep_{n-3}\not\in S\\
\gamma_3=\ep_{n-3}-\ep_n\in S,\,\gamma'_3=\ep_n-\ep_{n-2}\not\in S\\
\end{cases}$$

One checks easily that all conditions of Lemma \ref{condition(v)} are satisfied. Moreover since one can take the vectors $X,\,X',\,X''$ in Prop. \ref{condition(v)} equal to zero, one deduces that $x_{\ep_{n-3}-\ep_{n-2}}\in ad\,\p_{\Lambda}(y)+\g_T$.

Assume now that $k=n-4$ and set $$\begin{cases}\gamma_1=\ep_{n-3}+\ep_n\in S,\,\gamma'_1=-\ep_n-\ep_{n-4}\not\in S\\
\gamma_2=\ep_{n-2}+\ep_{n-4}\in S,\,\gamma'_2=\ep_n-\ep_{n-4}\not\in S\\
\gamma_3=\ep_{n-3}-\ep_n\in S,\,\gamma'_3=\ep_{n-2}-\ep_{n-3}\not\in S\\
\end{cases}$$

One checks that all conditions of Lemma \ref{condition(v)} are satisfied. Moreover since one can take the vectors $X,\,X',\,X''$ in Prop. \ref{condition(v)} equal to zero, one deduces that $x_{\ep_{n-3}-\ep_{n-4}}\in ad\,\p_{\Lambda}(y)+\g_T$.

Assume that $1\le k\le n-6$, $k$ odd, and set
$$\begin{cases} \gamma_1=\ep_{n-3}+\ep_n\in S,\,\gamma'_1=-\ep_k-\ep_n\not\in S\\
\gamma_2=\ep_k+\ep_{k+1}\in S,\,\gamma'_2=\ep_n-\ep_k\not\in S\\
\gamma_3=\ep_{n-3}-\ep_n\in S,\,\gamma'_3=\ep_{k+1}-\ep_{n-3}\not\in S\\
\end{cases}$$

One checks that all conditions of Lemma \ref{condition(v)} are satisfied. Moreover since one can take in Prop. \ref{condition(v)}, $X=X'=0$ and $X''=x_{\ep_{n-4-k}-\ep_{n-3}}=ad\,x_{-\ep_{n-5-k}-\ep_{n-3}}(y)$ if $k\le(n-7)/2$, $X''=0$ otherwise, one deduces that $x_{\ep_{n-3}-\ep_k}\in ad\,\p_{\Lambda}(y)+\g_T$.

A similar computation for $2\le k\le n-5$, $k$ even, shows that $x_{\ep_{n-3}-\ep_k}\in ad\,\p_{\Lambda}(y)+\g_T$.
\end{proof}

\subsection{Condition {\it (vi)} of Proposition \ref{propAP}.}\label{Tcas8} It follows from lemma below.

\begin{lm}
One has that $\lvert T\rvert=\ind\p_{\Lambda}=\lvert E(\pi')\rvert$.

\end{lm}
\begin{proof}
Recall that $E(\pi')$ is the set of $\langle ij\rangle$-orbits in $\pi$. One easily checks that

$$\begin{array}{cc}E(\pi')=\Bigl\{\Gamma_u=\{\alpha_u,\,\alpha_{n-3-u}\};\;1\le u\le (n-5)/2,\,\Gamma_{(n-3)/2}=\{\alpha_{(n-3)/2}\},\\
\,\Gamma_{n-3}=\{\alpha_{n-3}\},\,\Gamma_{n-2}=\{\alpha_{n-2}\},\,\Gamma_{n-1}=\{\alpha_{n-1}\},\,\Gamma_n=\{\alpha_n\}\Bigr\}\end{array}$$
Hence $\ind\p_{\Lambda}=\lvert E(\pi')\rvert=(n-5)/2+5=(n+5)/2$, which is equal to $\lvert T\rvert$ (see beginning of this Section).
\end{proof}

\subsection{The semisimple element of the adapted pair.}

All conditions of Proposition \ref{propAP} are satisfied. Hence $y=\sum_{\gamma\in S} x_{\gamma}$ is regular in $\p_{\Lambda}^*$ and there exists a uniquely defined semisimple element $h\in\h_{\Lambda}$ such that $ad\,h(y)=-y$, namely such that $(h,\,y)$ is an adapted pair for $\p_{\Lambda}$. Below we give the semisimple element $h$ :

$$\begin{array}{cc}
h=-\sum_{k=1}^{(n-5)/2}k\alpha_{2k}^\vee+\sum_{k=1}^{[(n-3)/4]}((n-3)/2+k)\alpha_{2k-1}^\vee+\\
\sum_{k=[(n-3)/4]+1}^{(n-5)/2}(3(n-3)/2+1-3k)\alpha_{2k-1}^\vee+\\
\alpha_{n-4}^\vee-((n-1)/2)\alpha_{n-2}^\vee\in\h'\subset\h_{\Lambda}\\
\end{array}$$

\subsection{Computation of the improved upper bound.}

Here both bounds in \eqref{ch} of Sect. \ref{bounds} do not coincide since, for $\Gamma=\Gamma_{n-3}\in E(\pi')$, we have that $\ep_{\Gamma_{n-3}}=1/2$ (recall  \eqref{epsilon} of subsection \ref{compepsilon}). Indeed by subsection \ref{compepsilon}, we have that $d_{\Gamma_{n-3}}=\varpi_{n-3}\in\mathcal B_{\pi}$ and $d'_{\Gamma_{n-3}}=0\in\mathcal B_{\pi'}$. Hence the existence of an adapted pair for $\p_{\Lambda}$ is not sufficient to assure the existence of a Weierstrass section for coadjoint action of $\p_{\Lambda}$. We will show below that \eqref{eqIUB} of Sect. \ref{IUB} holds and by what we said in Sect. \ref{IUB} it will be sufficient to provide a Weierstrass section.

\begin{lm}
We have that
$$\begin{array}{cc}
{\rm ch}\, (Y(\p_{\Lambda}))=(1-e^{-2\varpi_{n-3}})^{-(n-3)/2}\times(1-e^{-\varpi_{n-3}})^{-1}\times\\
(1-e^{-(\varpi_{n-3}+\varpi_{n-1}+\varpi_n)})^{-1}\times
(1-e^{-(\varpi_{n-1}+\varpi_n)})^{-2}\end{array}$$

\end{lm}

\begin{proof}

We will prove that the improved upper bound mentioned in Sect. \ref{IUB} is equal to the lower bound appearing in left hand side of \eqref{ch} of Sect. \ref{bounds}, namely that \eqref{eqIUB} of Sect. \ref{IUB} holds.

Recall that the lower bound for ${\rm ch}\, (Y(\p_{\Lambda}))$ is equal to $\prod_{\Gamma\in E(\pi')}(1-e^{\delta_{\Gamma}})^{-1}$ where $\delta_{\Gamma}$ is given by \eqref{poids} of subsection \ref{comppoids}.

Recall the set $E(\pi')$ computed in the proof of Lemma \ref{Tcas8} and that for all $\Gamma\in E(\pi')$ one has that $i(\Gamma\cap\pi')=j(\Gamma)\cap\pi'$ (by \ref{comppoids}). Then for $1\le u\le (n-5)/2$, and $\Gamma_u=\{\alpha_u,\,\alpha_{n-3-u}\}$, one has that
$$\begin{array}{ccl}\delta_{\Gamma_u}&=&-2(\varpi_u+\varpi_{n-3-u})+2(\varpi'_u+\varpi'_{n-3-u})\\
&=&-2(\ep_1+\ldots+\ep_u+\ep_1+\ldots+\ep_{n-3-u})+\\
&&2(\ep_1-\ep_{n-3}+\ldots+\ep_u-\ep_{n-2-u})\\
&=&-2\varpi_{n-3}.\end{array}$$ 
For $\Gamma_{(n-3)/2}=\{\alpha_{(n-3)/2}\}$, one has that 
$$\begin{array}{ccl}\delta_{\Gamma_{(n-3)/2}}&=&-2\varpi_{(n-3)/2}+2\varpi'_{(n-3)/2}\\
&=&-2(\ep_1+\ldots+\ep_{(n-3)/2})+\\
&&2(\ep_1-\ep_{n-3}+\ep_2-\ep_{n-4}+\ldots+\ep_{(n-3)/2}-\ep_{(n-1)/2})\\
&=&-(\ep_1+\ldots+\ep_{n-3})=-\varpi_{n-3}.\end{array}$$
Then for $\Gamma_{n-3}=\{\alpha_{n-3}\}$, one has that $\delta_{\Gamma_{n-3}}=-2\varpi_{n-3}$.
For $\Gamma_{n-2}=\{\alpha_{n-2}\}$, one has that 
$$\begin{array}{ccl}
\delta_{\Gamma_{n-2}}&=&-2\varpi_{n-2}+2\varpi'_{n-2}\\
&=&-2(\ep_1+\ldots+\ep_{n-2})+(\ep_{n-2}-\ep_{n-1})\\
&=&-2(\ep_1+\ldots+\ep_{n-3})-\ep_{n-2}-\ep_{n-1}\\
&=&-(\varpi_{n-3}+\varpi_{n-1}+\varpi_n)\\
\end{array}$$
Finally for $\Gamma_{n-1}=\{\alpha_{n-1}\}$ and for $\Gamma_n=\{\alpha_n\}=j(\Gamma_{n-1})$, one has that
$\delta_{\Gamma_{n-1}}=\delta_{\Gamma_n}=-(\varpi_{n-1}+\varpi_n)$. Hence the right hand side of equality of the lemma is equal to the lower bound for ${\rm ch}\, (Y(\p_{\Lambda}))$.

Now the improved upper bound for ${\rm ch}\, (Y(\p_{\Lambda}))$ is equal to $\prod_{\gamma\in T}(1-e^{-(\gamma+s(\gamma))})^{-1}$, where we have that $ad\,\p_{\Lambda}(y)\oplus\g_T=\p_{\Lambda}^*$ with $\dim\g_T= \ind\p_{\Lambda}$ and where, for all $\gamma\in T$, $s(\gamma)\in\mathbb Q S$ is such that $\gamma+s(\gamma)$ vanishes on $\h_{\Lambda}$, that is, $\gamma+s(\gamma)=k\varpi_{n-3}+k'(\varpi_{n-1}+\varpi_n)$, with $k,\,k'\in\mathbb Q$.

Set, for all $1\le i\le (n-5)/2$, $\gamma_i=\ep_{2i-1}-\ep_{2i}\in T$.

Assume first that $1\le i\le[(n-3)/4]$. Then one has that
$$\begin{array}{cl}
s(\gamma_i)=&(\ep_{n-3}-\ep_n)+(\ep_{n-3}+\ep_n)+2\sum_{j=1}^{2i-1}(\ep_{n-3-j}-\ep_j)+\\
&4\sum_{j=1}^{i-1} (\ep_{2j-1}+\ep_{2j})+2\sum_{j=i+1}^{(n-3)/2-i}(\ep_{2j-1}+\ep_{2j})+3(\ep_{2i-1}+\ep_{2i})\end{array}$$
so that $\gamma_i+s(\gamma_i)=2\varpi_{n-3}$.

Now for $[(n-3)/4]+1\le i\le (n-5)/2$, one has that

$$\begin{array}{ccl}
s(\gamma_i)&=&(\ep_{n-3}-\ep_n)+(\ep_{n-3}+\ep_n)+2\sum_{j=1}^{n-3-2i}(\ep_{n-3-j}-\ep_j)\\
&&+4\sum_{j=1}^{(n-3)/2-i} (\ep_{2j-1}+\ep_{2j})+2\sum_{j=(n-3)/2-i+1}^{i-1}(\ep_{2j-1}+\ep_{2j})\\
&&+(\ep_{2i-1}+\ep_{2i})\end{array}$$
so that $\gamma_i+s(\gamma_i)=2\varpi_{n-3}$.

For $\gamma=\ep_{n-4}+\ep_{n-3}\in T$, one has that $s(\gamma)=(\ep_1+\ep_2)+\ldots+(\ep_{n-6}+\ep_{n-5})$, so that
$\gamma+s(\gamma)=\varpi_{n-3}$.

For $\gamma=\ep_{n-4}-\ep_{n-2}\in T$, one has that $s(\gamma)=2((\ep_1+\ep_2)+\ldots+(\ep_{n-6}+\ep_{n-5}))+(\ep_{n-4}+\ep_{n-2})+(\ep_{n-3}+\ep_n)+(\ep_{n-3}-\ep_n)$ so that $\gamma+s(\gamma)=2\varpi_{n-3}$.

For $\gamma=\ep_{n-3}+\ep_{n-1}\in T$, one has that $s(\gamma)=(\ep_1+\ep_2)+\ldots+(\ep_{n-6}+\ep_{n-5})+(\ep_{n-4}+\ep_{n-2})$ so that $\gamma+s(\gamma)=\ep_1+\ep_2+\ldots+\ep_{n-3}+\ep_{n-2}+\ep_{n-1}=\varpi_{n-1}+\varpi_n$.

For $\gamma=\ep_{n-1}-\ep_n\in T$, one has that $s(\gamma)=(\ep_1+\ep_2)+\ldots+(\ep_{n-6}+\ep_{n-5})+(\ep_{n-4}+\ep_{n-2})+(\ep_{n-3}+\ep_n)$ so that $\gamma+s(\gamma)=\varpi_{n-1}+\varpi_n$.

Finally for $\gamma=\ep_{n-1}-\ep_{n-2}\in T$, one has that $s(\gamma)=2((\ep_1+\ep_2)+\ldots+(\ep_{n-6}+\ep_{n-5}))+2(\ep_{n-4}+\ep_{n-2})+(\ep_{n-3}+\ep_n)+(\ep_{n-3}-\ep_n)$ so that
$\gamma+s(\gamma)=2(\ep_1+\ldots+\ep_{n-3})+\ep_{n-2}+\ep_{n-1}=\varpi_{n-3}+\varpi_{n-1}+\varpi_n$.
Thus we obtain that the improved upper bound is also equal to the right hand side of the equality in the lemma, which gives the lemma, by what we said in Sect. \ref{IUB}.
\end{proof}

\subsection{Existence of a Weierstrass section.}

By the above (see also Remark \ref{rqeqIUB} of subsection \ref{REG}) one can deduce the following Theorem.

\begin{thm}

Let $\g$ be a simple Lie algebra of type ${\rm D}_n$, with $n\ge 5$, $n$ odd, and let $\p$ be a standard parabolic subalgebra of $\g$ associated with the subset of simple roots $\pi'=\pi\setminus\{\alpha_{n-3},\,\alpha_{n-1},\,\alpha_n\}$. Then there exists a Weierstrass section for coadjoint action of the canonical truncation $\p_{\Lambda}$ of $\p$ and it follows that $Sy(\p)=Y(\p_{\Lambda})$ is a polynomial algebra over $\Bbbk$.

\end{thm}

\subsection{Weights and degrees of a set of generators.}

As in subsection \ref{wdcas8} we give below the weights and degrees of each element of a set $\{p_{\gamma};\;\gamma\in T\}$ of homogeneous and $\h$-weight algebraically independent generators of the polynomial algebra $Y(\p_{\Lambda})$. Recall that the weight $wt(p_{\gamma})$ of $p_{\gamma}$ is equal to $-(\gamma+s(\gamma))$ and the degree $\deg(p_{\gamma})$ of $p_{\gamma}$ is equal to $1+\lvert s(\gamma)\rvert$ and that we set $\gamma_i=\ep_{2i-1}-\ep_{2i}$ for all $1\le i\le(n-5)/2$.
\bigskip

\begin{tabular}{|c|c|c|}
\hline
$\gamma\in T$&$wt(p_{\gamma})$&$\deg(p_{\gamma})$\\
\hline
$\gamma_i$ &$-2\varpi_{n-3}$&$n-3+4i$\\
$1\le i\le[(n-3)/4]$&&\\
\hline
$\gamma_i$&$-2\varpi_{n-3}$&$3n-4i-7$\\
 $[(n-3)/4]+1\le i\le(n-5)/2$&&\\
\hline
$\gamma=\ep_{n-4}+\ep_{n-3}$&$-\varpi_{n-3}$&$(n-3)/2$\\
\hline
$\gamma=\ep_{n-4}-\ep_{n-2}$&$-2\varpi_{n-3}$&$n-1$\\
\hline
$\gamma=\ep_{n-3}+\ep_{n-1}$&$-(\varpi_{n-1}+\varpi_n)$&$(n-1)/2$\\
\hline
$\gamma=\ep_{n-1}-\ep_n$&$-(\varpi_{n-1}+\varpi_n)$&$(n+1)/2$\\
\hline
$\gamma=\ep_{n-1}-\ep_{n-2}$&$-(\varpi_{n-3}+\varpi_{n-1}+\varpi_n)$&$n$\\
\hline
\end{tabular}

\section{Case \ref{cas9} for type D.}\label{Cas9}

Here we consider the parabolic subalgebra $\p=\q_{s,\,\ell}$ of $\g$ simple of type ${\rm D}_n$ with $n$ odd, $n\ge 5$, $s$ odd et $\ell\in\mathbb N$ such that $s+2\ell\le n-2$ (note that in this case one has that $s+2\ell\neq n-3$, hence it does not coincide with some $\p_{\ell'}$). This corresponds to the case \ref{cas9} of subsection \ref{Mainres} that is, the parabolic subalgebra $\q_{s,\,\ell}$ of $\g$ associated with $\pi'=\pi\setminus\{\alpha_s,\,\alpha_{s+2},\,\ldots,\,\alpha_{s+2\ell},\,\alpha_{n-1},\,\alpha_n\}$.

When $s+2\ell<n-2$, there exists a connected component of $\pi'$ of type ${\rm A}_{n-2-s-2\ell}$ which we denote by $\pi'_2$. Then, when moreover $s\ge 3$, there exist two connected components of $\pi'$ of type ${\rm A}_k$ with $k\ge 2$, namely $\pi'_1$ of type ${\rm A}_{s-1}$ and $\pi'_2$ above. Denote by $\beta_{\pi'_k}\subset\beta_{\pi'}$ the Kostant cascade of the simple factor of $\g'$ associated with $\pi'_k$ for $k\in\{1,\,2\}$. 
We have that $\beta_{\pi'_1}=\{\beta'_i=\ep_i-\ep_{s+1-i}\mid 1\le i\le (s-1)/2\}$ and $\beta_{\pi'_2}=\{\beta''_i=\ep_{s+2\ell+i}-\ep_{n-i}\mid 1\le i\le(n-s-2\ell-2)/2\}$. 
Recall that $\beta^0_{\pi'}=\beta_{\pi'}\setminus(\beta_{\pi'}\cap\pi')$. Then $\beta^0_{\pi'}=\beta_{\pi'_1}\cup\beta_{\pi'_2}$.
We also have that (see subsection \ref{KcasBD}) $\beta^0_{\pi}=\beta_{\pi}\setminus(\beta_{\pi}\cap\pi)=\{\beta_i=\ep_{2i-1}+\ep_{2i}\mid 1\le i\le (n-1)/2\}$.
We set $$S^+=\{\beta_i;\;1\le i\le (n-3)/2,\,\tilde\beta_{(n-1)/2}=\beta_{(n-1)/2}-\alpha_{n-1}=\ep_{n-2}+\ep_n\},$$
$$S^-=-\beta^0_{\pi'}$$ and $S=S^+\cup S^-$.

\subsection{Conditions {\it (i)} to {\it (v)} of Proposition \ref{propAP}.}

For all $\beta_i\in S^+$ with $1\le i\le (n-3)/2$, we set $\Gamma_{\beta_i}=H_{\beta_i}\subset\Delta^+$ 
and for all $\gamma\in S^-$, we set $\Gamma_{\gamma}=-H_{-\gamma}\subset\Delta^-_{\pi'}$ with the notation of subsection \ref{HS}. Finally we set $\Gamma_{\ep_{n-2}+\ep_n}=\{\ep_{n-2}+\ep_n,\,\ep_{n-2}-\ep_{n-1},\,\ep_{n-1}+\ep_n\}$, $T^+=\{\beta_{(n-1)/2},\,\ep_{n-2}-\ep_n,\,\ep_{n-1}-\ep_n,\,\ep_{2i-1}-\ep_{2i},\;1\le i\le (n-3)/2\}$ and $T^-=-(\beta_{\pi'}\cap\pi')$. 
By construction every set $\Gamma_{\gamma}$, for $\gamma\in S$,  is a Heisenberg set with centre $\gamma$ such that, if $\gamma\in S^+$, then $\Gamma_{\gamma}\subset\Delta^+$ and if $\gamma\in S^-$, then $\Gamma_{\gamma}\subset\Delta^-_{\pi'}$. Moreover, for all $1\le i\le (n-1)/2$, we have that $\ep_{2i-1}-\ep_{2i}=\alpha_{2i-1}\in\beta_\pi\cap\pi$ (see subsection \ref{KcasBD}) and $H_{\alpha_{2i-1}}=\{\alpha_{2i-1}\}$. We observe that $H_{\beta_{(n-1)/2}}\sqcup H_{\alpha_{n-2}}=\Gamma_{\ep_{n-2}+\ep_n}\sqcup (T^+\cap H_{\beta_{(n-1)/2}})$. Then by $i)$ of Lemma \ref{HS}, the sets $T^+$, $T^-$ and the $\Gamma_{\gamma}$'s, $\gamma\in S$, are disjoint and one has that $\Delta^+=\sqcup_{\gamma\in S^+}\Gamma_{\gamma}\sqcup T^+$ and $\Delta^-_{\pi'}=\sqcup_{\gamma\in S^-}\Gamma_{\gamma}\sqcup T^-$. Then by setting $T^*=\emptyset$, condition $(iv)$ of Prop. \ref{propAP} holds, with $T=T^+\sqcup T^-$. One also deduces by $ii)$, $iii)$ and $iv)$ of Lemma \ref{HS} that conditions $(ii)$ and $(iii)$ of Prop. \ref{propAP} are satisfied.
Condition $(i)$ follows from the following Lemma.

\begin{lm}

$S_{\mid_{\h_{\Lambda}}}$ is a basis for $\h_{\Lambda}^*$.

\end{lm} 

\begin{proof}

First we observe that $\h_{\Lambda}=\h'\oplus\Bbbk\mathcal H^{-1}(\ep_n)$ by what we said in Sect. \ref{not}. Hence $\dim\h_{\Lambda}=\dim\h'+1=n-\ell-3+1=n-\ell-2$.
We first verify that $\lvert S\rvert=(n-3)/2+1+(s-1)/2+(n-s-2\ell-2)/2=n-\ell-2=\dim\h_{\Lambda}$.

Then we order the elements $s_u$ of $S$ as follows :

$$\beta_1,\,\ldots,\,\beta_{(n-3)/2},\,-\beta'_1,\,\ldots,\,-\beta'_{(s-1)/2},\,-\beta''_1,\,\ldots,\,-\beta''_{(n-s-2\ell-2)/2},\,\ep_{n-2}+\ep_n$$
 Set $t=[(s+1)/4]$ and $t'=[(n-s-2\ell]/4]$.
For a basis $(h_v)$ of $\h_{\Lambda}$ we take, in this order,

$$\begin{array}{cc}\alpha_2^\vee,\,\alpha_4^\vee,\,\ldots,\,\alpha_{n-3}^\vee,\\
\,h'_1=\alpha_1^\vee,\,h'_2=\alpha_{s-2}^\vee,\ldots,h'_{2t-1}=\alpha_{2t-1}^\vee, h'_{2t}=\alpha_{s-2t}^\vee,\\
h''_1=\alpha_{n-2}^\vee,h''_2=\alpha_{s+2\ell+2}^\vee,\,\ldots,\,h''_{2t'-1}=\alpha_{n-2t'}^\vee,\,h''_{2t'}=\alpha_{s+2\ell+2t'}^\vee\\
\mathcal H^{-1}(\ep_n)\\
\end{array}$$

without repetitions for the $h'_j$'s and the $h''_j$'s.
Recall that $\beta_i=\varpi_{2i}-\varpi_{2i-2}$ (where $\varpi_0=0$) and Lemma \ref{KCA}. Then we obtain that

$$(s_u(h_v))_{1\le u,\,v\le n-2-\ell}=\begin{pmatrix} A&0&0&0\\
*&B&0&0\\
*&*&C&0\\
*&*&*&1\\
\end{pmatrix}$$
with $A=(\beta_u(\alpha_{2v}^\vee))_{1\le u,\,v\le (n-3)/2}$, resp. $B=(-\beta'_u(h'_v))_{1\le u,\,v\le(s-1)/2}$, and $C=(-\beta''_u(h''_v))_{1\le u,\,v\le (n-s-2\ell-2)/2}$, which are lower triangular matrices with $1$, resp. $-1$ on their diagonal. Hence the lemma.
\end{proof}

\subsection{Condition {\it (vi)} of Proposition \ref{propAP}.}\label{TDS}

Condition $(vi)$ of Prop. \ref{propAP} follows from the following Lemma.

\begin{lm}
We have that $\lvert T\rvert=\ind\p_{\Lambda}$.

\end{lm}

\begin{proof}
One easily checks that
$$\begin{array}{cc}
E(\pi')=\Bigl\{\Gamma_u=\{\alpha_u,\,\alpha_{s-u}\},\;1\le u\le (s-1)/2,\\
\,\Gamma_{s+2\ell+v}=\{\alpha_{s+2\ell+v},\,\alpha_{n-1-v}\},\,1\le v\le (n-s-2\ell-2)/2,\\
\;\Gamma_{s+k}=\{\alpha_{s+k}\},\,0\le k\le 2\ell,\\
\Gamma_{n-1}=\{\alpha_{n-1}\},\,\Gamma_n=\{\alpha_n\}\Bigr\}\end{array}$$
hence $\ind\p_{\Lambda}=\lvert E(\pi')\rvert=(s-1)/2+(n-s-2\ell-2)/2+2\ell+1+2=(n-3)/2+\ell+3=\lvert T^+\rvert+\lvert T^-\rvert=\vert T\rvert$.
\end{proof}

\subsection{} 

All conditions of Proposition \ref{propAP} are satisfied, hence setting $y=\sum_{\gamma\in S} x_{\gamma}$ and $h\in\h_{\Lambda}$ such that for all $\gamma\in S$,  $\gamma(h)=-1$, we obtain that $(h,\,y)$ is an adapted pair for $\p_{\Lambda}$. This is sufficient by Remark \ref{rqeqbounds} of subsection \ref{REG} to provide a Weierstrass section for coadjoint action of $\p_{\Lambda}$, by the following Lemma.

\begin{lm}

For every $\Gamma\in E(\pi')$, we have that $\varepsilon_{\Gamma}=1$.

\end{lm}

\begin{proof}

Recall the set $E(\pi')$ given in the proof of Lemma \ref{TDS}. Recall subsection \ref{compepsilon}. Set $1\le u\le (s-1)/2$. Then $d_{\Gamma_u}=\varpi_u+\varpi_{s-u}\not\in\mathcal B_{\pi}$ since $u$ and $s-u$ are of different parity. For the same reason, for $1\le v\le(n-s-2\ell-2)/2$, we have that $d_{\Gamma_{s+2\ell+v}}\not\in\mathcal B_{\pi}$. Now for $0\le k\le 2\ell$ and $k$ odd,
$d_{\Gamma_{s+k}}=\varpi_{s+k}\in\mathcal B_{\pi}$, but $d'_{\Gamma_{s+k}}=\varpi'_{s+k}\not\in\mathcal B_{\pi'}$ since $\alpha_{s+k}$ belongs to a connected component of $\pi'$ of type ${\rm A}_1$. If $0\le k\le 2\ell$ and $k$ even, then $\alpha_{s+k}\not\in\pi'$ and $d'_{\Gamma_{s+k}}=0\in\mathcal B_{\pi'}$ but $d_{\Gamma_{s+k}}=\varpi_{s+k}\not\in\mathcal B_{\pi}$. Finally $d_{\Gamma_{n-1}}\not\in\mathcal B_{\pi}$ and $d_{\Gamma_n}\not\in\mathcal B_{\pi}$. Hence the lemma.
\end{proof}

We then obtain the following Theorem.

\begin{thm}
Let $\g$ be a simple Lie algebra of type ${\rm D}_n$, with $n\ge 5$, $n$ odd and let $s,\,\ell$ be integers such that $s$ is odd and $s+2\ell\le n-2$. \par
Let $\q_{s,\,\ell}$ be the parabolic subalgebra of $\g$ associated with the subset $\pi'=\pi\setminus\{\alpha_s,\,\alpha_{s+2},\,\ldots,\,\alpha_{s+2\ell},\,\alpha_{n-1},\,\alpha_n\}$. Then there exists a Weierstrass section for coadjoint action of the canonical truncation of $\q_{s,\,\ell}$.

\end{thm}

\begin{proof}

Indeed (with the notation in Prop. \ref{propAP}) $y+\g_T$ is a Weierstrass section for coadjoint action of the canonical truncation of $\q_{s,\,\ell}$ by Remark \ref{rqeqbounds} of subsection \ref{REG}.
\end{proof}

\subsection{Weights and degrees}

Here both bounds (see \eqref{ch} in Sect. \ref{bounds}) for ${\ch}(Y(\p_{\Lambda}))$ coincide and then $Y(\p_{\Lambda})$ is a polynomial algebra whose homogeneous and $\h$-weight generators have weights and degrees which can be easily computed.
To each $\Gamma\in E(\pi')$ is associated an homogeneous and $\h$-weight  generator of $Y(\p_{\Lambda})$ which has weight $\delta_{\Gamma}$ given by \eqref{poids} and a degree $\partial_{\Gamma}$ given by \eqref{degre} or by \eqref{degrebis} of Sect. \ref{bounds}.

Below  we give for completeness weights and degrees of a set of homogeneous and $\h$-weight algebraically independent generators of $Y(\p_{\Lambda})$, each of them corresponding to an $\langle ij\rangle$-orbit $\Gamma_r$ in $E(\pi')$.\smallskip

\begin{tabular}{|c|c|c|}
\hline
$\langle ij\rangle$-orbit in $E(\pi')$& Weight & Degree\\
\hline
$\Gamma_u=\{\alpha_u,\,\alpha_{s-u}\}$&$-2\varpi_s$&$s+1+2u$\\
$1\le u\le (s-1)/2$&&\\
\hline
$\Gamma_{s+2\ell+v}=\{\alpha_{s+2\ell+v},\,\alpha_{n-1-v}\}$& $-2(\varpi_{n-1}+\varpi_n)$&$n+3s+6\ell+2v$\\
$1\le v\le(n-s-2\ell-2)/2$&&\\
\hline
$\Gamma_{s+k}=\{\alpha_{s+k}\}$& $-2\varpi_{s+k}$&$s+k+1$\\
$0\le k\le 2\ell$,\,$k$ even&&\\
\hline
$\Gamma_{s+k}=\{\alpha_{s+k}\}$& $-\varpi_{s+k-1}-\varpi_{s+k+1}$&$2(s+k)$\\
$1\le k\le 2\ell-1$,\,$k$ odd&&\\
\hline
$\Gamma_{n-1}=\{\alpha_{n-1}\}$& $-\varpi_{n-1}-\varpi_n$&$(n-1)/2$\\
\hline
$\Gamma_{n}=\{\alpha_n\}$&$-\varpi_n-\varpi_{n-1}$&$(n+1)/2$\\
\hline

\end{tabular}

\subsection{}

\begin{Rqs}\rm

\begin{enumerate}

\item Consider now the parabolic subalgebra $\p=\q_{s,\,\ell}$ of $\g$ of type ${\rm D}_n$, with $s$ an even integer and assume that we have found an adapted pair $(h,\,y)\in\h_{\Lambda}\times\p_{\Lambda}^*$ for $\p_{\Lambda}$. Then the set $S$ cannot contain, as in the case $s$ odd and $n$ odd, the set $\{\beta_i\mid 1\le i\le [(n-3)/2]\}$, at least for $s/2\le[(n-3)/2]$. Indeed by \eqref{ch} of Sect. \ref{bounds}, one has that $-2\varpi_s\in\Lambda(\p)$ then necessarily $\varpi_s(h)=0\iff\beta_1+\ldots+\beta_{s/2}=0$ in contradiction with the fact that, for all $1\le i\le [(n-3)/2]$, one should have also that $\beta_i(h)=-1$. Moreover for $s=n-2$ (with $s$ even), the set $S=\{\beta_i;\;1\le i\le (n-4)/2,\,\tilde\beta_{(n-2)/2}=\ep_{n-3}+\ep_{n-1}\}\cup(-\beta_{\pi'}^0)$ is such that $S_{\mid\h_{\Lambda}}$ is not in general a basis for $\h_{\Lambda}^*$ (since for all $s\in S$, $s(\alpha_{n/2-1}^\vee)=0$
for $n\ge 8$).

\item Now consider in $\g$ simple of type ${\rm D}_n$, the parabolic subalgebra $\p=\q_{s,\,\ell}$ with $s$ odd and $n$ even, and take for $S$ a similar set as in case $s$ odd and $n$ odd, namely $S=\{\beta_i;\;1\le i\le (n-4)/2,\,\tilde\beta_{(n-2)/2}=\ep_{n-3}+\ep_{n-1}\}\cup(-\beta_{\pi'}^0)$. Then either $S_{\mid\h_{\Lambda}}$ is not a basis for $\h_{\Lambda}^*$ or in case it is, then take the Heisenberg sets similar as those taken in case $n$ and $s$ odd (with $\Gamma_{\tilde\beta_{(n-2)/2}}=\{\tilde\beta_{(n-2)/2},\,\ep_{n-3}\pm\ep_n,\,\ep_{n-1}\mp\ep_n,\,\ep_{n-3}-\ep_{n-2},\,\ep_{n-2}+\ep_{n-1}\}$). Take also $T$ and $T^*$ disjoint sets such that conditions $(iv)$ and $(vi)$ of Proposition \ref{propAP} hold. But then condition $(v)$ of Proposition \ref{propAP} is not satisfied.

\item Finally consider a parabolic subalgebra $\p$ of $\g$ simple of type ${\rm B}_n$ or ${\rm C}_n$, associated with the subset $\pi'=\pi\setminus\{\alpha_s,\,\alpha_{s+2},\,\ldots,\,\alpha_{s+2\ell},\,\alpha_n\}$ for $s+2\ell\le n-1$. Then a similar construction as this made for $\q_{s,\,\ell}$ for $\g$ simple of type ${\rm D}_n$ with $n$ and $s$ odd does not give a regular element $y$ in $\p^*_{\Lambda}$.

\end{enumerate}
\end{Rqs}


\bigskip

\begin{thebibliography}{40}

\bigskip

\bibitem{BGR} Borho W., Gabriel P., Rentschler R.:
 {\it Primideale in Einh\" ullenden aufl\"osbarer Lie-Algebren} (Beschreibung durch Bahnenr\"aume),
 Lecture Notes in Math. {\bf 357} Springer-Verlag, Berlin (1973).
\bibitem{BOU} Bourbaki N.: 
{\it El\'ements de math\'ematique, Groupes et Alg\`ebres de Lie, Chapitres 4, 5 et 6},
 Masson, Paris et al. (1981).
 \bibitem{Bou1} Bourbaki N.: {\it El\'ements de math\'ematique, Groupes et alg\`ebres de Lie, Chapitres 7 et 8}, Springer-Verlag, Berlin Heidelberg, 2006.
 \bibitem{CM} Charbonnel J.Y., Moreau A.:
 {\it The symmetric invariants of the centralizers and Slodowy grading}, Math. Zeitschrift {\bf 282} (2016), no. 1-2, 273--339.
 \bibitem{D0} Dixmier J.:
 {\it Sur les repr\'esentations unitaires des groupes de Lie nilpotents II}, Bull. Soc. Math. France {\bf 85} (1957), 325--388.
\bibitem{D1} Dixmier J.:
{\it Sur le centre de l'alg\`ebre enveloppante d'une alg\`ebre de Lie}, 
C.R. Acad. Sc. Paris {\bf 265} (1967) 408--410.
\bibitem{D} Dixmier J.:
{\it Alg\`ebres Enveloppantes},
 Editions Jacques Gabay, les grands classiques Gauthier-Villars, Paris et al. (1974).
\bibitem{F} Fauquant-Millet F.:
 {\it Sur la polynomialit\'e de certaines alg\`ebres d'invariants d'alg\`ebres de Lie}, 
 M\'emoire d'Habilitation \`a Diriger des Recherches,\\
  http://tel.archives-ouvertes.fr/tel-00994655.
 \bibitem{F1} Fauquant-Millet F.: {\it About polynomiality of the Poisson Semicentre for Parabolic Subalgebras}. In: Gorelik M., Hinich V., Melnikov A. (eds)  Representations and Nilpotent Orbits of Lie Algebraic Systems, In Honour of the 75th Birthday of Tony
Joseph. Progress in Mathematics, vol. 330 (2019) Birkh\"{a}user, Cham. \\
 https://doi.org/10.1007/978-3-030-23531-4
\bibitem{FJ1} Fauquant-Millet F., Joseph A.: {\it Sur les semi-invariants  d'une sous-alg\`ebre parabolique 
d'une alg\`ebre enveloppante quantifi\'ee},
 Transform. Groups {\bf 6} (2001) 125--142.
\bibitem{FJ2} Fauquant-Millet F., Joseph A.: 
 {\it Semi-centre de l'alg\`ebre enveloppante d'une sous-alg\`ebre parabolique d'une alg\`ebre de Lie semi-simple}, 
 Ann. Sci. \'Ec. Norm. Sup.
{\bf{38}} (2005)  155--191.
\bibitem{FJ3} Fauquant-Millet F.,  Joseph A.: 
 {\it La somme des faux degr\'es - un
myst\`ere en th\'eorie des invariants}, 
 Adv. Math. {\bf 217} (2008) 1476--1520.
\bibitem{FJ4} Fauquant-Millet F.,  Joseph A.:  {\it Adapted  pairs and  Weierstrass  sections},\\
 http://arxiv.org/abs/1503.02523.
\bibitem{FL} Fauquant-Millet F.,  Lamprou P.: 
 {\it Slices for maximal parabolic subalgebras of a semisimple Lie algebra},
 Transform. Groups {\bf 22} (2017) 911--932.
 \bibitem{FL1} Fauquant-Millet F., Lamprou P.: 
 {\it Polynomiality for the Poisson centre of truncated
 maximal parabolic subalgebras}, J. Lie Theory {\bf 28} (2018) 1063--1094.
 \bibitem{FiJ} Fittouhi Y., Joseph A.:
 {\it 
Weierstrass sections for parabolic adjoint action in type A},\\
http://arxiv.org/abs/2001.00447.
\bibitem{J1} Joseph A.:
 {\it A preparation theorem for the prime spectrum of a semisimple Lie
algebra}, 
J. Algebra {\bf{48}} (1977) 241--289.
\bibitem{J6} Joseph A.:
{\it On semi - invariants and index for biparabolic (seaweed) algebras
I}, 
J. Algebra {\bf 305} (2006) 487--515.
\bibitem{J7} Joseph A.:
 {\it On semi - invariants and index for biparabolic (seaweed) algebras
II},  
 J. Algebra {\bf 312} (2007) 158--193.
\bibitem{J5}  Joseph A.:
 {\it Slices for biparabolic coadjoint actions in type
$A$},  
 J. Algebra {\bf 319} (2008) 5060--5100.
\bibitem{J6bis} Joseph A.:
 {\it Compatible adapted pairs and a common slice theorem for some centralizers}, Transform. Groups {\bf 13} (2008) 637--669.
\bibitem{J8} Joseph A.:
 {\it An algebraic slice in the coadjoint space of the Borel and the Coxeter element}, Adv.  Math. {\bf 227} (2011) 522--585.
 \bibitem{JL} Joseph A.,  Lamprou P.:
 {\it Maximal Poisson commutative subalgebras for truncated parabolic subalgebras of maximal
index in $\mathfrak{sl}_n$},
 Transform. Groups {\bf 12} (2007) 549--571.
 \bibitem{JS} Joseph A.,  Shafrir D.: 
 {\it Polynomiality of invariants, unimodularity and adapted pairs},
 Transform. Groups {\bf 15} (2010) 851--882.
\bibitem{K} Kostant B.:
 {\it Lie group representations on polynomial
rings}, 
 Amer. J. Math. {\bf{85}} (1963) 327--404.
 \bibitem{O} Ooms A.: {\it The polynomiality of the Poisson center and semi-center of a Lie algebra and Dixmier's fourth problem}, 
 J. Algebra {\bf 477} (2017) 95-146.
 \bibitem{OV} Ooms A., Van Den Bergh M.:
 {\it A degree inequality for Lie algebras with a regular Poisson semi-center}, J. Algebra {\bf 323} (2010) 305--322.
 \bibitem{P} Panyushev D.: {\it On the coadjoint representation of $\mathbb Z_2$-contractions
 of reductive Lie algebras}, Adv. Math. {\bf 213} (2007), no.1, 343--404.
\bibitem{PPY} Panyushev D.,  Premet A.,  Yakimova O.:
 {\it On symmetric invariants of centralisers in reductive Lie algebras}, 
 J. Algebra {\bf 313} (2007) 343--391.
\bibitem{PPY1} Panyushev D.,  Yakimova O.:
 {\it A remarkable contraction of semisimple Lie algebras}, 
 Ann. Inst. Fourier (Grenoble) {\bf 62} (2012) 2053--2068.
\bibitem{PPY2} Panyushev D.,  Yakimova O.:
 {\it Parabolic contractions of semisimple Lie algebras and their invariants}, 
 Sel. Math. New Ser. {\bf 19} (2013) 699--717.
\bibitem{PPY3} Panyushev D., Yakimova O.:
{\it Symmetric invariants related to representations of exceptional simple groups}, 
Trans. Moscow Math. Soc. {\bf 78} (2017) 161--170.
\bibitem{Po} Popov V. L. {\it Sections in invariant theory.} The Sophus Lie Memorial Conference (Oslo 1992),  Scand. Univ. Press, Oslo (1994) 315-361.
\bibitem{PoVin} Popov V. L., Vinberg E.B. {\it Invariant Theory}, in Encyclopaedia of Mathematical Sciences, Vol. 55, Algebraic Geometry IV, Springer-Verlag Berlin, Heidelberg 1994.
\bibitem{TY} Tauvel P.,  Yu R.W.T.: 
 {\it Lie Algebras and Algebraic Groups}, Springer Monographs in Mathematics, Springer Verlag Berlin Heidelberg (2005).
\bibitem{Y} Yakimova O.:
 {\it A counterexample to Premet's and Joseph's conjectures},
  Bull. Lond. Math. Soc. {\bf 39} (2007) 749--754.
\bibitem{Y1} Yakimova O.:
 {\it Symmetric invariants of $\mathbb Z_2$-contractions and other semi-direct products}, 
 Int. Math. Res. Notices {\bf 2017} (2017)1674--1716.
\bibitem{Y2} Yakimova O.:
{\it Some semi-direct products with free algebras of symmetric invariants}, 
  In: Perspectives in Lie Theory, F. Callegaro, G. Carnovale, F. Caselli, C. De Concini, A. De Sole (eds.),  Springer, Cham {\bf 19} (2017) 267--279.
\end{thebibliography}
\end{document}